\documentclass[12pt,reqno]{amsart}
\ifx\diagram\isundefined\else\expandafter\endinput
\fi

\edef\cdrestoreat{%%
\noexpand\catcode`\noexpand\@=\the\catcode`\@%%
\noexpand\catcode`\noexpand\#=\the\catcode`\#%%
\noexpand\catcode`\noexpand\$=\the\catcode`\$%%
\noexpand\catcode`\noexpand\<=\the\catcode`\<%%
\noexpand\catcode`\noexpand\>=\the\catcode`\>%%
\noexpand\catcode`\noexpand\:=\the\catcode`\:%% Johannes L. Braams's
\noexpand\catcode`\noexpand\;=\the\catcode`\;%% Babel languages package
\noexpand\catcode`\noexpand\!=\the\catcode`\!%% makes these \active.
\noexpand\catcode`\noexpand\?=\the\catcode`\?%%
\noexpand\catcode`\noexpand\+=\the\catcode'53%% texinfo @+ is @outer@active
}\catcode`\@=11 \catcode`\#=6 \catcode`\<=12 \catcode`\>=12 \catcode'53=12
\catcode`\:=12 \catcode`\;=12 \catcode`\!=12 \catcode`\?=12

\ifx\diagram@help@messages\CD@qK\let\diagram@help@messages y\fi

\def\cdps@Rokicki#1{\special{ps:#1}}\let\cdps@dvips\cdps@Rokicki\let
\cdps@RadicalEye\cdps@Rokicki\let\CD@HB\cdps@Rokicki\let\CD@IK\cdps@Rokicki
\let\CD@HB\cdps@Rokicki%%
\def\cdps@Bechtolsheim#1{\special{dvitps: Literal "#1"}}%
\let\cdps@dvitps\cdps@Bechtolsheim\let\cdps@IntegratedComputerSystems
\cdps@Bechtolsheim%%
\def\cdps@Clark#1{\special{dvitops: inline #1}}%%
\let\cdps@dvitops\cdps@Clark%%
\let\cdps@OzTeX\empty\let\cdps@oztex\empty\let\cdps@Trevorrow\empty%%
\def\cdps@Coombes#1{\special{ps-string #1}}%%
\count@=\year\multiply\count@12 \advance\count@\month%%
\ifnum\count@>24382 %% (December 2029)
\ifnum\count@>24385 %% (March 2030)
\errhelp{You may press RETURN and carry on for the time being.}\errmessage{! remain
compatible with your files. Please get it NOW.}\fi\fi

\def\CD@DE{\global\let}\def\CD@RH{\outer\def}

{\escapechar\m@ne\xdef\CD@o{\string\{}\xdef\CD@yC{\string\}}%%
\catcode`\&=4 \CD@DE\CD@Q=&\xdef\CD@S{\string\&}%%ascii three ands
\catcode`\$=3 \CD@DE\CD@k=$\CD@DE\CD@ND=$%%ascii three dollars
\xdef\CD@nC{\string\$}\gdef\CD@LG{$$}%%ascii three dollars
\catcode`\_=8 \CD@DE\CD@lJ=_%%ascii two underlines
\obeylines\catcode`\^=7 \CD@DE\@super=^%%ascii two carets
\ifnum\newlinechar=10 \gdef\CD@uG{^^J}%%ascii two carets
\else\ifnum\newlinechar=13 \gdef\CD@uG{^^M}%%ascii two carets
\else\ifnum\newlinechar=-1 \gdef\CD@uG{^^J}%%ascii two carets
\else\CD@DE\CD@uG\space\expandafter\message{! input error: \noexpand
\newlinechar\space is ASCII \the\newlinechar, not LF=10 or CR=13.}%%
\fi\fi\fi}%%

\mathchardef\lessthan='30474 \mathchardef\greaterthan='30476

\ifx\tenln\CD@qK%%
\font\tenln=line10\relax%% Hit return - who needs diagonals?
\fi\ifx\tenlnw\CD@qK\ifx\tenln\nullfont\let\tenlnw\nullfont\else%%
\font\tenlnw=linew10\relax%% Hit return - who needs diagonals?
\fi\fi%%

\ifx\inputlineno\CD@qK\csname newcount\endcsname\inputlineno\inputlineno\m@ne
\fi

\def\cd@shouldnt#1{\CD@KB{* THIS (#1) SHOULD NEVER HAPPEN! *}}

\def\get@round@pair#1(#2,#3){#1{#2}{#3}}%%ascii round brackets ()
\def\get@square@arg#1[#2]{#1{#2}}%%ascii square brackets []
\def\CD@AE#1{\CD@PK\let\CD@DH\CD@@E\CD@@E#1,],}%%ascii sq brackets
\def\CD@m{[}\def\CD@RD{]}\def\commdiag#1{{\let\enddiagram\relax\diagram[]#1%
\enddiagram}}

\def\CD@BF{{\ifx\CD@EH[\aftergroup\get@square@arg\aftergroup\CD@YH\else
\aftergroup\CD@JH\fi}}%%
\def\CD@CF#1#2{\def\CD@YH{#1}\def\CD@JH{#2}\futurelet\CD@EH\CD@BF}

\def\CD@KK{|}

\def\CD@PB{%% arguments to maps inside diagrams
\tokcase\CD@DD:\CD@y\break@args;\catcase\@super:\upper@label;\catcase\CD@lJ:%
\lower@label;\tokcase{~}:\middle@label;%%ascii tilde
\tokcase<:\CD@iF;%%ascii less-than
\tokcase>:\CD@iI;%%ascii greater-than
\tokcase(:\CD@BC;%%)%ascii open round bracket
\tokcase[:\optional@;%%]%ascii open square bracket
\tokcase.:\CD@JJ;%%ascii dot 12.7.94
\catcase\space:\eat@space;\catcase\bgroup:\positional@;\default:\CD@@A
\break@args;\endswitch}

\def\switch@arg{%% arguments to horizontal maps outside diagrams
\catcase\@super:\upper@label;\catcase\CD@lJ:\lower@label;\tokcase[:\optional@
;%%]%ascii open square bracket
\tokcase.:\CD@JJ;%%ascii dot 12.7.94 % ; was : before 15.6.97
\catcase\space:\eat@space;\catcase\bgroup:\positional@;\tokcase{~}:%
\middle@label;%%ascii tilde (questionable!)
\default:\CD@y\break@args;\endswitch}

\let\CD@tJ\relax\ifx\protect\CD@qK\let\protect\relax\fi\ifx\AtEndDocument
\CD@qK\def\CD@PG{\CD@gB}\def\CD@GF#1#2{}\else\def\CD@PG#1{\edef\CD@CH{#1}%
\expandafter\CD@oC\CD@CH\CD@OD}\def\CD@oC#1\CD@OD{\AtEndDocument{\typeout{%
\CD@tA: #1}}}\def\CD@GF#1#2{\gdef#1{#2}\AtEndDocument{#1}}\fi\def\CD@ZA#1#2{%
\def#1{\CD@PG{#2\CD@mD\CD@W}\CD@DE#1\relax}}\def\CD@uF#1\repeat{\def\CD@p{#1}%
\CD@OF}\def\CD@OF{\CD@p\relax\expandafter\CD@OF\fi}\def\CD@sF#1\repeat{\def
\CD@q{#1}\CD@PF}\def\CD@PF{\CD@q\relax\expandafter\CD@PF\fi}\def\CD@tF#1%
\repeat{\def\CD@r{#1}\CD@QF}\def\CD@QF{\CD@r\relax\expandafter\CD@QF\fi}\def
\CD@tG#1#2#3{\def#2{\let#1\iftrue}\def#3{\let#1\iffalse}#3}\if y%
\diagram@help@messages\def\CD@rG#1#2{\csname newtoks\endcsname#1#1=%
\expandafter{\csname#2\endcsname}}\else\csname newtoks\endcsname\no@cd@help
\no@cd@help{See the manual}\def\CD@rG#1#2{\let#1\no@cd@help}\fi\chardef\CD@lF
=1 \chardef\CD@lI=2 \chardef\CD@MH=5 \chardef\CD@tH=6 \chardef\CD@sH=7
\chardef\CD@PC=9 \dimendef\CD@hI=2 \dimendef\CD@hF=3 \dimendef\CD@mF=4
\dimendef\CD@mI=5 \dimendef\CD@wJ=6 \dimendef\CD@tI=8 \dimendef\CD@sI=9
\skipdef\CD@uB=1 \skipdef\CD@NF=2 \skipdef\CD@tB=3 \skipdef\CD@ZE=4 \skipdef
\countdef\CD@JC=9 \countdef\CD@gD=8 \countdef\CD@A=7 \def\sdef#1#2{\def#1{#2}%
}\def\CD@L#1{\expandafter\aftergroup\csname#1\endcsname}\def\CD@RC#1{%
\expandafter\def\csname#1\endcsname}\def\CD@sD#1{\expandafter\gdef\csname#1%
\endcsname}\def\CD@vC#1{\expandafter\edef\csname#1\endcsname}\def\CD@nF#1#2{%
\expandafter\let\csname#1\expandafter\endcsname\csname#2\endcsname}\def\CD@EE
#1#2{\expandafter\CD@DE\csname#1\expandafter\endcsname\csname#2\endcsname}%
\def\CD@AK#1{\csname#1\endcsname}\def\CD@XJ#1{\expandafter\show\csname#1%
\endcsname}\def\CD@ZJ#1{\expandafter\showthe\csname#1\endcsname}\def\CD@WJ#1{%
\expandafter\showbox\csname#1\endcsname}\def\CD@tA{Commutative Diagram}\edef
\CD@kH{\string\par}\edef\CD@dC{\string\diagram}\edef\CD@HD{\string\enddiagram
}\edef\CD@EC{\string\\}\def\CD@eF{LaTeX}\ifx\@ignoretrue\CD@qK\expandafter
\CD@tG\csname if@ignore\endcsname\ignore@true\ignore@false\def\@ignoretrue{%
\global\ignore@true}\def\@ignorefalse{\global\ignore@false}\fi

\def\CD@g{{\ifnum0=`}\fi}\def\CD@wC{\ifnum0=`{\fi}}\def\catcase#1:{\ifcat
\noexpand\CD@EH#1\CD@tJ\expandafter\CD@kC\else\expandafter\CD@dJ\fi}\def
\tokcase#1:{\ifx\CD@EH#1\CD@tJ\expandafter\CD@kC\else\expandafter\CD@dJ\fi}%
\def\CD@kC#1;#2\endswitch{#1}\def\CD@dJ#1;{}\let\endswitch\relax\def\default:%
#1;#2\endswitch{#1}\ifx\at@\CD@qK\def\at@{@}\fi\edef\CD@P{\CD@o pt\CD@yC}%
\CD@RC{\CD@P>}#1>#2>{\CD@z\rTo\sp{#1}\sb{#2}\CD@z}\CD@RC{\CD@P<}#1<#2<{\CD@z
\lTo\sp{#1}\sb{#2}\CD@z}\CD@RC{\CD@P)}#1)#2){\CD@z\rTo\sp{#1}\sb{#2}\CD@z}%
\CD@RC{\CD@P(}#1(#2({\CD@z\lTo\sp{#1}\sb{#2}\CD@z}%%ascii brack
\def\CD@O{\def\endCD{\enddiagram}\CD@RC{\CD@P A}##1A##2A{\uTo<{##1}>{##2}%
\CD@z\CD@z}\CD@RC{\CD@P V}##1V##2V{\dTo<{##1}>{##2}\CD@z\CD@z}\CD@RC{\CD@P=}{%
\CD@z\hEq\CD@z}\CD@RC{\CD@P\CD@KK}{\vEq\CD@z\CD@z}\CD@RC{\CD@P\string\vert}{%
\vEq\CD@z\CD@z}\CD@RC{\CD@P.}{\CD@z\CD@z}\let\CD@z\CD@Q}\def\CD@IE{\let\tmp
\CD@JE\ifcat A\noexpand\CD@CH\else\ifcat=\noexpand\CD@CH\else\ifcat\relax
\noexpand\CD@CH\else\let\tmp\at@\fi\fi\fi\tmp}\def\CD@JE#1{\CD@nF{tmp}{\CD@P
\string#1}\ifx\tmp\relax\def\tmp{\at@#1}\fi\tmp}\def\CD@z{}\begingroup
\aftergroup\def\aftergroup\CD@T\aftergroup{\aftergroup\def\catcode`\@\active
\aftergroup @\endgroup{\futurelet\CD@CH\CD@IE}}\def\CD@uK#1{\CD@nF{x}{tex_#1:%
D}\ifx\x\relax\else\CD@nF{#1}{x}\fi} \def\CD@vK{\CD@uK{par}\CD@uK{everypar}%
\CD@uK{noindent}\CD@uK{parskip}\CD@uK{unskip}\CD@uK{hskip}\CD@uK{indent}}

\newcount\CD@uA\newcount\CD@vA\newcount\CD@wA\newcount\CD@xA\newdimen\CD@OA
\newdimen\CD@PA\CD@tG\CD@gE\CD@@A\CD@y\CD@tG\CD@hE\CD@EA\CD@BA\newdimen\CD@RA
\newdimen\CD@SA\newcount\CD@yA\newcount\CD@zA\newdimen\CD@QA\newbox\CD@DA
\CD@tG\CD@lE\CD@dA\CD@bA\newcount\CD@LH\newcount\CD@TC\def\CD@V#1#2{\ifdim#1<%
#2\relax#1=#2\relax\fi}\def\CD@X#1#2{\ifdim#1>#2\relax#1=#2\relax\fi}%
\newdimen\CD@XH\CD@XH=1sp \newdimen\CD@zC\CD@zC\z@\def\CD@cJ{\ifdim\CD@zC=1em%
\else\CD@nJ\fi}\def\CD@nJ{\CD@zC1em\def\CD@NC{\fontdimen8\textfont3 }\CD@@J
\CD@NJ\setbox0=\vbox{\CD@t\noindent\CD@k\null\penalty-9993\null\CD@ND\null
\endgraf\setbox0=\lastbox\unskip\unpenalty\setbox1=\lastbox\global\setbox
\CD@IG=\hbox{\unhbox0\unskip\unskip\unpenalty\setbox0=\lastbox}\global\setbox
\CD@KG=\hbox{\unhbox1\unskip\unpenalty\setbox1=\lastbox}}}\newdimen\CD@@I
\CD@@I=1true in \divide\CD@@I300 \def\CD@zH#1{\multiply#1\tw@\advance#1\ifnum
#1<\z@-\else+\fi\CD@@I\divide#1\tw@\divide#1\CD@@I\multiply#1\CD@@I}\def
\MapBreadth{\afterassignment\CD@gI\CD@LF}\newdimen\CD@LF\newdimen\CD@oI\def
\CD@gI{\CD@oI\CD@LF\CD@V\CD@@I{4\CD@XH}\CD@X\CD@@I\p@\CD@zH\CD@oI\ifdim\CD@LF
>\z@\CD@V\CD@oI\CD@@I\fi\CD@cJ}\def\CD@RJ#1{\CD@zD\count@\CD@@I#1\ifnum
\count@>\z@\divide\CD@@I\count@\fi\CD@gI\CD@NJ}\def\CD@NJ{\dimen@\CD@QC
\count@\dimen@\divide\count@5\divide\count@\CD@@I\edef\CD@OC{\the\count@}}%
\def\CD@AJ{\CD@QJ\z@}\def\CD@QJ#1{\CD@tI\axisheight\advance\CD@tI#1\relax
\advance\CD@tI-.5\CD@oI\CD@zH\CD@tI\CD@sI-\CD@tI\advance\CD@tI\CD@LF}%
\newdimen\CD@DC\CD@DC\z@\newdimen\CD@eJ\CD@eJ\z@\def\CD@CJ#1{\CD@sI#1\relax
\CD@tI\CD@sI\advance\CD@tI\CD@LF\relax}\def\horizhtdp{height\CD@tI depth%
\CD@sI}\def\axisheight{\fontdimen22\the\textfont\tw@}\def\script@axisheight{%
\fontdimen22\the\scriptfont\tw@}\def\ss@axisheight{\fontdimen22\the
\scriptscriptfont\tw@}\def\CD@NC{0.4pt}\def\CD@VK{\fontdimen3\textfont\z@}%
\def\CD@UK{\fontdimen3\textfont\z@}\newdimen\PileSpacing\newdimen\CD@nA\CD@nA
\z@\def\CD@RG{\ifincommdiag1.3em\else2em\fi}\newdimen\CD@YB\def\CellSize{%
\afterassignment\CD@kB\DiagramCellHeight}\newdimen\DiagramCellHeight
\DiagramCellHeight-\maxdimen\newdimen\DiagramCellWidth\DiagramCellWidth-%
\maxdimen\def\CD@kB{\DiagramCellWidth\DiagramCellHeight}\def\CD@QC{3em}%
\newdimen\MapShortFall\def\MapsAbut{\MapShortFall\z@\objectheight\z@
\objectwidth\z@}\newdimen\CD@iA\CD@iA\z@\CD@tG\CD@vE\CD@aB\CD@ZB\expandafter
\ifx\expandafter\iftrue\csname ifUglyObsoleteDiagrams\endcsname\CD@ZB\else
\CD@aB\fi\CD@nF{ifUglyObsoleteDiagrams}{relax}\newif\ifUglyObsoleteDiagrams
\def\CD@nK{\CD@aB\UglyObsoleteDiagramsfalse}\def\CD@oK{\CD@ZB
\UglyObsoleteDiagramstrue}\CD@vE\CD@nK\else\CD@oK\fi\CD@tG\CD@hK\CD@dK\CD@cK
\CD@cK\def\CD@sK{\ifx\pdfoutput\CD@qK\else\ifx\pdfoutput\relax\else\ifnum
\pdfoutput>\z@\CD@pK\fi\fi\fi} \def\CD@pK{\global\CD@dK\global\CD@aB\global
\UglyObsoleteDiagramsfalse\global\let\CD@n\empty\global\let\CD@oK\relax
\global\let\CD@pK\relax\global\let\CD@sK\relax}\def\CD@tK#1{}\ifx\pdfliteral\CD@qK\else\ifx\pdfliteral\relax\else\let\CD@tK
\pdfliteral\fi\fi\ifx\XeTeXrevision\CD@qK\else\ifx\XeTeXrevision\relax\else
\ifdim\XeTeXrevision pt<.996pt \expandafter\message{! XeTeX version
\XeTeXrevision\space does not support PDF literals, so diagonals will not work%
!}\else\expandafter\message{RUNNING UNDER XETEX \XeTeXrevision}\CD@pK\fi\fi
\fi\CD@sK\def\newarrowhead{\CD@mG h\CD@BG\CD@GG>}\def\newarrowtail{\CD@mG t%
\CD@BG\CD@GG>}\def\newarrowmiddle{\CD@mG m\CD@BG\hbox@maths\empty}\def
\newarrowfiller{\CD@mG f\CD@bE\CD@MK-}\def\CD@mG#1#2#3#4#5#6#7#8#9{\CD@RC{r#1%
:#5}{#2{#6}}\CD@RC{l#1:#5}{#2{#7}}\CD@RC{d#1:#5}{#3{#8}}\CD@RC{u#1:#5}{#3{#9}%
}\CD@vC{-#1:#5}{\expandafter\noexpand\csname-#1:#4\endcsname\noexpand\CD@MC}%
\CD@vC{+#1:#5}{\expandafter\noexpand\csname+#1:#4\endcsname\noexpand\CD@MC}}%
\CD@ZA\CD@MC{\CD@eF\space diagonals are used unless PostScript is set}\def
\defaultarrowhead#1{\edef\CD@sJ{#1}\CD@@J}\def\CD@@J{\CD@IJ\CD@sJ<>ht\CD@IJ
\CD@sJ<>th}\def\CD@IJ#1#2#3#4#5{\CD@HJ{r#4}{#3}{l#5}{#2}{r#4:#1}\CD@HJ{r#5}{#%
2}{l#4}{#3}{l#4:#1}\CD@HJ{d#4}{#3}{u#5}{#2}{d#4:#1}\CD@HJ{d#5}{#2}{u#4}{#3}{u%
#4:#1}}\def\CD@HJ#1#2#3#4#5{\begingroup\aftergroup\CD@GJ\CD@L{#1+:#2}\CD@L{#1%
:#2}\CD@L{#3:#4}\CD@L{#5}\endgroup}\def\CD@GJ#1#2#3#4{\csname newbox%
\endcsname#1\def#2{\copy#1}\def#3{\copy#1}\setbox#1=\box\voidb@x}\def\CD@sJ{}%
\CD@@J\def\CD@GJ#1#2#3#4{\setbox#1=#4}\ifx\tenln\nullfont\def\CD@sJ{vee}\else
\let\CD@sJ\CD@eF\fi\def\CD@xF#1#2#3{\begingroup\aftergroup\CD@wF\CD@L{#1#2:#3%
#3}\CD@L{#1#2:#3}\aftergroup\CD@yF\CD@L{#1#2:#3-#3}\CD@L{#1#2:#3}\endgroup}%
\def\CD@wF#1#2{\def#1{\hbox{\rlap{#2}\kern.4\CD@zC#2}}}\def\CD@yF#1#2{\def#1{%
\hbox{\rlap{#2}\kern.4\CD@zC#2\kern-.4\CD@zC}}}\CD@xF lh>\CD@xF rt>\CD@xF rh<%
\CD@xF rt<\def\CD@yF#1#2{\def#1{\hbox{\kern-.4\CD@zC\rlap{#2}\kern.4\CD@zC#2}%
}}\CD@xF rh>\CD@xF lh<\CD@xF lt>\CD@xF lt<\def\CD@wF#1#2{\def#1{\vbox{\vbox to%
\z@{#2\vss}\nointerlineskip\kern.4\CD@zC#2}}}\def\CD@yF#1#2{\def#1{\vbox{%
\vbox to\z@{#2\vss}\nointerlineskip\kern.4\CD@zC#2\kern-.4\CD@zC}}}\CD@xF uh>%
\CD@xF dt>\CD@xF dh<\CD@xF dt<\def\CD@yF#1#2{\def#1{\vbox{\kern-.4\CD@zC\vbox
to\z@{#2\vss}\nointerlineskip\kern.4\CD@zC#2}}}\CD@xF dh>\CD@xF ut>\CD@xF uh<%
\CD@xF ut<\def\CD@BG#1{\hbox{\mathsurround\z@\offinterlineskip\CD@k\mkern-1.5%
mu{#1}\mkern-1.5mu\CD@ND}}\def\hbox@maths#1{\hbox{\CD@k#1\CD@ND}}\def\CD@GG#1%
{\hbox to\CD@LF{\setbox0=\hbox{\offinterlineskip\mathsurround\z@\CD@k{#1}%
\CD@ND}\dimen0.5\wd0\advance\dimen0-.5\CD@oI\CD@zH{\dimen0}\kern-\dimen0%
\unhbox0\hss}}\def\CD@sB#1{\hbox to2\CD@LF{\hss\offinterlineskip\mathsurround
\z@\CD@k{#1}\CD@ND\hss}}\def\CD@vF#1{\hbox{\mathsurround\z@\CD@k{#1}\CD@ND}}%
\def\CD@bE#1{\hbox{\kern-.15\CD@zC\CD@k{#1}\CD@ND\kern-.15\CD@zC}}\def\CD@MK#%
1{\vbox{\offinterlineskip\kern-.2ex\CD@GG{#1}\kern-.2ex}}\def\@fillh{%
\xleaders\vrule\horizhtdp}\def\@fillv{\xleaders\hrule width\CD@LF}\CD@nF{rf:-%
}{@fillh}\CD@nF{lf:-}{@fillh}\CD@nF{df:-}{@fillv}\CD@nF{uf:-}{@fillv}\CD@nF{%
rh:}{null}\CD@nF{rm:}{null}\CD@nF{rt:}{null}\CD@nF{lh:}{null}\CD@nF{lm:}{null%
}\CD@nF{lt:}{null}\CD@nF{dh:}{null}\CD@nF{dm:}{null}\CD@nF{dt:}{null}\CD@nF{%
uh:}{null}\CD@nF{um:}{null}\CD@nF{ut:}{null}\CD@nF{+h:}{null}\CD@nF{+m:}{null%
}\CD@nF{+t:}{null}\CD@nF{-h:}{null}\CD@nF{-m:}{null}\CD@nF{-t:}{null}\def
\CD@@D{\hbox{\vrule height 1pt depth-1pt width 1pt}}\CD@RC{rf:}{\CD@@D}\CD@nF
{lf:}{rf:}\CD@nF{+f:}{rf:}\CD@RC{df:}{\CD@@D}\CD@nF{uf:}{df:}\CD@nF{-f:}{df:}%
\def\CD@BD{\CD@U\null\CD@@D\null\CD@@D\null}\edef\CD@lG{\string\newarrow}\def
\newarrow#1#2#3#4#5#6{\begingroup\edef\@name{#1}\edef\CD@oJ{#2}\edef\CD@iD{#3%
}\edef\CD@QG{#4}\edef\CD@jD{#5}\edef\CD@LE{#6}\let\CD@HE\CD@sG\let\CD@FK
\CD@BH\let\@x\CD@AH\ifx\CD@oJ\CD@iD\let\CD@oJ\empty\fi\ifx\CD@LE\CD@jD\let
\CD@LE\empty\fi\def\CD@LI{r}\def\CD@SF{l}\def\CD@IC{d}\def\CD@yJ{u}\def\CD@gH
{+}\def\@m{-}\ifx\CD@iD\CD@jD\ifx\CD@QG\CD@iD\let\CD@QG\empty\fi\ifx\CD@LE
\empty\ifx\CD@iD\CD@aE\let\@x\CD@yG\else\let\@x\CD@zG\fi\fi\else\edef\CD@a{%
\CD@iD\CD@oJ}\ifx\CD@a\empty\ifx\CD@QG\CD@jD\let\CD@QG\empty\fi\fi\fi\ifmmode
\aftergroup\CD@kG\else\CD@@A\CD@oB rh{head\space\space}\CD@LE\CD@oB rf{filler%
}\CD@iD\CD@oB rm{middle}\CD@QG\ifx\CD@jD\CD@iD\else\CD@oB rf{filler}\CD@jD\fi
\CD@oB rt{tail\space\space}\CD@oJ\CD@gE\CD@HE\CD@FK\@x\CD@nG l-2+2{lu}{nw}%
\NorthWest\CD@nG r+2+2{ru}{ne}\NorthEast\CD@nG l-2-2{ld}{sw}\SouthWest\CD@nG r%
+2-2{rd}{se}\SouthEast\else\aftergroup\CD@b\CD@L{r\@name}\fi\fi\endgroup}\def
\CD@sG{\CD@vG\CD@LI\CD@SF rl\Horizontal@Map}\def\CD@BH{\CD@vG\CD@IC\CD@yJ du%
\Vertical@Map}\def\CD@AH{\CD@vG\CD@gH\@m+-\Vector@Map}\def\CD@yG{\CD@vG\CD@gH
\@m+-\Slant@Map}\def\CD@zG{\CD@vG\CD@gH\@m+-\Slope@Map}\catcode`\/=\active
\def\CD@vG#1#2#3#4#5{\CD@jG#1#3#5t:\CD@oJ/f:\CD@iD/m:\CD@QG/f:\CD@jD/h:\CD@LE
//\CD@jG#2#4#5h:\CD@LE/f:\CD@jD/m:\CD@QG/f:\CD@iD/t:\CD@oJ//}\def\CD@jG#1#2#3%
#4//{\edef\CD@fG{#2}\aftergroup\sdef\CD@L{#1\@name}\aftergroup{\aftergroup#3%
\CD@M#4//\aftergroup}}\def\CD@M#1/{\edef\CD@EH{#1}\ifx\CD@EH\empty\else\CD@L{%
\CD@fG#1}\expandafter\CD@M\fi}\catcode`\/=12 \def\CD@nG#1#2#3#4#5#6#7#8{%
\aftergroup\sdef\CD@L{#6\@name}\aftergroup{\CD@L{#2\@name}\if#2#4\aftergroup
\CD@CI\else\aftergroup\CD@BI\fi\CD@L{#1\@name}%
\aftergroup(\aftergroup#3\aftergroup,\aftergroup#5\aftergroup)\aftergroup}}%
\def\CD@oB#1#2#3#4{\expandafter\ifx\csname#1#2:#4\endcsname\relax\CD@y\CD@gB{%
arrow#3 "#4" undefined}\fi}\CD@rG\CD@VE{All five components must be defined
before an arrow.}\CD@rG\CD@SE{\CD@lG, unlike \string\HorizontalMap, is a
declaration.}\def\CD@b#1{\CD@YA{Arrows \string#1 etc could not be defined}%
\CD@VE}\def\CD@kG{\CD@YA{misplaced \CD@lG}\CD@SE}\def\newdiagramgrid#1#2#3{%
\CD@RC{cdgh@#1}{#2,],}%% ASCII close square bracket
\CD@RC{cdgv@#1}{#3,],}}%% ASCII close square bracket
\CD@tG\ifincommdiag\incommdiagtrue\incommdiagfalse\CD@tG\CD@@F\CD@IF\CD@HF
\newcount\CD@VA\CD@VA=0 \def\CD@yH{\CD@VA6 }\def\CD@OB{\CD@VA1 \global\CD@yA1
\CD@DE\CD@YF\empty}\def\CD@YF{}\def\CD@nB#1{\relax\CD@MD\edef\CD@vJ{#1}%
\begingroup\CD@rE\else\ifcase\CD@VA\ifmmode\else\CD@YG\CD@E0\fi\or\CD@cE5\or
\CD@YG\CD@F5\or\CD@YG\CD@B5\or\CD@YG\CD@B5\or\CD@YG\CD@C5\or\CD@cE7\or\CD@YG
\CD@D7\fi\fi\endgroup\xdef\CD@YF{#1}}\def\CD@pB#1#2#3#4#5{\relax\CD@MD\xdef
\CD@vJ{#4}\begingroup\ifnum\CD@VA<#1 \expandafter\CD@cE\ifcase\CD@VA0\or#2\or
#3\else#2\fi\else\ifnum\CD@VA<6 \CD@tJ\CD@YG\CD@B#2\else\CD@YG\CD@G#2\fi\fi
\endgroup\CD@DE\CD@YF\CD@vJ\ifincommdiag\let\CD@ZD#5\else\let\CD@ZD\CD@LK\fi}%
\def\CD@yI{\global\CD@yA=\ifnum\CD@VA<5 1\else2\fi\relax}\def\CD@OI{\CD@VA
\CD@yA}\def\CD@cE#1{\aftergroup\CD@VA\aftergroup#1\aftergroup\relax}\def
\CD@HH{\def\CD@nB##1{\relax}\let\CD@pB\CD@FH\let\CD@yH\relax\let\CD@OB\relax
\let\CD@yI\relax\let\CD@OI\relax}\def\CD@FH#1#2#3#4#5{\ifincommdiag\let\CD@ZD
#5\else\xdef\CD@vJ{#4}\let\CD@ZD\CD@LK\fi}\def\CD@YG#1{\aftergroup#1%
\aftergroup\relax\CD@cE}\def\CD@B{\CD@YE\CD@S\CD@ME\CD@Q}\def\CD@G{\CD@YE{%
\CD@yC\CD@S}\CD@XE\CD@QD\CD@Q}\def\CD@F{\CD@YE{*\CD@S}\CD@RE\clubsuit\CD@Q}%
\def\CD@C{\CD@YE{\CD@S*\CD@S}\CD@RE\CD@Q\clubsuit\CD@Q}\def\CD@D{\CD@YE\CD@EC
\CD@TE\\}\def\CD@E{\CD@YE\CD@nC\CD@QE\CD@k}\def\CD@LK{\CD@YA{\CD@vJ\space
ignored \CD@dH}\CD@WE}\def\CD@FE{}\def\CD@d{\CD@YA{maps must never be enclosed
in braces}\CD@OE}\def\CD@dH{outside diagram}\def\CD@FC{\string\HonV, \string
\VonH\space and \string\HmeetV}\CD@rG\CD@ME{The way that horizontal and
vertical arrows are terminated implicitly means\CD@uG that they cannot be
mixed with each other or with \CD@FC.}\CD@rG\CD@XE{\string\pile\space is for
parallel horizontal arrows; verticals can just be put together in\CD@uG a cell%
. \CD@FC\space are not meaningful in a \string\pile.}\CD@rG\CD@RE{The
horizontal maps must point to an object, not each other (I've put in\CD@uG one
which you're unlikely to want). Use \string\pile\space if you want them
parallel.}\CD@rG\CD@TE{Parallel horizontal arrows must be in separate layers
of a \string\pile.}\CD@rG\CD@QE{Horizontal arrows may be used \CD@dH s, but
must still be in maths.}\CD@rG\CD@WE{Vertical arrows, \CD@FC\space\CD@dH s don%
't know where\CD@uG where to terminate.}\CD@rG\CD@OE{This prevents them from
stretching correctly.}\def\CD@YE#1{\CD@YA{"#1" inserted \ifx\CD@YF\empty
before \CD@vJ\else between \CD@YF\ifx\CD@YF\CD@vJ s\else\space and \CD@vJ\fi
\fi}}\count@=\year\multiply\count@12 \advance\count@\month\ifnum\count@>24391
\expandafter\endinput\fi\def
\Horizontal@Map{\CD@nB{horizontal map}\CD@LC\CD@TJ\CD@qD}\def\CD@TJ{\CD@GB-%
9999 \let\CD@ZD\CD@XD\ifincommdiag\else\CD@cJ\ifinpile\else\skip2\z@ plus 1.5%
\CD@VK minus .5\CD@UK\skip4\skip2 \fi\fi\let\CD@kD\@fillh\CD@nF{fill@dot}{rf:%
.}}\def\Vector@Map{\CD@HK4}\def\Slant@Map{\CD@HK{\CD@EF255\else6\fi}}\def
\Slope@Map{\CD@HK\CD@OC}\def\CD@HK#1#2#3#4#5#6{\CD@LC\def\CD@WK{2}\def\CD@aK{%
2}\def\CD@ZK{1}\def\CD@bK{1}\let\Horizontal@Map\CD@nI\def\CD@OG{#1}\def\CD@NI
{\CD@U#2#3#4#5#6}}\def\CD@nI{\CD@TJ\CD@JB\let\CD@ZD\CD@TD\CD@qD}\CD@tG\CD@pE
\CD@rA\CD@qA\CD@rA\def\cds@missives{\CD@rA}\def\CD@TD{\CD@vE\let\CD@OG\CD@OC
\CD@x\CD@zE\CD@WF\fi\setbox0\hbox{\incommdiagfalse\CD@HI}\CD@pE\CD@aD\else
\global\CD@YC\CD@bD\fi\ifvoid6 \ifvoid7 \CD@eE\fi\fi\CD@zE\else\CD@BD\global
\CD@YC\let\CD@CG\CD@IH\CD@YD\fi\else\CD@NI\CD@MI\global\CD@YC\CD@YD\fi}\def
\CD@LC{\begingroup\dimen1=\MapShortFall\dimen2=\CD@RG\dimen5=\MapShortFall
\setbox3=\box\voidb@x\setbox6=\box\voidb@x\setbox7=\box\voidb@x\CD@pD
\mathsurround\z@\skip2\z@ plus1fill minus 1000pt\skip4\skip2 \CD@TB}\CD@tG
\CD@tE\CD@UB\CD@TB\def\CD@U#1#2#3#4#5{\let\CD@oJ#1\let\CD@iD#2\let\CD@QG#3%
\let\CD@jD#4\let\CD@LE#5\CD@TB\ifx\CD@iD\CD@jD\CD@UB\fi}\def\CD@qD#1#2#3#4#5{%
\CD@U#1#2#3#4#5\CD@tD}\def\Vertical@Map{\CD@pB433{vertical map}\CD@cD\CD@LC
\CD@GB-9995 \let\CD@kD\@fillv\CD@nF{fill@dot}{df:.}\CD@qD}\def\break@args{%
\def\CD@tD{\CD@ZD}\CD@ZD\endgroup\aftergroup\CD@FE}\def\CD@MJ{\setbox1=\CD@oJ
\setbox5=\CD@LE\ifvoid3 \ifx\CD@QG\null\else\setbox3=\CD@QG\fi\fi\CD@@G2%
\CD@iD\CD@@G4\CD@jD}\def\CD@pF#1{\ifvoid1\else\CD@oF1#1\fi\ifvoid2\else\CD@oF
2#1\fi\ifvoid3\else\CD@oF3#1\fi\ifvoid4\else\CD@oF4#1\fi\ifvoid5\else\CD@oF5#%
1\fi} \def\CD@oF#1#2{\setbox#1\vbox{\offinterlineskip\box#1\dimen@\prevdepth
\advance\dimen@-#2\relax\setbox0\null\dp0\dimen@\ht0-\dimen@\box0}}\def\CD@@G
#1#2{\ifx#2\CD@kD\setbox#1=\box\voidb@x\else\setbox#1=#2\def#2{\xleaders\box#%
1}\fi}\CD@ZA\CD@BK{\string\HorizontalMap, \string\VerticalMap\space and
\string\DiagonalMap\CD@uG are obsolete - use \CD@lG\space to pre-define maps}%
\def\HorizontalMap#1#2#3#4#5{\CD@BK\CD@nB{old horizontal map}\CD@LC\CD@TJ\def
\CD@oJ{\CD@UH{#1}}\CD@SH\CD@iD{#2}\def\CD@QG{\CD@UH{#3}}\CD@SH\CD@jD{#4}\def
\CD@LE{\CD@UH{#5}}\CD@tD}\def\VerticalMap#1#2#3#4#5{\CD@BK\CD@pB433{vertical
map}\CD@cD\CD@LC\CD@GB-9995 \let\CD@kD\@fillv\def\CD@oJ{\CD@GG{#1}}\CD@VH
\CD@iD{#2}\def\CD@QG{\CD@GG{#3}}\CD@VH\CD@jD{#4}\def\CD@LE{\CD@GG{#5}}\CD@tD}%
\def\DiagonalMap#1#2#3#4#5{\CD@BK\CD@LC\def\CD@OG{4}\let\CD@kD\CD@qK\let
\CD@ZD\CD@YD\def\CD@WK{2}\def\CD@aK{2}\def\CD@ZK{1}\def\CD@bK{1}\def\CD@QG{%
\CD@vF{#3}}\ifPositiveGradient\let\mv\raise\def\CD@oJ{\CD@vF{#5}}\def\CD@iD{%
\CD@vF{#4}}\def\CD@jD{\CD@vF{#2}}\def\CD@LE{\CD@vF{#1}}\else\let\mv\lower\def
\CD@oJ{\CD@vF{#1}}\def\CD@iD{\CD@vF{#2}}\def\CD@jD{\CD@vF{#4}}\def\CD@LE{%
\CD@vF{#5}}\fi\CD@tD}\def\CD@aE{-}\def\CD@AD{\empty}\def\CD@SH{\CD@EG\CD@bE
\CD@aE\@fillh}\def\CD@VH{\CD@EG\CD@MK\CD@KK\@fillv}\def\CD@EG#1#2#3#4#5{\def
\CD@CH{#5}\ifx\CD@CH#2\let#4#3\else\let#4\null\ifx\CD@CH\empty\else\ifx\CD@CH
\CD@AD\else\let#4\CD@CH\fi\fi\fi}\def\CD@UH#1{\hbox{\mathsurround\z@
\offinterlineskip\def\CD@CH{#1}\ifx\CD@CH\empty\else\ifx\CD@CH\CD@AD\else
\CD@k\mkern-1.5mu{\CD@CH}\mkern-1.5mu\CD@ND\fi\fi}}\def\CD@yD#1#2{\setbox#1=%
\hbox\bgroup\setbox0=\hbox{\CD@k\labelstyle()\CD@ND}%% ASCII round brackets
\setbox1=\null\ht1\ht0\dp1\dp0\box1 \kern.1\CD@zC\CD@k\bgroup\labelstyle
\aftergroup\CD@LD\CD@xD}\def\CD@LD{\CD@ND\kern.1\CD@zC\egroup\CD@tD}\def
\CD@xD{\futurelet\CD@EH\CD@mJ}\def\CD@mJ{%% qualifiers on label arguments
\catcase\bgroup:\CD@v;\catcase\egroup:\missing@label;\catcase\space:\CD@TF;%
\tokcase[:\CD@XF;%%]%ascii close square bracket 
\default:\CD@zJ;\endswitch}\def\CD@v{\let\CD@MD\CD@c\let\CD@CH}\def\CD@zJ#1{%
\let\CD@UF\egroup{\let\actually@braces@missing@around@macro@in@label\CD@ZH
\let\CD@MD\CD@xC\let\CD@UF\CD@VF#1%
\actually@braces@missing@around@macro@in@label}\CD@UF}\def
\actually@braces@missing@around@macro@in@label{\let\CD@CH=}\def\missing@label
{\egroup\CD@YA{missing label}\CD@PE}\def\CD@xC{\egroup\missing@label}\outer
\def\CD@ZH{}\def\CD@UF{}\def\CD@VF{\CD@wC\CD@UF}\def\CD@MD{}\def\CD@XF{\let
\CD@N\CD@xD\get@square@arg\CD@AE}\CD@rG\CD@PE{The text which has just been
read is not allowed within map labels.}\def\CD@c{\egroup\CD@YA{missing \CD@yC
\space inserted after label}\CD@PE}\def\upper@label{\CD@oD\CD@yD6}\def
\lower@label{\def\positional@{\CD@@A\break@args}\CD@yD7}\def\middle@label{%
\CD@yD3}\CD@tG\CD@yE\CD@pD\CD@oD\def\CD@iF{\ifPositiveGradient\CD@tJ
\expandafter\upper@label\else\expandafter\lower@label\fi}\def\CD@iI{%
\ifPositiveGradient\CD@tJ\expandafter\lower@label\else\expandafter
\upper@label\fi}\def\positional@{\CD@gB{labels as positional arguments are
obsolete}\CD@yE\CD@tJ\expandafter\upper@label\else\expandafter\lower@label\fi
-}\def\CD@tD{\futurelet\CD@EH\switch@arg}\def\eat@space{\afterassignment
\CD@tD\let\CD@EH= }\def\CD@TF{\afterassignment\CD@xD\let\CD@EH= }\def\CD@BC{%
\get@round@pair\CD@uD}\def\CD@uD#1#2{\def\CD@WK{#1}\def\CD@aK{#2}\CD@tD}\def
\optional@{\let\CD@N\CD@tD\get@square@arg\CD@AE}\def\CD@JJ.{\CD@sC\CD@tD}\def
\CD@sC{\let\CD@iD\fill@dot\let\CD@jD\fill@dot\def\CD@MI{\let\CD@iD\dfdot\let
\CD@jD\dfdot}}\def\CD@MI{}\def\CD@@E#1,{\CD@nH#1,\begingroup\ifx\@name\CD@RD
\CD@FF\aftergroup\CD@e\fi\aftergroup\CD@jC\else\expandafter\def\expandafter
\CD@RF\expandafter{\csname\@name\endcsname}\expandafter\CD@vD\CD@RF\CD@KD\ifx
\CD@RF\empty\aftergroup\CD@pC\expandafter\aftergroup\csname\CD@FB\@name
\endcsname\expandafter\aftergroup\csname\CD@FB @\@name\endcsname\else\gdef
\CD@GE{#1}\CD@gB{\string\relax\space inserted before `[\CD@GE'}\message{(I was
trying to read this as a \CD@tA\ option.)}\aftergroup\CD@H\fi\fi\endgroup}%
\def\CD@vD#1#2\CD@KD{\def\CD@RF{#2}}\def\CD@jC{\let\CD@CH\CD@N\let\CD@N\relax
\CD@CH}\def\CD@H#1],{%% ASCII close square bracket
\CD@jC\relax\def\CD@RF{#1}\ifx\CD@RF\empty\def\CD@RF{[\CD@GE]}%
\else\def\CD@RF{[\CD@GE,#1]}%% ASCII open and close square bracket
\fi\CD@RF}\def\CD@pC#1#2{\ifx#2\CD@qK\ifx#1\CD@qK\CD@gB{option `\@name'
undefined}\else#1\fi\else\CD@FF\expandafter#2\CD@GK\CD@PK\else\CD@QK\fi\fi
\CD@DH}\CD@tG\CD@FF\CD@QK\CD@PK\def\CD@nH#1,{\CD@FF\ifx\CD@GK\CD@qK\CD@e\else
\expandafter\CD@oH\CD@GK,#1,(,),(,)[]%
\fi\fi\CD@FF\else\CD@mH#1==,\fi}\def\CD@e{\CD@gB{option `\@name' needs (x,y)
value}\CD@PK\let\@name\empty}\def\CD@mH#1=#2=#3,{\def\@name{#1}\def\CD@GK{#2}%
\def\CD@RF{#3}\ifx\CD@RF\empty\let\CD@GK\CD@qK\fi}%
\def\CD@oH#1(#2,#3)#4,(#5,#6)#7[]{\def\CD@GK{{#2}{#3}}\def\CD@RF{#1#4#5#6}%
\ifx\CD@RF\empty\def\CD@RF{#7}\ifx\CD@RF\empty\CD@e\fi\else\CD@e\fi}\def
\CD@FB{cds@}\let\CD@N\relax\def\CD@zD#1{\ifx\CD@GK\CD@qK\CD@gB{option `\@name
' needs a value}\else#1\CD@GK\relax\fi}\def\CD@BE#1#2{\ifx\CD@GK\CD@qK#1#2%
\relax\else#1\CD@GK\relax\fi}\def\cds@@showpair#1#2{\message{x=#1,y=#2}}\def
\cds@@diagonalbase#1#2{\edef\CD@ZK{#1}\edef\CD@bK{#2}}\def\CD@DI#1{\def\CD@CH
{#1}\CD@nF{@x}{cdps@#1}\ifx\CD@CH\empty\CD@f\CD@CH{cannot be used}\else\ifx
\CD@CH\relax\CD@f\CD@CH{unknown}\else\let\CD@IK\@x\fi\fi}\def\CD@f#1#2{\CD@gB
{PostScript translator `#1' #2}}\def\CD@PH{}\def\CD@PJ{\CD@fA\edef\CD@PH{%
\noexpand\CD@KB{\@name\space ignored within maths}}}\def\diagramstyle{\CD@cJ
\let\CD@N\relax\CD@CF\CD@AE\CD@AE}\CD@tG\CD@sE
\CD@SB\CD@RB\CD@tG\CD@qE\CD@EB\CD@DB\CD@tG\CD@oE\CD@pA\CD@oA\CD@tG\CD@iE
\CD@HA\CD@GA\CD@HA\CD@tG\CD@jE\CD@JA\CD@IA\CD@tG\CD@kE\CD@LA\CD@KA\CD@tG
\CD@EF\CD@DK\CD@CK\CD@tG\CD@rE\CD@JB\CD@IB\CD@tG\CD@mE\CD@gA\CD@fA\CD@tG
\CD@nE\CD@kA\CD@jA\CD@tG\CD@AF\CD@iG\CD@hG\CD@RC{cds@ }{}\CD@RC{cds@}{}\CD@RC
{cds@1em}{\CellSize1\CD@zC}\CD@RC{cds@1.5em}{\CellSize1.5\CD@zC}\CD@RC{cds@2%
em}{\CellSize2\CD@zC}\CD@RC{cds@2.5em}{\CellSize2.5\CD@zC}\CD@RC{cds@3em}{%
\CellSize3\CD@zC}\CD@RC{cds@3.5em}{\CellSize3.5\CD@zC}\CD@RC{cds@4em}{%
\CellSize4\CD@zC}\CD@RC{cds@4.5em}{\CellSize4.5\CD@zC}\CD@RC{cds@5em}{%
\CellSize5\CD@zC}\CD@RC{cds@6em}{\CellSize6\CD@zC}\CD@RC{cds@7em}{\CellSize7%
\CD@zC}\CD@RC{cds@8em}{\CellSize8\CD@zC}\def\cds@abut{\MapsAbut\dimen1\z@
\dimen5\z@}\def\cds@alignlabels{\CD@IA\CD@KA}\def\cds@amstex{\ifincommdiag
\CD@O\else\def\CD{\diagram[amstex]}%%ascii square brackets []
\fi\CD@T\catcode`\@\active}\def\cds@b{\let\CD@dB\CD@bB}\def\cds@balance{\let
\CD@hA\CD@AA}\let\cds@bottom\cds@b\def\cds@center{\cds@vcentre\cds@nobalance}%
\let\cds@centre\cds@center\def\cds@centerdisplay{\CD@HA\CD@PJ\cds@balance}%
\let\cds@centredisplay\cds@centerdisplay\def\cds@crab{\CD@BE\CD@DC{.5%
\PileSpacing}}\CD@RC{cds@crab-}{\CD@DC-.5\PileSpacing}\CD@RC{cds@crab+}{%
\CD@DC.5\PileSpacing}\CD@RC{cds@crab++}{\CD@DC1.5\PileSpacing}\CD@RC{cds@crab%
--}{\CD@DC-1.5\PileSpacing}\def\cds@defaultsize{\CD@BE{\let\CD@QC}{3em}\CD@NJ
}\def\cds@displayoneliner{\CD@DB}\let\cds@dotted\CD@sC\def\cds@dpi{\CD@RJ{1%
truein}}\def\cds@dpm{\CD@RJ{100truecm}}\let\CD@XA\CD@qK\def\cds@eqno{\let
\CD@XA\CD@GK\let\CD@EJ\empty}\def\cds@fixed{\CD@qA}\CD@tG\CD@fE\CD@J\CD@I\def
\cds@flushleft{\CD@I\CD@GA\CD@PJ\cds@nobalance\CD@BE\CD@nA\CD@nA}\def\cds@gap
{\CD@AJ\setbox3=\null\ht3=\CD@tI\dp3=\CD@sI\CD@BE{\wd3=}\MapShortFall} \def
\cds@grid{\ifx\CD@GK\CD@qK\let\h@grid\relax\let\v@grid\relax\else\CD@nF{%
h@grid}{cdgh@\CD@GK}\CD@nF{v@grid}{cdgv@\CD@GK}\ifx\h@grid\relax\CD@gB{%
unknown grid `\CD@GK'}\else\CD@WB\fi\fi}\let\h@grid\relax\let\v@grid\relax
\def\cds@gridx{\ifx\CD@GK\CD@qK\else\cds@grid\fi\let\CD@CH\h@grid\let\h@grid
\v@grid\let\v@grid\CD@CH}\def\cds@h{\CD@zD\DiagramCellHeight}\def\cds@hcenter
{\let\CD@hA\CD@aA}\let\cds@hcentre\cds@hcenter\def\cds@heads{\CD@BE{\let
\CD@sJ}\CD@sJ\CD@@J\CD@vE\else\ifx\CD@sJ\CD@eF\else\CD@MC\fi\fi}\let
\cds@height\cds@h\let\cds@hmiddle\cds@balance\def\cds@htriangleheight{\CD@BE
\DiagramCellHeight\DiagramCellHeight\DiagramCellWidth1.73205%
\DiagramCellHeight}\def\cds@htrianglewidth{\CD@BE\DiagramCellWidth
\DiagramCellWidth\DiagramCellHeight.57735\DiagramCellWidth}\CD@tG\CD@zE\CD@eE
\CD@dE\CD@eE\def\cds@hug{\CD@eE} \def\cds@inline{\CD@gA\let\CD@PH\empty}\def
\cds@inlineoneliner{\CD@EB}\CD@RC{cds@l>}{\CD@zD{\let\CD@RG}\dimen2=\CD@RG}%
\def\cds@labelstyle{\CD@zD{\let\labelstyle}}\def\cds@landscape{\CD@kA}\def
\cds@large{\CellSize5\CD@zC}\let\CD@EJ\empty\def\CD@FJ{\refstepcounter{%
equation}\def\CD@XA{\hbox{\@eqnnum}}}\def\cds@LaTeXeqno{\let\CD@EJ\CD@FJ}\def
\cds@lefteqno{\CD@pA}\def\cds@leftflush{\cds@flushleft\CD@J}\def
\cds@leftshortfall{\CD@zD{\dimen1 }}\def\cds@lowershortfall{%
\ifPositiveGradient\cds@leftshortfall\else\cds@rightshortfall\fi}\def
\cds@loose{\CD@VB}\def\cds@midhshaft{\CD@JA}\def\cds@midshaft{\CD@JA}\def
\cds@midvshaft{\CD@LA}\def\cds@moreoptions{\CD@@A}\let\cds@nobalance
\cds@hcenter\def\cds@nohcheck{\CD@HH}\def\cds@nohug{\CD@dE} \def
\cds@nooptions{\def\CD@aC{\CD@WD}}\let\cds@noorigin\cds@nobalance\def
\cds@nopixel{\CD@@I4\CD@XH\CD@cJ}\def\cds@UO{\CD@oK\global\let\CD@n\empty}%
\def\cds@UglyObsolete{\cds@UO\let\cds@PS\empty}\def\CD@rK#1{\CD@gB{option `#1%
' renamed as `UglyObsolete'}}\def\cds@noPostScript{\CD@rK{noPostScript}}\def
\cds@noPS{\CD@rK{noPostScript}}\def\cds@notextflow{\CD@RB}\def\cds@noTPIC{%
\CD@CK}\def\cds@objectstyle{\CD@zD{\let\objectstyle}}\def\cds@origin{\let
\CD@hA\CD@iB}\def\cds@p{\CD@zD\PileSpacing}\let\cds@pilespacing\cds@p\def
\cds@pixelsize{\CD@zD\CD@@I\CD@gI}\def\cds@portrait{\CD@jA}\def
\cds@PostScript{\CD@nK\global\let\CD@n\empty\CD@BE\CD@DI\empty}\def\cds@PS{%
\CD@nK\global\let\CD@n\empty}\CD@GF\CD@n{\typeout{\CD@tA: try the PostScript
option for better results}}\def\cds@repositionpullbacks{\let\make@pbk\CD@fH
\let\CD@qH\CD@pH}\def\cds@righteqno{\CD@oA}\def\cds@rightshortfall{\CD@zD{%
\dimen5 }}\def\cds@ruleaxis{\CD@zD{\let\axisheight}}\def\cds@cmex{\let\CD@GG
\CD@sB\let\CD@QJ\CD@CJ}\def\cds@s{\cds@height\DiagramCellWidth
\DiagramCellHeight}\def\cds@scriptlabels{\let\labelstyle\scriptstyle}\def
\cds@shortfall{\CD@zD\MapShortFall\dimen1\MapShortFall\dimen5\MapShortFall}%
\def\cds@showfirstpass{\CD@BE{\let\CD@nD}\z@}\def\cds@silent{\def\CD@KB##1{}%
\def\CD@gB##1{}}\let\cds@size\cds@s\def\cds@small{\CellSize2\CD@zC}\def
\cds@snake{\CD@BE\CD@eJ\z@}\def\cds@t{\let\CD@dB\CD@fB}\def\cds@textflow{%
\CD@SB\CD@PJ}\def\cds@thick{\let\CD@rF\tenlnw\CD@LF\CD@NC\CD@BE\MapBreadth{2%
\CD@LF}\CD@@J}\def\cds@thin{\let\CD@rF\tenln\CD@BE\MapBreadth{\CD@NC}\CD@@J}%
\def\cds@tight{\CD@WB}\let\cds@top\cds@t\def\cds@TPIC{\CD@DK}\def
\cds@uppershortfall{\ifPositiveGradient\cds@rightshortfall\else
\cds@leftshortfall\fi}\def\cds@vcenter{\let\CD@dB\CD@cB}\let\cds@vcentre
\cds@vcenter\def\cds@vtriangleheight{\CD@BE\DiagramCellHeight
\DiagramCellHeight\DiagramCellWidth.577035\DiagramCellHeight}\def
\cds@vtrianglewidth{\CD@BE\DiagramCellWidth\DiagramCellWidth
\DiagramCellHeight1.73205\DiagramCellWidth}\def\cds@vmiddle{\let\CD@dB\CD@eB}%
\def\cds@w{\CD@zD\DiagramCellWidth}\let\cds@width\cds@w\def\diagram{\relax
\protect\CD@bC}\def\enddiagram{\protect\CD@SG}\def\CD@bC{\CD@g\CD@uI
\incommdiagtrue\edef\CD@wI{\the\CD@NB}\global\CD@NB\z@\boxmaxdepth\maxdimen
\everycr{}\CD@sK\everymath{}\everyhbox{}\ifx\pdfsyncstop\CD@qK\else
\pdfsyncstop\fi\CD@aC}\def\CD@aC{\CD@y\let\CD@N\CD@ZC\CD@CF\CD@AE\CD@WD}\def
\CD@ZC{\CD@gE\expandafter\CD@aC\else\expandafter\CD@WD\fi}\def\CD@WD{\let
\CD@EH\relax\CD@nE\CD@vE\else\CD@hK\else\CD@KB{landscape ignored without
PostScript}\CD@jA\fi\fi\fi\CD@EJ\setbox2=\vbox\bgroup\CD@JF\CD@VD}\def\CD@cH{%
\CD@nE\CD@fB\else\CD@dB\fi\CD@hA\nointerlineskip\setbox0=\null\ht0-\CD@pI\dp0%
\CD@pI\wd0\CD@kI\box0 \global\CD@QA\CD@kF\global\CD@yA\CD@XB\ifx\CD@NK\CD@qK
\global\CD@RA\CD@kF\else\global\CD@RA\CD@NK\fi\egroup\CD@zF\CD@nE\setbox2=%
\hbox to\dp2{\vrule height\wd2 depth\CD@QA width\z@\global\CD@QA\ht2\ht2\z@
\dp2\z@\wd2\z@\CD@hK\CD@tK{q 0 1 -1 0 0 0 cm}\else\global\CD@iG\CD@IK{0 1
bturn}\fi\box2\CD@gK\hss}\CD@DB\fi\ifnum\CD@yA=1 \else\CD@DB\fi\global
\@ignorefalse\CD@mE\leavevmode\fi\ifvmode\CD@TA\else\ifmmode\CD@PH\CD@GI\else
\CD@qE\CD@gA\fi\ifinner\CD@gA\fi\CD@mE\CD@GI\else\CD@sE\CD@QB\else\CD@TA\fi
\fi\fi\fi\CD@dD}\def\CD@dD{\global\CD@NB\CD@wI\relax\CD@xE\global\CD@ID\else
\aftergroup\CD@mC\fi\if@ignore\aftergroup\ignorespaces\fi\CD@wC\ignorespaces}%
\def\CD@fB{\advance\CD@pI\dimen1\relax}\def\CD@eB{\advance\CD@pI.5\dimen1%
\relax}\def\CD@bB{}\def\CD@cB{\CD@fB\advance\CD@pI\CD@YB\divide\CD@pI2
\advance\CD@pI-\axisheight\relax}\def\CD@aA{}\def\CD@iB{\CD@kF\z@}\def\CD@AA{%
\ifdim\dimen2>\CD@kF\CD@kF\dimen2 \else\dimen2\CD@kF\CD@kI\dimen0 \advance
\CD@kI\dimen2 \fi}\def\CD@QB{\skip0\z@\relax\loop\skip1\lastskip\ifdim\skip1>%
\z@\unskip\advance\skip0\skip1 \repeat\vadjust{\prevdepth\dp\strutbox\penalty
\predisplaypenalty\vskip\abovedisplayskip\CD@UA\penalty\postdisplaypenalty
\vskip\belowdisplayskip}\ifdim\skip0=\z@\else\hskip\skip0 \global\@ignoretrue
\fi}\def\CD@TA{\CD@LG\kern-\displayindent\CD@UA\CD@LG\global\@ignoretrue}\def
\CD@UA{\hbox to\hsize{\CD@fE\ifdim\CD@RA=\z@\else\advance\CD@QA-\CD@RA\setbox
2=\hbox{\kern\CD@RA\box2}\fi\fi\setbox1=\hbox{\ifx\CD@XA\CD@qK\else\CD@k
\CD@XA\CD@ND\fi}\CD@oE\CD@iE\else\advance\CD@QA\wd1 \fi\wd1\z@\box1 \fi\dimen
0\wd2 \advance\dimen0\wd1 \advance\dimen0-\hsize\ifdim\dimen0>-\CD@nA\CD@HA
\fi\advance\dimen0\CD@QA\ifdim\dimen0>\z@\CD@KB{wider than the page by \the
\dimen0 }\CD@HA\fi\CD@iE\hss\else\CD@V\CD@QA\CD@nA\fi\CD@GI\hss\kern-\wd1\box
1 }}\def\CD@GI{\CD@AF\CD@@F\else\CD@SC\global\CD@hG\fi\fi\kern\CD@QA\box2 }%
\CD@tG\CD@wE\CD@YC\CD@XC\def\CD@JF{\CD@cJ\ifdim\DiagramCellHeight=-\maxdimen
\DiagramCellHeight\CD@QC\fi\ifdim\DiagramCellWidth=-\maxdimen
\DiagramCellWidth\CD@QC\fi\global\CD@XC\CD@IF\let\CD@FE\empty\let\CD@z\CD@Q
\let\overprint\CD@eH\let\CD@s\CD@rJ\let\enddiagram\CD@ED\let\\\CD@cC\let\par
\CD@jH\let\CD@MD\empty\let\switch@arg\CD@PB\let\shift\CD@iA\baselineskip
\DiagramCellHeight\lineskip\z@\lineskiplimit\z@\mathsurround\z@\tabskip\z@
\CD@OB}\def\CD@VD{\penalty-123 \begingroup\CD@jA\aftergroup\CD@K\halign
\bgroup\global\advance\CD@NB1 \vadjust{\penalty1}\global\CD@FA\z@\CD@OB\CD@j#%
#\CD@DD\CD@Q\CD@Q\CD@OI\CD@j##\CD@DD\cr}\def\CD@ED{\CD@MD\CD@GD\crcr\egroup
\global\CD@JD\endgroup}\def\CD@j{\global\advance\CD@FA1 \futurelet\CD@EH\CD@i
}\def\CD@i{\ifx\CD@EH\CD@DD\CD@tJ\hskip1sp plus 1fil \relax\let\CD@DD\relax
\CD@vI\else\hfil\CD@k\objectstyle\let\CD@FE\CD@d\fi}\def\CD@DD{\CD@MD\relax
\CD@yI\CD@vI\global\CD@QA\CD@iA\penalty-9993 \CD@ND\hfil\null\kern-2\CD@QA
\null}\def\CD@cC{\cr}\def\across#1{\span\omit\mscount=#1 \global\advance
\CD@FA\mscount\global\advance\CD@FA\m@ne\CD@sF\ifnum\mscount>2 \CD@fJ\repeat
\ignorespaces}\def\CD@fJ{\relax\span\omit\advance\mscount\m@ne}\def\CD@qJ{%
\ifincommdiag\ifx\CD@iD\@fillh\ifx\CD@jD\@fillh\ifdim\dimen3>\z@\else\ifdim
\dimen2>93\CD@@I\ifdim\dimen2>18\p@\ifdim\CD@LF>\z@\count@\CD@bJ\advance
\count@\m@ne\ifnum\count@<\z@\count@20\let\CD@aJ\CD@uJ\fi\xdef\CD@bJ{\the
\count@}\fi\fi\fi\fi\fi\fi\fi}\def\CD@cG#1{\vrule\horizhtdp width#1\dimen@
\kern2\dimen@}\def\CD@uJ{\rlap{\dimen@\CD@@I\CD@V\dimen@{.182\p@}\CD@zH
\dimen@\advance\CD@tI\dimen@\CD@cG0\CD@cG0\CD@cG2\CD@cG6\CD@cG6\CD@cG2\CD@cG0%
\CD@cG0\CD@cG2\CD@cG6\CD@cG0\CD@cG0\CD@cG2\CD@cG2\CD@cG6\CD@cG0\CD@cG0\CD@cG2%
\CD@cG6\CD@cG2\CD@cG2\CD@cG0\CD@cG0}}\def\CD@bJ{10}\def\CD@aJ{}\def\CD@XD{%
\CD@gE\CD@TB\fi\CD@x\CD@WF\CD@HI}\def\CD@x{\CD@QJ\CD@DC\CD@MJ\ifdim\CD@DC=\z@
\else\CD@pF\CD@DC\fi\ifvoid3 \setbox3=\null\ht3\CD@tI\dp3\CD@sI\else\CD@V{\ht
3}\CD@tI\CD@V{\dp3}\CD@sI\fi\dimen3=.5\wd3 \ifdim\dimen3=\z@\CD@tE\else\dimen
3-\CD@XH\fi\else\CD@TB\fi\CD@V{\dimen2}{\wd7}\CD@V{\dimen2}{\wd6}\CD@qJ
\advance\dimen2-2\dimen3 \dimen4.5\dimen2 \dimen2\dimen4 \advance\dimen2%
\CD@eJ\advance\dimen4-\CD@eJ\advance\dimen2-\wd1 \advance\dimen4-\wd5 \ifvoid
2 \else\CD@V{\ht3}{\ht2}\CD@V{\dp3}{\dp2}\CD@V{\dimen2}{\wd2}\fi\ifvoid4 \else
\CD@V{\ht3}{\ht4}\CD@V{\dp3}{\dp4}\CD@V{\dimen4}{\wd4}\fi\advance\skip2\dimen
2 \advance\skip4\dimen4 \CD@tE\advance\skip2\skip4 \dimen0\dimen5 \advance
\dimen0\wd5 \skip3-\skip4 \advance\skip3-\dimen0 \let\CD@jD\empty\else\skip3%
\z@\relax\dimen0\z@\fi}\def\CD@WF{\offinterlineskip\lineskip.2\CD@zC\ifvoid6
\else\setbox3=\vbox{\hbox to2\dimen3{\hss\box6\hss}\box3}\fi\ifvoid7 \else
\setbox3=\vtop{\box3 \hbox to2\dimen3{\hss\box7\hss}}\fi}\def\CD@HI{\kern
\dimen1 \box1 \CD@aJ\CD@iD\hskip\skip2 \kern\dimen0 \ifincommdiag\CD@jE
\penalty1\fi\kern\dimen3 \penalty\CD@GB\hskip\skip3 \null\kern-\dimen3 \else
\hskip\skip3 \fi\box3 \CD@jD\hskip\skip4 \box5 \kern\dimen5}\def\CD@MF{\ifnum
\CD@LH>\CD@TC\CD@V{\dimen1}\objectheight\CD@V{\dimen5}\objectheight\else\CD@V
{\dimen1}\objectwidth\CD@V{\dimen5}\objectwidth\fi}\def\CD@Y{\begingroup
\ifdim\dimen7=\z@\kern\dimen8 \else\ifdim\dimen6=\z@\kern\dimen9 \else\dimen5%
\dimen6 \dimen6\dimen9 \CD@KJ\dimen4\dimen2 \CD@dG{\dimen4}\dimen6\dimen5
\dimen7\dimen8 \CD@KJ\CD@iC{\dimen2}\ifdim\dimen2<\dimen4 \kern\dimen2 \else
\kern\dimen4 \fi\fi\fi\endgroup}\def\CD@jJ{\CD@JI\setbox\z@\hbox{\lower
\axisheight\hbox to\dimen2{\CD@DF\ifPositiveGradient\dimen8\ht\CD@MH\dimen9%
\CD@mI\else\dimen8\dp3 \dimen9\dimen1 \fi\else\dimen8 \ifPositiveGradient
\objectheight\else\z@\fi\dimen9\objectwidth\fi\advance\dimen8
\ifPositiveGradient-\fi\axisheight\CD@Y\unhbox\z@\CD@DF\ifPositiveGradient
\dimen8\dp3 \dimen9\dimen0 \else\dimen8\ht\CD@MH\dimen9\CD@mF\fi\else\dimen8
\ifPositiveGradient\z@\else\objectheight\fi\dimen9\objectwidth\fi\advance
\dimen8 \ifPositiveGradient\else-\fi\axisheight\CD@Y}}}\def\CD@bD{\dimen6
\CD@aK\DiagramCellHeight\dimen7 \CD@WK\DiagramCellWidth\CD@jJ
\ifPositiveGradient\advance\dimen7-\CD@ZK\DiagramCellWidth\else\dimen7 \CD@ZK
\DiagramCellWidth\dimen6\z@\fi\advance\dimen6-\CD@bK\DiagramCellHeight\CD@mK
\setbox0=\rlap{\kern-\dimen7 \lower\dimen6\box\z@}\ht0\z@\dp0\z@\raise
\axisheight\box0 }\def\CD@mK{\setbox0\hbox{\ht\z@\z@\dp\z@\z@\wd\z@\z@\CD@hK
\expandafter\CD@tK{q \CD@eK\space\CD@lK\space\CD@kK\space\CD@eK\space0 0 cm}%
\else\global\CD@iG\CD@eD{\the\CD@TC\space\ifPositiveGradient\else-\fi\the
\CD@LH\space bturn}\fi\box\z@\CD@gK}}\def\CD@vB{\advance\CD@hF-\CD@mI\CD@wJ
\CD@hF\advance\CD@wJ\CD@hI\ifvoid\CD@sH\ifdim\CD@wJ<.1em\ifnum\CD@gD=\@m\else
\CD@aG h\CD@wJ<.1em:objects overprint:\CD@FA\CD@gD\fi\fi\else\ifhbox\CD@sH
\CD@SK\else\CD@TK\fi\advance\CD@wJ\CD@mI\CD@bH{-\CD@mI}{\box\CD@sH}{\CD@wJ}%
\z@\fi\CD@hF-\CD@mF\CD@gD\CD@FA\CD@hI\z@}\def\CD@SK{\setbox\CD@sH=\hbox{%
\unhbox\CD@sH\unskip\unpenalty}\setbox\CD@tH=\hbox{\unhbox\CD@tH\unskip
\unpenalty}\setbox\CD@sH=\hbox to\CD@wJ{\CD@OA\wd\CD@sH\unhbox\CD@sH\CD@PA
\lastkern\unkern\ifdim\CD@PA=\z@\CD@UB\advance\CD@OA-\wd\CD@tH\else\CD@TB\fi
\ifnum\lastpenalty=\z@\else\CD@JA\unpenalty\fi\kern\CD@PA\ifdim\CD@hF<\CD@OA
\CD@JA\fi\ifdim\CD@hI<\wd\CD@tH\CD@JA\fi\CD@jE\CD@hI\CD@wJ\advance\CD@hI-%
\CD@OA\advance\CD@hI\wd\CD@tH\ifdim\CD@hI<2\wd\CD@tH\CD@aG h\CD@hI<2\wd\CD@tH
:arrow too short:\CD@FA\CD@gD\fi\divide\CD@hI\tw@\CD@hF\CD@wJ\advance\CD@hF-%
\CD@hI\fi\CD@tE\kern-\CD@hI\fi\hbox to\CD@hI{\unhbox\CD@tH}\CD@HG}}\CD@tG
\ifinpile\inpiletrue\inpilefalse\inpilefalse\def\pile{\protect\CD@UJ\protect
\CD@uH}\def\CD@uH#1{\CD@l#1\CD@QD}\def\CD@UJ{\CD@nB{pile}\setbox0=\vtop
\bgroup\aftergroup\CD@lD\inpiletrue\let\CD@FE\empty\let\pile\CD@KF\let\CD@QD
\CD@PD\let\CD@GD\CD@FD\CD@yH\baselineskip.5\PileSpacing\lineskip.1\CD@zC
\relax\lineskiplimit\lineskip\mathsurround\z@\tabskip\z@\let\\\CD@wH}\def
\CD@l{\CD@DE\CD@YF\empty\halign\bgroup\hfil\CD@k\let\CD@FE\CD@d\let\\\CD@vH##%
\CD@MD\CD@ND\hfil\CD@Q\CD@R##\cr}\CD@rG\CD@NE{pile only allows one column.}%
\CD@rG\CD@UE{you left it out!}\def\CD@R{\CD@QD\CD@Q\relax\CD@YA{missing \CD@yC
\space inserted after \string\pile}\CD@NE}\def\CD@PD{\CD@MD\crcr\egroup
\egroup}\def\CD@GD{\CD@MD}\def\CD@FD{\CD@MD\relax\CD@QD\CD@YA{missing \CD@yC
\space inserted between \string\pile\space and \CD@HD}\CD@UE}\def\CD@QD{%
\CD@MD}\def\CD@lD{\vbox{\dimen1\dp0 \unvbox0 \setbox0=\lastbox\advance\dimen1%
\dp0 \nointerlineskip\box0 \nointerlineskip\setbox0=\null\dp0.5\dimen1\ht0-%
\dp0 \box0}\ifincommdiag\CD@tJ\penalty-9998 \fi\xdef\CD@YF{pile}}\def\CD@vH{%
\cr}\def\CD@wH{\noalign{\skip@\prevdepth\advance\skip@-\baselineskip
\prevdepth\skip@}}\def\CD@KF#1{#1}\def\CD@TK{\setbox\CD@sH=\vbox{\unvbox
\CD@sH\setbox1=\lastbox\setbox0=\box\voidb@x\CD@tF\setbox\CD@sH=\lastbox
\ifhbox\CD@sH\CD@rC\repeat\unvbox0 \global\CD@QA\CD@ZE}\CD@ZE\CD@QA}\def
\CD@rC{\CD@jE\setbox\CD@sH=\hbox{\unhbox\CD@sH\unskip\setbox\CD@sH=\lastbox
\unskip\unhbox\CD@sH}\ifdim\CD@wJ<\wd\CD@sH\CD@aG h\CD@wJ<\wd\CD@sH:arrow in
pile too short:\CD@FA\CD@gD\else\setbox\CD@sH=\hbox to\CD@wJ{\unhbox\CD@sH}%
\fi\else\CD@gJ\fi\setbox0=\vbox{\box\CD@sH\nointerlineskip\ifvoid0 \CD@tJ\box
1 \else\vskip\skip0 \unvbox0 \fi}\skip0=\lastskip\unskip}\def\CD@gJ{\penalty7
\noindent\unhbox\CD@sH\unskip\setbox\CD@sH=\lastbox\unskip\unhbox\CD@sH
\endgraf\setbox\CD@tH=\lastbox\unskip\setbox\CD@tH=\hbox{\CD@JG\unhbox\CD@tH
\unskip\unskip\unpenalty}\ifcase\prevgraf\cd@shouldnt P\or\ifdim\CD@wJ<\wd
\CD@tH\CD@aG h\CD@wJ<\wd\CD@sH:object in pile too wide:\CD@FA\CD@gD\setbox
\CD@sH=\hbox to\CD@wJ{\hss\unhbox\CD@tH\hss}\else\setbox\CD@sH=\hbox to\CD@wJ
{\hss\kern\CD@hF\unhbox\CD@tH\kern\CD@hI\hss}\fi\or\setbox\CD@sH=\lastbox
\unskip\CD@SK\else\cd@shouldnt Q\fi\unskip\unpenalty}\def\CD@cD{\CD@MJ\ifvoid
3 \setbox3=\null\ht3\axisheight\dp3-\ht3 \dimen3.5\CD@LF\else\dimen4\dp3
\dimen3.5\wd3 \setbox3=\CD@GG{\box3}\dp3\dimen4 \ifdim\ht3=-\dp3 \else\CD@TB
\fi\fi\dimen0\dimen3 \advance\dimen0-.5\CD@LF\setbox0\null\ht0\ht3\dp0\dp3\wd
0\wd3 \ifvoid6\else\setbox6\hbox{\unhbox6\kern\dimen0\kern2pt}\dimen0\wd6 \fi
\ifvoid7\else\setbox7\hbox{\kern2pt\kern\dimen3\unhbox7}\dimen3\wd7 \fi
\setbox3\hbox{\ifvoid6\else\kern-\dimen0\unhbox6\fi\unhbox3 \ifvoid7\else
\unhbox7\kern-\dimen3\fi}\ht3\ht0\dp3\dp0\wd3\wd0 \CD@tE\dimen4=\ht\CD@MH
\advance\dimen4\dp5 \advance\dimen4\dimen1 \let\CD@jD\empty\else\dimen4\ht3
\fi\setbox0\null\ht0\dimen4 \offinterlineskip\setbox8=\vbox spread2ex{\kern
\dimen5 \box1 \CD@iD\vfill\CD@tE\else\kern\CD@eJ\fi\box0}\ht8=\z@\setbox9=%
\vtop spread2ex{\kern-\ht3 \kern-\CD@eJ\box3 \CD@jD\vfill\box5 \kern\dimen1}%
\dp9=\z@\hskip\dimen0plus.0001fil \box9 \kern-\CD@LF\box8 \CD@kE\penalty2 \fi
\CD@tE\penalty1 \fi\kern\PileSpacing\kern-\PileSpacing\kern-.5\CD@LF\penalty
\CD@GB\null\kern\dimen3}\def\CD@cI{\ifhbox\CD@VA\CD@KB{clashing verticals}\ht
\CD@MH.5\dp\CD@VA\dp\CD@MH-\ht5 \CD@yB\ht\CD@MH\z@\dp\CD@MH\z@\fi\dimen1\dp
\CD@VA\CD@xA\prevgraf\unvbox\CD@VA\CD@wA\lastpenalty\unpenalty\setbox\CD@VA=%
\null\setbox\CD@lI=\hbox{\CD@JG\unhbox\CD@lI\unskip\unpenalty\dimen0\lastkern
\unkern\unkern\unkern\kern\dimen0 \CD@HG}\setbox\CD@lF=\hbox{\unhbox\CD@lF
\dimen0\lastkern\unkern\unkern\global\CD@QA\lastkern\unkern\kern\dimen0 }%
\CD@tF\ifnum\CD@xA>4 \CD@zI\repeat\unskip\unskip\advance\CD@mF.5\wd\CD@VA
\advance\CD@mF\wd\CD@lF\advance\CD@mI.5\wd\CD@VA\advance\CD@mI\wd\CD@lI\ifnum
\CD@FA=\CD@lA\CD@OA.5\wd\CD@VA\edef\CD@NK{\the\CD@OA}\fi\setbox\CD@VA=\hbox{%
\kern-\CD@mF\box\CD@lF\unhbox\CD@VA\box\CD@lI\kern-\CD@mI\penalty\CD@wA
\penalty\CD@NB}\ht\CD@VA\dimen1 \dp\CD@VA\z@\wd\CD@VA\CD@tB\CD@vB}\def\CD@zI{%
\ifdim\wd\CD@lF<\CD@QA\setbox\CD@lF=\hbox to\CD@QA{\CD@JG\unhbox\CD@lF}\fi
\advance\CD@xA\m@ne\setbox\CD@VA=\hbox{\box\CD@lF\unhbox\CD@VA}\unskip\setbox
\CD@lF=\lastbox\setbox\CD@lF=\hbox{\unhbox\CD@lF\unskip\unpenalty\dimen0%
\lastkern\unkern\unkern\global\CD@QA\lastkern\unkern\kern\dimen0 }}\def\CD@yB
{\dimen1\dp\CD@VA\ifhbox\CD@VA\CD@xB\else\CD@zB\fi\setbox\CD@VA=\vbox{%
\penalty\CD@NB}\dp\CD@VA-\dp\CD@MH\wd\CD@VA\CD@tB}\def\CD@zB{\unvbox\CD@VA
\CD@wA\lastpenalty\unpenalty\ifdim\dimen1<\ht\CD@MH\CD@aG v\dimen1<\ht\CD@MH:%
rows overprint:\CD@NB\CD@wA\fi}\def\CD@xB{\dimen0=\ht\CD@VA\setbox\CD@VA=%
\hbox\bgroup\advance\dimen1-\ht\CD@MH\unhbox\CD@VA\CD@xA\lastpenalty
\unpenalty\CD@wA\lastpenalty\unpenalty\global\CD@RA-\lastkern\unkern\setbox0=%
\lastbox\CD@tF\setbox\CD@VA=\hbox{\box0\unhbox\CD@VA}\setbox0=\lastbox\ifhbox
0 \CD@kJ\repeat\global\CD@SA-\lastkern\unkern\global\CD@QA\CD@JK\unhbox\CD@VA
\egroup\CD@JK\CD@QA\CD@bH{\CD@SA}{\box\CD@VA}{\CD@RA}{\dimen1}}\def\CD@kJ{%
\setbox0=\hbox to\wd0\bgroup\unhbox0 \unskip\unpenalty\dimen7\lastkern\unkern
\ifnum\lastpenalty=1 \unpenalty\CD@UB\else\CD@TB\fi\ifnum\lastpenalty=2
\unpenalty\dimen2.5\dimen0\advance\dimen2-.5\dimen1\advance\dimen2-%
\axisheight\else\dimen2\z@\fi\setbox0=\lastbox\dimen6\lastkern\unkern\setbox1%
=\lastbox\setbox0=\vbox{\unvbox0 \CD@tE\kern-\dimen1 \else\ifdim\dimen2=\z@
\else\kern\dimen2 \fi\fi}\ifdim\dimen0<\ht0 \CD@aG v\dimen0<\ht0:upper part of
vertical too short:{\CD@tE\CD@NB\else\CD@wA\fi}\CD@xA\else\setbox0=\vbox to%
\dimen0{\unvbox0}\fi\setbox1=\vtop{\unvbox1}\ifdim\dimen1<\dp1 \CD@aG v\dimen
1<\dp1:lower part of vertical too short:\CD@NB\CD@wA\else\setbox1=\vtop to%
\dimen1{\ifdim\dimen2=\z@\else\kern-\dimen2 \fi\unvbox1 }\fi\box1 \kern\dimen
6 \box0 \kern\dimen7 \CD@HG\global\CD@QA\CD@JK\egroup\CD@JK\CD@QA\relax}%
\countdef\CD@u=14 \newcount\CD@CA\newcount\CD@XB\newcount\CD@NB\let\CD@LB
\insc@unt\newcount\CD@FA\newcount\CD@lA\let\CD@mA\CD@XB\newcount\CD@MB\CD@tG
\CD@DF\CD@bI\CD@aI\CD@aI\def\CD@nD{-1}\def\CD@K{\CD@t\ifnum\CD@nD<\z@\else
\begingroup\scrollmode\showboxdepth\CD@nD\showboxbreadth\maxdimen\showlists
\endgroup\fi\CD@bI\CD@zF\CD@CA=\CD@u\advance\CD@CA1 \CD@XB=\CD@CA\ifnum\CD@NB
=1 \CD@JA\fi\advance\CD@XB\CD@NB\dimen1\z@\skip0\z@\count@=\insc@unt\advance
\count@\CD@u\divide\count@2 \ifnum\CD@XB>\count@\CD@KB{The diagram has too
many rows! It can't be reformatted.}\else\CD@NG\CD@WI\fi\CD@cH}\def\CD@NG{%
\CD@NB\CD@CA\CD@uF\ifnum\CD@NB<\CD@XB\setbox\CD@NB\box\voidb@x\advance\CD@NB1%
\relax\repeat\CD@NB\CD@CA\skip\z@\z@\CD@uF\CD@GB\lastpenalty\unpenalty\ifnum
\CD@GB>\z@\CD@KE\repeat\ifnum\CD@GB=-123 \CD@tJ\unpenalty\else\cd@shouldnt D%
\fi\ifx\v@grid\relax\else\CD@NB\CD@XB\advance\CD@NB\m@ne\expandafter\CD@VJ
\v@grid\fi\CD@MB\CD@mA\CD@tB\z@\CD@XG\ifx\h@grid\relax\else\expandafter\CD@LJ
\h@grid\fi\count@\CD@XB\advance\count@\m@ne\CD@YB\ht\count@}\def\CD@KE{%
\ifcase\CD@GB\or\CD@MG\else\CD@uA-\lastpenalty\unpenalty\CD@vA\lastpenalty
\unpenalty\setbox0=\lastbox\CD@WG\fi\CD@wD}\def\CD@wD{\skip1\lastskip\unskip
\advance\skip0\skip1 \ifdim\skip1=\z@\else\expandafter\CD@wD\fi}\def\CD@MG{%
\setbox0=\lastbox\CD@pI\dp0 \advance\CD@pI\skip\z@\skip\z@\z@\advance\CD@NF
\CD@pI\CD@uE\ifnum\CD@NB>\CD@CA\CD@NF\DiagramCellHeight\CD@pI\CD@NF\advance
\CD@pI-\CD@qI\fi\fi\CD@qI\ht0 \CD@NF\CD@qI\setbox\CD@NB\hbox{\unhbox\CD@NB
\unhbox0}\dp\CD@NB\CD@pI\ht\CD@NB\CD@qI\advance\CD@NB1 }\def\CD@WG{\ifnum
\CD@uA<\z@\advance\CD@uA\CD@XB\ifnum\CD@uA<\CD@CA\CD@UG\else\CD@OA\dp\CD@uA
\CD@PA\ht\CD@uA\setbox\CD@uA\hbox{\box\z@\penalty\CD@vA\penalty\CD@GB\unhbox
\CD@uA}\dp\CD@uA\CD@OA\ht\CD@uA\CD@PA\fi\else\CD@UG\fi}\def\CD@UG{\CD@KB{%
diagonal goes outside diagram (lost)}}\def\CD@fI{\advance\CD@uA\CD@XB\ifnum
\CD@uA<\CD@CA\CD@UG\else\ifnum\CD@uA=\CD@NB\CD@VG\else\ifnum\CD@uA>\CD@NB
\cd@shouldnt M\else\CD@OA\dp\CD@uA\CD@PA\ht\CD@uA\setbox\CD@uA\hbox{\box\z@
\penalty\CD@vA\penalty\CD@GB\unhbox\CD@uA}\dp\CD@uA\CD@OA\ht\CD@uA\CD@PA\fi
\fi\fi}\def\CD@WI{\CD@AJ\setbox\CD@PC=\hbox{\CD@k A\@super f\CD@lJ f\CD@ND}%
\CD@ZE\z@\CD@JK\z@\CD@kI\z@\CD@kF\z@\CD@NB=\CD@XB\CD@NF\z@\CD@uB\z@\CD@uF
\ifnum\CD@NB>\CD@CA\advance\CD@NB\m@ne\CD@qI\ht\CD@NB\CD@pI\dp\CD@NB\advance
\CD@NF\CD@qI\CD@rI\advance\CD@uB\CD@NF\CD@KC\CD@ZI\CD@w\ht\CD@NB\CD@qI\dp
\CD@NB\CD@pI\nointerlineskip\box\CD@NB\CD@NF\CD@pI\setbox\CD@NB\null\ht\CD@NB
\CD@uB\repeat\CD@wB\nointerlineskip\box\CD@NB\CD@gG\CD@ZE\DiagramCellWidth{%
width}\CD@gG\CD@JK\DiagramCellHeight{height}\CD@VA\CD@LB\advance\CD@VA-\CD@lA
\advance\CD@VA\m@ne\advance\CD@VA\CD@mA\dimen0\wd\CD@VA\CD@tI\axisheight
\dimen1\CD@uB\advance\dimen1-\CD@YB\dimen2\CD@kI\advance\dimen2-\dimen0
\advance\CD@XB-\CD@CA\advance\CD@LB-\CD@lA}\count@\year\multiply\count@12
\advance\count@\month\ifnum\count@>24398 \loop\iftrue\repeat\fi\def\CD@wB{\CD@qI-\CD@NF\CD@pI\CD@NF\setbox\CD@MH=\null\dp
\CD@MH\CD@NF\ht\CD@MH-\CD@NF\CD@mF\z@\CD@mI\z@\CD@lA\CD@LB\advance\CD@lA-%
\CD@MB\advance\CD@lA\CD@mA\CD@FA\CD@LB\CD@VA\CD@MB\CD@sF\ifnum\CD@FA>\CD@lA
\advance\CD@FA\m@ne\advance\CD@VA\m@ne\CD@tB\wd\CD@VA\setbox\CD@FA=\box
\voidb@x\CD@yB\repeat\CD@w\ht\CD@NB\CD@qI\dp\CD@NB\CD@pI}\def\CD@gG#1#2#3{%
\ifdim#1>.01\CD@zC\CD@PA#2\relax\advance\CD@PA#1\relax\advance\CD@PA.99\CD@zC
\count@\CD@PA\divide\count@\CD@zC\CD@KB{increase cell #3 to \the\count@ em}%
\fi}\def\CD@rI{\CD@FA=\CD@LB\penalty4 \noindent\unhbox\CD@NB\CD@sF\unskip
\setbox0=\lastbox\ifhbox0 \advance\CD@FA\m@ne\setbox\CD@FA\hbox to\wd0{\null
\penalty-9990\null\unhbox0}\repeat\CD@lA\CD@FA\advance\CD@FA\CD@MB\advance
\CD@FA-\CD@mA\ifnum\CD@FA<\CD@LB\count@\CD@FA\advance\count@\m@ne\dimen0=\wd
\count@\count@\CD@MB\advance\count@\m@ne\CD@tB\wd\count@\CD@sF\ifnum\CD@FA<%
\CD@LB\CD@DJ\CD@XG\dimen0\wd\CD@FA\advance\CD@FA1 \repeat\fi\CD@sF\CD@GB
\lastpenalty\unpenalty\ifnum\CD@GB>\z@\CD@vA\lastpenalty\unpenalty\CD@VG
\repeat\endgraf\unskip\ifnum\lastpenalty=4 \unpenalty\else\cd@shouldnt S\fi}%
\def\CD@VG{\advance\CD@vA\CD@lA\advance\CD@vA\m@ne\setbox0=\lastbox\ifnum
\CD@vA<\CD@LB\setbox\CD@vA\hbox{\box0\penalty\CD@GB\unhbox\CD@vA}\else\CD@UG
\fi}\def\CD@bG{}\CD@tG\CD@uE\CD@WB\CD@VB\def\CD@DJ{\advance\dimen0\wd\CD@FA
\divide\dimen0\tw@\CD@uE\dimen0\DiagramCellWidth\else\CD@V{\dimen0}%
\DiagramCellWidth\CD@pJ\fi\advance\CD@tB\dimen0 }\def\CD@XG{\setbox\CD@MB=%
\vbox{}\dp\CD@MB=\CD@uB\wd\CD@MB\CD@tB\advance\CD@MB1 }\def\CD@LJ#1,{\def
\CD@GK{#1}\ifx\CD@GK\CD@RD\else\advance\CD@tB\CD@GK\DiagramCellWidth\CD@XG
\expandafter\CD@LJ\fi}\def\CD@VJ#1,{\def\CD@GK{#1}\ifx\CD@GK\CD@RD\else\ifnum
\CD@NB>\CD@CA\CD@NF\CD@GK\DiagramCellHeight\advance\CD@NF-\dp\CD@NB\advance
\CD@NB\m@ne\ht\CD@NB\CD@NF\fi\expandafter\CD@VJ\fi}\def\CD@pJ{\CD@wE\CD@OA
\dimen0 \advance\CD@OA-\DiagramCellWidth\ifdim\CD@OA>2\MapShortFall\CD@KB{%
badly drawn diagonals (see manual)}\let\CD@pJ\empty\fi\else\let\CD@pJ\empty
\fi}\def\CD@KC{\CD@VA\CD@mA\CD@sF\ifnum\CD@VA<\CD@MB\dimen0\dp\CD@VA\advance
\dimen0\CD@NF\dp\CD@VA\dimen0 \advance\CD@VA1 \repeat}\def\CD@bH#1#2#3#4{%
\ifnum\CD@FA<\CD@LB\CD@OA=#1\relax\setbox\CD@FA=\hbox{\setbox0=#2\dimen7=#4%
\relax\dimen8=#3\relax\ifhbox\CD@FA\unhbox\CD@FA\advance\CD@OA-\lastkern
\unkern\fi\ifdim\CD@OA=\z@\else\kern-\CD@OA\fi\raise\dimen7\box0 \kern-\dimen
8 }\ifnum\CD@FA=\CD@lA\CD@V\CD@kF\CD@OA\fi\else\cd@shouldnt O\fi}\def\CD@w{%
\setbox\CD@NB=\hbox{\CD@FA\CD@lA\CD@VA\CD@mA\CD@PA\z@\relax\CD@sF\ifnum\CD@FA
<\CD@LB\CD@tB\wd\CD@VA\relax\CD@eI\advance\CD@FA1 \advance\CD@VA1 \repeat}%
\CD@V\CD@kI{\wd\CD@NB}\wd\CD@NB\z@}\def\CD@eI{\ifhbox\CD@FA\CD@OA\CD@tB\relax
\advance\CD@OA-\CD@PA\relax\ifdim\CD@OA=\z@\else\kern\CD@OA\fi\CD@PA\CD@tB
\advance\CD@PA\wd\CD@FA\relax\unhbox\CD@FA\advance\CD@PA-\lastkern\unkern\fi}%
\def\CD@ZI{\setbox\CD@sH=\box\voidb@x\CD@VA=\CD@MB\CD@FA\CD@LB\CD@VA\CD@mA
\advance\CD@VA\CD@FA\advance\CD@VA-\CD@lA\advance\CD@VA\m@ne\CD@tB\wd\CD@VA
\count@\CD@LB\advance\count@\m@ne\CD@hF.5\wd\count@\advance\CD@hF\CD@tB\CD@A
\m@ne\CD@gD\@m\CD@sF\ifnum\CD@FA>\CD@lA\advance\CD@FA\m@ne\advance\CD@hF-%
\CD@tB\CD@PI\wd\CD@VA\CD@tB\advance\CD@hF\CD@tB\advance\CD@VA\m@ne\CD@tB\wd
\CD@VA\repeat\CD@mF\CD@kF\CD@mI-\CD@mF\CD@vB}\newcount\CD@GB\def\CD@s{}\def
\CD@t{\CD@vK\mathsurround\z@\hsize\z@\rightskip\z@ plus1fil minus\maxdimen
\parfillskip\z@\linepenalty9000 \looseness0 \hfuzz\maxdimen\hbadness10000
\clubpenalty0 \widowpenalty0 \displaywidowpenalty0 \interlinepenalty0
\predisplaypenalty0 \postdisplaypenalty0 \interdisplaylinepenalty0
\interfootnotelinepenalty0 \floatingpenalty0 \brokenpenalty0 \everypar{}%
\leftskip\z@\parskip\z@\parindent\z@\pretolerance10000 \tolerance10000
\hyphenpenalty10000 \exhyphenpenalty10000 \binoppenalty10000 \relpenalty10000
\adjdemerits0 \doublehyphendemerits0 \finalhyphendemerits0 \baselineskip\z@
\CD@IA\def\parskip{\cd@shouldnt{PS}}\prevdepth\z@}\newbox\CD@KG\newbox\CD@IG
\def\CD@JG{\unhcopy\CD@KG}\def\CD@HG{\unhcopy\CD@IG}\def\CD@iJ{\hbox{}%
\penalty1\nointerlineskip}\def\CD@PI{\penalty5 \noindent\setbox\CD@MH=\null
\CD@mF\z@\CD@mI\z@\ifnum\CD@FA<\CD@LB\ht\CD@MH\ht\CD@FA\dp\CD@MH\dp\CD@FA
\unhbox\CD@FA\skip0=\lastskip\unskip\else\CD@OK\skip0=\z@\fi\endgraf\ifcase
\prevgraf\cd@shouldnt Y \or\cd@shouldnt Z \or\CD@RI\or\CD@XI\else\CD@QI\fi
\unskip\setbox0=\lastbox\unskip\unskip\unpenalty\noindent\unhbox0\setbox0%
\lastbox\unpenalty\unskip\unskip\unpenalty\setbox0\lastbox\CD@tF\CD@GB
\lastpenalty\unpenalty\ifnum\CD@GB>\z@\setbox\z@\lastbox\CD@lB\repeat\endgraf
\unskip\unskip\unpenalty}\def\CD@YJ{\CD@uA\CD@XB\advance\CD@uA-\CD@NB\CD@vA
\CD@FA\advance\CD@vA-\CD@lA\advance\CD@vA1 \expandafter\message{prevgraf=\the
\prevgraf at (\the\CD@uA,\the\CD@vA)}}\def\CD@XI{\CD@CE\setbox\CD@lI=\lastbox
\setbox\CD@lI=\hbox{\unhbox\CD@lI\unskip\unpenalty}\unskip\ifdim\ht\CD@lI>\ht
\CD@PC\setbox\CD@MH=\copy\CD@lI\else\ifdim\dp\CD@lI>\dp\CD@PC\setbox\CD@MH=%
\copy\CD@lI\else\CD@FG\CD@lI\fi\fi\advance\CD@mF.5\wd\CD@lI\advance\CD@mI.5%
\wd\CD@lI\setbox\CD@lI=\hbox{\unhbox\CD@lI\CD@HG}\CD@bH\CD@mF{\box\CD@lI}%
\CD@mI\z@\CD@yB\CD@vB}\def\CD@CE{\ifnum\CD@A>0 \advance\dimen0-\CD@tB\CD@iA-.%
5\dimen0 \CD@A-\CD@A\else\CD@A0 \CD@iA\z@\fi\setbox\CD@MH=\lastbox\setbox
\CD@MH=\hbox{\unhbox\CD@MH\unskip\unskip\unpenalty\setbox0=\lastbox\global
\CD@QA\lastkern\unkern}\advance\CD@iA-.5\CD@QA\unskip\setbox\CD@MH=\null
\CD@mI\CD@iA\CD@mF-\CD@iA}\def\CD@Z{\ht\CD@MH\CD@tI\dp\CD@MH\CD@sI}\def\CD@FG
#1{\setbox\CD@MH=\hbox{\CD@V{\ht\CD@MH}{\ht#1}\CD@V{\dp\CD@MH}{\dp#1}\CD@V{%
\wd\CD@MH}{\wd#1}\vrule height\ht\CD@MH depth\dp\CD@MH width\wd\CD@MH}}\def
\CD@QI{\CD@CE\CD@Z\setbox\CD@lI=\lastbox\unskip\setbox\CD@lF=\lastbox\unskip
\setbox\CD@lF=\hbox{\unhbox\CD@lF\unskip\global\CD@yA\lastpenalty\unpenalty}%
\advance\CD@yA9999 \ifcase\CD@yA\CD@VI\or\CD@YI\or\CD@TI\or\CD@dI\or\CD@cI\or
\CD@SI\else\cd@shouldnt9\fi}\def\CD@VI{\CD@FG\CD@lI\CD@UI\setbox\CD@sH=\box
\CD@lF\setbox\CD@tH=\box\CD@lI}\def\CD@YI{\CD@FG\CD@lF\setbox\CD@lI\hbox{%
\penalty8 \unhbox\CD@lI\unskip\unpenalty\ifnum\lastpenalty=8 \else\CD@xH\fi}%
\CD@UI\setbox\CD@lF=\hbox{\unhbox\CD@lF\unskip\unpenalty\global\setbox\CD@DA=%
\lastbox}\ifdim\wd\CD@lF=\z@\else\CD@xH\fi\setbox\CD@sH=\box\CD@DA}\def\CD@xH
{\CD@KB{extra material in \string\pile\space cell (lost)}}\def\CD@UI{\CD@yB
\ifvoid\CD@sH\else\CD@KB{Clashing horizontal arrows}\CD@mI.5\CD@hF\CD@mF-%
\CD@mI\CD@vB\CD@mI\z@\CD@mF\z@\fi\CD@hI\CD@hF\advance\CD@hI-\CD@mI\CD@hF-%
\CD@mF\CD@JC\CD@FA}\def\CD@RI{\setbox0\lastbox\unskip\CD@iA\z@\CD@Z\ifdim
\skip0>\z@\CD@tJ\CD@A0 \else\ifnum\CD@A<1 \CD@A0 \dimen0\CD@tB\fi\advance
\CD@A1 \fi}\def\VonH{\CD@MA46\VonH{.5\CD@LF}}\def\HonV{\CD@MA57\HonV{.5\CD@LF
}}\def\HmeetV{\CD@MA44\HmeetV{-\MapShortFall}}\def\CD@MA#1#2#3#4{\CD@pB34#1{%
\string#3}\CD@SD\CD@GB-999#2 \dimen0=#4\CD@tI\dimen0\advance\CD@tI\axisheight
\CD@sI\dimen0\advance\CD@sI-\axisheight\CD@CF\CD@HC\CD@ZD}\def\CD@HC#1{%
\setbox0=\hbox{\CD@k#1\CD@ND}\dimen0.5\wd0 \CD@tI\ht0 \CD@sI\dp0 \CD@ZD}\def
\CD@SD{\setbox0=\null\ht0=\CD@tI\dp0=\CD@sI\wd0=\dimen0 \copy0\penalty\CD@GB
\box0 }\def\CD@TI{\CD@GC\CD@yB}\def\CD@dI{\CD@GC\CD@vB}\def\CD@SI{\CD@GC
\CD@yB\CD@vB}\def\CD@GC{\setbox\CD@lI=\hbox{\unhbox\CD@lI}\setbox\CD@lF=\hbox
{\unhbox\CD@lF\global\setbox\CD@DA=\lastbox}\ht\CD@MH\ht\CD@DA\dp\CD@MH\dp
\CD@DA\advance\CD@mF\wd\CD@DA\advance\CD@mI\wd\CD@lI}\CD@tG
\ifPositiveGradient\CD@CI\CD@BI\CD@CI\CD@tG\ifClimbing\CD@rB\CD@qB\CD@rB
\newcount\DiagonalChoice\DiagonalChoice\m@ne\ifx\tenln\nullfont\CD@tJ\def
\CD@qF{\CD@KH\ifPositiveGradient/\else\CD@k\backslash\CD@ND\fi}\else\def
\CD@qF{\CD@rF\char\count@}\fi\let\CD@rF\tenln\def\Use@line@char#1{\hbox{#1%
\CD@rF\ifPositiveGradient\else\advance\count@64 \fi\char\count@}}\def\CD@cF{%
\Use@line@char{\count@\CD@TC\multiply\count@8\advance\count@-9\advance\count@
\CD@LH}}\def\CD@ZF{\Use@line@char{\ifcase\DiagonalChoice\CD@gF\or\CD@fF\or
\CD@fF\else\CD@gF\fi}}\def\CD@gF{\ifnum\CD@TC=\z@\count@'33 \else\count@
\CD@TC\multiply\count@\sixt@@n\advance\count@-9\advance\count@\CD@LH\advance
\count@\CD@LH\fi}\def\CD@fF{\count@'\ifcase\CD@LH55\or\ifcase\CD@TC66\or22\or
52\or61\or72\fi\or\ifcase\CD@TC66\or25\or22\or63\or52\fi\or\ifcase\CD@TC66\or
16\or36\or22\or76\fi\or\ifcase\CD@TC66\or27\or25\or67\or22\fi\fi\relax}\def
\CD@uC#1{\hbox{#1\setbox0=\Use@line@char{#1}\ifPositiveGradient\else\raise.3%
\ht0\fi\copy0 \kern-.7\wd0 \ifPositiveGradient\raise.3\ht0\fi\box0}}\def
\CD@jF#1{\hbox{\setbox0=#1\kern-.75\wd0 \vbox to.25\ht0{\ifPositiveGradient
\else\vss\fi\box0 \ifPositiveGradient\vss\fi}}}\def\CD@jI#1{\hbox{\setbox0=#1%
\dimen0=\wd0 \vbox to.25\ht0{\ifPositiveGradient\vss\fi\box0
\ifPositiveGradient\else\vss\fi}\kern-.75\dimen0 }}\CD@RC{+h:>}{%
\Use@line@char\CD@fF}\CD@RC{-h:>}{\Use@line@char\CD@gF}\CD@nF{+t:<}{-h:>}%
\CD@nF{-t:<}{+h:>}\CD@RC{+t:>}{\CD@jF{\Use@line@char\CD@fF}}\CD@RC{-t:>}{%
\CD@jI{\Use@line@char\CD@gF}}\CD@nF{+h:<}{-t:>}\CD@nF{-h:<}{+t:>}\CD@RC{+h:>>%
}{\CD@uC\CD@fF}\CD@RC{-h:>>}{\CD@uC\CD@gF}\CD@nF{+t:<<}{-h:>>}\CD@nF{-t:<<}{+%
h:>>}\CD@nF{+h:>->}{+h:>>}\CD@nF{-h:>->}{-h:>>}\CD@nF{+t:<-<}{-h:>>}\CD@nF{-t%
:<-<}{+h:>>}\CD@RC{+t:>>}{\CD@jF{\CD@uC\CD@fF}}\CD@RC{-t:>>}{\CD@jI{\CD@uC
\CD@gF}}\CD@nF{+h:<<}{-t:>>}\CD@nF{-h:<<}{+t:>>}\CD@nF{+t:>->}{+t:>>}\CD@nF{-%
t:>->}{-t:>>}\CD@nF{+h:<-<}{-t:>>}\CD@nF{-h:<-<}{+t:>>}\CD@RC{+f:-}{\CD@EF
\null\else\CD@cF\fi}\CD@nF{-f:-}{+f:-}\def\CD@tC#1#2{\vbox to#1{\vss\hbox to#%
2{\hss.\hss}\vss}}\def\hfdot{\CD@tC{2\axisheight}{.5em}}%
\def\vfdot{\CD@tC{1ex}\z@}%% % 1.46ex until 29.7.98
\def\CD@bF{\hbox{\dimen0=.3\CD@zC\dimen1\dimen0 \ifnum\CD@LH>\CD@TC\CD@iC{%
\dimen1}\else\CD@dG{\dimen0}\fi\CD@tC{\dimen0}{\dimen1}}}\newarrowfiller{.}%
\hfdot\hfdot\vfdot\vfdot\def\dfdot{\CD@bF\CD@CK}\CD@RC{+f:.}{\dfdot}\CD@RC{-f%
:.}{\dfdot}\def\CD@@K#1{\hbox\bgroup\def\CD@CH{#1\egroup}\afterassignment
\CD@CH%%
\count@='}\def\lnchar{\CD@@K\CD@qF}\def\CD@dF#1{\setbox#1=\hbox{\dimen5\dimen
#1 \setbox8=\box#1 \dimen1\wd8 \count@\dimen5 \divide\count@\dimen1 \ifnum
\count@=0 \box8 \ifdim\dimen5<.95\dimen1 \CD@gB{diagonal line too short}\fi
\else\dimen3=\dimen5 \advance\dimen3-\dimen1 \divide\dimen3\count@\dimen4%
\dimen3 \CD@dG{\dimen4}\ifPositiveGradient\multiply\dimen4\m@ne\fi\dimen6%
\dimen1 \advance\dimen6-\dimen3 \loop\raise\count@\dimen4\copy8 \ifnum\count@
>0 \kern-\dimen6 \advance\count@\m@ne\repeat\fi}}\def\CD@CG#1{\CD@EF\CD@xJ{#1%
}\else\CD@dF{#1}\fi}\def\CD@IH#1{}\newdimen\objectheight\objectheight1.8ex
\newdimen\objectwidth\objectwidth1em \def\CD@YD{\dimen6=\CD@aK
\DiagramCellHeight\dimen7=\CD@WK\DiagramCellWidth\CD@KJ\ifnum\CD@LH>0 \ifnum
\CD@TC>0 \CD@aF\else\aftergroup\CD@VC\fi\else\aftergroup\CD@UC\fi}\def\CD@VC{%
\CD@YA{diagonal map is nearly vertical}\CD@NA}\def\CD@UC{\CD@YA{diagonal map
is nearly horizontal}\CD@NA}\CD@rG\CD@NA{Use an orthogonal map instead}\def
\CD@aF{\CD@MJ\dimen3\dimen7\dimen7\dimen6\CD@iC{\dimen7}\advance\dimen3-%
\dimen7 \CD@MF\ifnum\CD@LH>\CD@TC\advance\dimen6-\dimen1\advance\dimen6-%
\dimen5 \CD@iC{\dimen1}\CD@iC{\dimen5}\else\dimen0\dimen1\advance\dimen0%
\dimen5\CD@dG{\dimen0}\advance\dimen6-\dimen0 \fi\dimen2.5\dimen7\advance
\dimen2-\dimen1 \dimen4.5\dimen7\advance\dimen4-\dimen5 \ifPositiveGradient
\dimen0\dimen5 \advance\dimen1-\CD@WK\DiagramCellWidth\advance\dimen1 \CD@ZK
\DiagramCellWidth\setbox6=\llap{\unhbox6\kern.1\ht2}\setbox7=\rlap{\kern.1\ht
2\unhbox7}\else\dimen0\dimen1 \advance\dimen1-\CD@ZK\DiagramCellWidth\setbox7%
=\llap{\unhbox7\kern.1\ht2}\setbox6=\rlap{\kern.1\ht2\unhbox6}\fi\setbox6=%
\vbox{\box6\kern.1\wd2}\setbox7=\vtop{\kern.1\wd2\box7}\CD@dG{\dimen0}%
\advance\dimen0-\axisheight\advance\dimen0-\CD@bK\DiagramCellHeight\dimen5-%
\dimen0 \advance\dimen0\dimen6 \advance\dimen1.5\dimen3 \ifdim\wd3>\z@\ifdim
\ht3>-\dp3\CD@TB\fi\fi\dimen3\dimen2 \dimen7\dimen2\advance\dimen7\dimen4
\ifvoid3 \else\CD@tE\else\dimen8\ht3\advance\dimen8-\axisheight\CD@iC{\dimen8%
}\CD@X{\dimen8}{.5\wd3}\dimen9\dp3\advance\dimen9\axisheight\CD@iC{\dimen9}%
\CD@X{\dimen9}{.5\wd3}\ifPositiveGradient\advance\dimen2-\dimen9\advance
\dimen4-\dimen8 \else\advance\dimen4-\dimen9\advance\dimen2-\dimen8 \fi\fi
\advance\dimen3-.5\wd3 \fi\dimen9=\CD@aK\DiagramCellHeight\advance\dimen9-2%
\DiagramCellHeight\CD@tE\advance\dimen2\dimen4 \CD@CG{2}\dimen2-\dimen0%
\advance\dimen2\dp2 \else\CD@CG{2}\CD@CG{4}\ifPositiveGradient\dimen2-\dimen0%
\advance\dimen2\dp2 \dimen4\dimen5\advance\dimen4-\ht4 \else\dimen4-\dimen0%
\advance\dimen4\dp4 \dimen2\dimen5\advance\dimen2-\ht2 \fi\fi\setbox0=\hbox to%
\z@{\kern\dimen1 \ifvoid1 \else\ifPositiveGradient\advance\dimen0-\dp1 \lower
\dimen0 \else\advance\dimen5-\ht1 \raise\dimen5 \fi\rlap{\unhbox1}\fi\raise
\dimen2\rlap{\unhbox2}\ifvoid3 \else\lower.5\dimen9\rlap{\kern\dimen3\unhbox3%
}\fi\kern.5\dimen7 \lower.5\dimen9\box6 \lower.5\dimen9\box7 \kern.5\dimen7
\CD@tE\else\raise\dimen4\llap{\unhbox4}\fi\ifvoid5 \else\ifPositiveGradient
\advance\dimen5-\ht5 \raise\dimen5 \else\advance\dimen0-\dp5 \lower\dimen0 \fi
\llap{\unhbox5}\fi\hss}\ht0=\axisheight\dp0=-\ht0\box0 }\def\NorthWest{\CD@BI
\CD@rB\DiagonalChoice0 }\def\NorthEast{\CD@CI\CD@rB\DiagonalChoice1 }\def
\SouthWest{\CD@CI\CD@qB\DiagonalChoice3 }\def\SouthEast{\CD@BI\CD@qB
\DiagonalChoice2 }\def\CD@aD{\vadjust{\CD@uA\CD@FA\advance\CD@uA
\ifPositiveGradient\else-\fi\CD@ZK\relax\CD@vA\CD@NB\advance\CD@vA-\CD@bK
\relax\hbox{\advance\CD@uA\ifPositiveGradient-\fi\CD@WK\advance\CD@vA\CD@aK
\hbox{\box6 \kern\CD@DC\kern\CD@eJ\penalty1 \box7 \box\z@}\penalty\CD@uA
\penalty\CD@vA}\penalty\CD@uA\penalty\CD@vA\penalty104}}\def\CD@eH#1{\relax
\vadjust{\hbox@maths{#1}\penalty\CD@FA\penalty\CD@NB\penalty\tw@}}\def\CD@lB{%
\ifcase\CD@GB\or\or\CD@bH{.5\wd0}{\box0}{.5\wd0}\z@\or\unhbox\z@\setbox\z@
\lastbox\CD@bH{.5\wd0}{\box0}{.5\wd0}\z@\unpenalty\unpenalty\setbox\z@
\lastbox\or\CD@TG\else\advance\CD@GB-100 \ifnum\CD@GB<\z@\cd@shouldnt B\fi
\setbox\z@\hbox{\kern\CD@mF\copy\CD@MH\kern\CD@mI\CD@uA\CD@XB\advance\CD@uA-%
\CD@NB\penalty\CD@uA\CD@uA\CD@FA\advance\CD@uA-\CD@lA\penalty\CD@uA\unhbox\z@
\global\CD@yA\lastpenalty\unpenalty\global\CD@zA\lastpenalty\unpenalty}\CD@uA
-\CD@yA\CD@vA\CD@zA\CD@fI\fi}\def\CD@TG{\unhbox\z@\setbox\z@\lastbox\CD@uA
\lastpenalty\unpenalty\advance\CD@uA\CD@mA\CD@vA\CD@XB\advance\CD@vA-%
\lastpenalty\unpenalty\dimen1\lastkern\unkern\setbox3\lastbox\dimen0\lastkern
\unkern\setbox0=\hbox to\z@{\unhbox0\setbox0\lastbox\setbox7\lastbox
\unpenalty\CD@eJ\lastkern\unkern\CD@DC\lastkern\unkern\setbox6\lastbox\dimen7%
\CD@tB\advance\dimen7-\wd\CD@uA\ifdim\dimen7<\z@\CD@CI\multiply\dimen7\m@ne
\let\mv\empty\else\CD@BI\def\mv{\raise\ht1}\kern-\dimen7 \fi\ifnum\CD@vA>%
\CD@NB\dimen6\CD@uB\advance\dimen6-\ht\CD@vA\else\dimen6\z@\fi\CD@jJ\CD@mK
\setbox1\null\ht1\dimen6\wd1\dimen7 \dimen7\dimen2 \dimen6\wd1 \CD@KJ\CD@uA
\CD@LH\CD@vA\CD@TC\dimen6\ht1 \CD@KJ\setbox2\null\divide\dimen2\tw@\advance
\dimen2\CD@eJ\CD@eG{\dimen2}\wd2\dimen2 \dimen0.5\dimen7 \advance\dimen0%
\ifPositiveGradient\else-\fi\CD@eJ\CD@dG{\dimen0}\advance\dimen0-\axisheight
\ht2\dimen0 \dimen0\CD@DC\CD@eG{\dimen0}\advance\dimen0\ht2\ht2\dimen0 \dimen
0\ifPositiveGradient-\fi\CD@DC\CD@dG{\dimen0}\advance\dimen0\wd2\wd2\dimen0
\setbox4\null\dimen0 .6\CD@zC\CD@eG{\dimen0}\ht4\dimen0 \dimen0 .2\CD@zC
\CD@dG{\dimen0}\wd4\dimen0 \dimen0\wd2 \ifvoid6\else\dimen1\ht4 \advance
\dimen1\ht2 \CD@CC6+-\raise\dimen1\rlap{\ifPositiveGradient\advance\dimen0-%
\wd6\advance\dimen0-\wd4 \else\advance\dimen0\wd4 \fi\kern\dimen0\box6}\fi
\dimen0\wd2 \ifvoid7\else\dimen1\ht4 \advance\dimen1-\ht2 \CD@CC7-+\lower
\dimen1\rlap{\ifPositiveGradient\advance\dimen0\wd4 \else\advance\dimen0-\wd7%
\advance\dimen0-\wd4 \fi\kern\dimen0\box7}\fi\mv\box0\hss}\ht0\z@\dp0\z@
\CD@bH{\z@}{\box\z@}{\z@}{\axisheight}}\def\CD@CC#1#2#3{\dimen4.5\wd#1 \ifdim
\dimen4>.25\dimen7\dimen4=.25\dimen7\fi\ifdim\dimen4>\CD@zC\dimen4.4\dimen4
\advance\dimen4.6\CD@zC\fi\CD@eG{\dimen4}\dimen5\axisheight\CD@dG{\dimen5}%
\advance\dimen4-\dimen5 \dimen5\dimen4\CD@eG{\dimen5}\advance\dimen0%
\ifPositiveGradient#2\else#3\fi\dimen5 \CD@dG{\dimen4}\advance\dimen1\dimen4 }
\def\CD@eD#1{\expandafter\CD@IK{#1}}\CD@ZA\CD@EK{output is PostScript
dependent}\def\CD@SC{\CD@IK{/bturn {gsave currentpoint currentpoint translate
4 2 roll neg exch atan rotate neg exch neg exch translate } def /eturn {%
currentpoint grestore moveto} def}}\def\CD@gK{\relax\CD@hK\CD@tK{Q}\else
\CD@IK{eturn}\fi} \def\CD@OJ#1{\count@#1\relax\multiply\count@7\advance
\count@16577\divide\count@33154 }\def\CD@fD#1{\expandafter\special{#1}} \def
\CD@xJ#1{\setbox#1=\hbox{\dimen0\dimen#1\CD@dG{\dimen0}\CD@OJ{\dimen0}\setbox
0=\null\ifPositiveGradient\count@-\count@\ht0\dimen0 \else\dp0\dimen0 \fi\box
0 \CD@uA\count@\CD@OJ\CD@LF\CD@fD{pn \the\count@}\CD@fD{pa 0 0}\CD@OJ{\dimen#%
1}\CD@fD{pa \the\count@\space\the\CD@uA}\CD@fD{fp}\kern\dimen#1}}\def\CD@JI{%
\CD@KJ\begingroup\ifdim\dimen7<\dimen6 \dimen2=\dimen6 \dimen6=\dimen7 \dimen
7=\dimen2 \count@\CD@LH\CD@LH\CD@TC\CD@TC\count@\else\dimen2=\dimen7 \fi
\ifdim\dimen6>.01\p@\CD@KI\global\CD@QA\dimen0 \else\global\CD@QA\dimen7 \fi
\endgroup\dimen2\CD@QA\CD@iK\CD@lK{\ifPositiveGradient\else-\fi\dimen6}\CD@iK
\CD@kK{\ifPositiveGradient-\fi\dimen6}\CD@iK\CD@eK{\dimen7}}\def\CD@KI{\CD@hJ
\ifdim\dimen7>1.73\dimen6 \divide\dimen2 4 \multiply\CD@TC2 \else\dimen2=0.%
353553\dimen2 \advance\CD@LH-\CD@TC\multiply\CD@TC4 \fi\dimen0=4\dimen2 \CD@ZG
4\CD@ZG{-2}\CD@ZG2\CD@ZG{-2.5}}\def\CD@AI{\begingroup\count@\dimen0 \dimen2 45%
pt \divide\count@\dimen2 \ifdim\dimen0<\z@\advance\count@\m@ne\fi\ifodd
\count@\advance\count@1\CD@@A\else\CD@y\fi\advance\dimen0-\count@\dimen2
\CD@gE\multiply\dimen0\m@ne\fi\ifnum\count@<0 \multiply\count@-7 \fi\dimen3%
\dimen1 \dimen6\dimen0 \dimen7 3754936sp \ifdim\dimen0<6\p@\def\CD@OG{4000}%
\fi\CD@KJ\dimen2\dimen3\CD@dG{\dimen2}\CD@hJ\multiply\CD@TC-6 \dimen0\dimen2
\CD@ZG1\CD@ZG{0.3}\dimen1\dimen0 \dimen2\dimen3 \dimen0\dimen3 \CD@ZG3\CD@ZG{%
1.5}\CD@ZG{0.3}\divide\count@2 \CD@gE\multiply\dimen1\m@ne\fi\ifodd\count@
\dimen2\dimen1\dimen1\dimen0\dimen0-\dimen2 \fi\divide\count@2 \ifodd\count@
\multiply\dimen0\m@ne\multiply\dimen1\m@ne\fi\global\CD@QA\dimen0\global
\CD@RA\dimen1\endgroup\dimen6\CD@QA\dimen7\CD@RA}\def\CD@OC{255}\let\CD@OG
\CD@OC\def\CD@KJ{\begingroup\ifdim\dimen7<\dimen6 \dimen9\dimen7\dimen7\dimen
6\dimen6\dimen9\CD@@A\else\CD@y\fi\dimen2\z@\dimen3\CD@XH\dimen4\CD@XH\dimen0%
\z@\dimen8=\CD@OG\CD@XH\CD@lC\global\CD@yA\dimen\CD@gE0\else3\fi\global\CD@zA
\dimen\CD@gE3\else0\fi\endgroup\CD@LH\CD@yA\CD@TC\CD@zA}\def\CD@lC{\count@
\dimen6 \divide\count@\dimen7 \advance\dimen6-\count@\dimen7 \dimen9\dimen4
\advance\dimen9\count@\dimen0 \ifdim\dimen9>\dimen8 \CD@@C\else\CD@AC\ifdim
\dimen6>\z@\dimen9\dimen6 \dimen6\dimen7 \dimen7\dimen9 \expandafter
\expandafter\expandafter\CD@lC\fi\fi}\def\CD@@C{\ifdim\dimen0=\z@\ifdim\dimen
9<2\dimen8 \dimen0\dimen8 \fi\else\advance\dimen8-\dimen4 \divide\dimen8%
\dimen0 \ifdim\count@\CD@XH<2\dimen8 \count@\dimen8 \dimen9\dimen4 \advance
\dimen9\count@\dimen0 \CD@AC\fi\fi}\def\CD@AC{\dimen4\dimen0 \dimen0\dimen9
\advance\dimen2\count@\dimen3 \dimen9\dimen2 \dimen2\dimen3 \dimen3\dimen9 }%
\def\CD@ZG#1{\CD@dG{\dimen2}\advance\dimen0 #1\dimen2 }\def\CD@dG#1{\divide#1%
\CD@TC\multiply#1\CD@LH}\def\CD@eG#1{\divide#1\CD@vA\multiply#1\CD@uA}\def
\CD@iC#1{\divide#1\CD@LH\multiply#1\CD@TC}\def\CD@hJ{\dimen6\CD@LH\CD@XH
\multiply\dimen6\CD@LH\dimen7\CD@TC\CD@XH\multiply\dimen7\CD@TC\CD@KJ}\def
\CD@iK#1#2{\begingroup\dimen@#2\relax\loop\ifdim\dimen2<.4\maxdimen\multiply
\dimen2\tw@\multiply\dimen@\tw@\repeat\divide\dimen2\@cclvi\divide\dimen@
\dimen2\relax\multiply\dimen@\@cclvi\expandafter\CD@jK\the\dimen@\endgroup
\let#1\CD@fK}{\catcode`p=12 \catcode`0=12 \catcode`.=12 \catcode`t=12 \gdef
\CD@jK#1pt{\gdef\CD@fK{#1}}}\ifx\errorcontextlines\CD@qK\CD@tJ\let\CD@GH
\relax\else\def\CD@GH{\errorcontextlines\m@ne}\fi\ifnum\inputlineno<0 \let
\CD@CD\empty\let\CD@W\empty\let\CD@mD\relax\let\CD@uI\relax\let\CD@vI\relax
\let\CD@zF\relax\else\def\CD@W{ at line \number\inputlineno}\def\CD@mD{ - first occurred%
}\def\CD@uI{\edef\CD@h{\the\inputlineno}\global\let\CD@jB\CD@h}\def\CD@h{9999%
}\def\CD@vI{\xdef\CD@jB{\the\inputlineno}}\def\CD@jB{\CD@h}\def\CD@zF{\ifnum
\CD@h<\inputlineno\edef\CD@CD{\space at lines \CD@h--\the\inputlineno}\else
\edef\CD@CD{\CD@W}\fi}\fi\let\CD@CD\empty\def\CD@YA#1#2{\CD@GH\errhelp=#2%
\expandafter\errmessage{\CD@tA: #1}}\def\CD@KB#1{\begingroup\expandafter
\message{! \CD@tA: #1\CD@CD}\ifnum\CD@XB>\CD@NB\ifnum\CD@CA>\CD@NB\else\ifnum
\CD@lA>\CD@FA\else\ifnum\CD@LB>\CD@FA\advance\CD@XB-\CD@NB\advance\CD@FA-%
\CD@lA\advance\CD@FA1\relax\expandafter\message{! (error detected at row \the
\CD@XB, column \the\CD@FA, but probably caused elsewhere)}\fi\fi\fi\fi
\endgroup}\def\CD@gB#1{{\expandafter\message{\CD@tA\space Warning: #1\CD@W}}}%
\def\CD@CB#1#2{\CD@gB{#1 \string#2 is obsolete\CD@mD}}\def\CD@AB#1{\CD@CB{%
Dimension}{#1}\CD@DE#1\CD@BB\CD@BB}\def\CD@BB{\CD@OA=}\def\CD@@B#1{\CD@CB{%
Count}{#1}\CD@DE#1\CD@OH\CD@OH}\def\CD@OH{\count@=}\def\HorizontalMapLength{%
\CD@AB\HorizontalMapLength}\def\VerticalMapHeight{\CD@AB\VerticalMapHeight}%
\def\VerticalMapDepth{\CD@AB\VerticalMapDepth}\def\VerticalMapExtraHeight{%
\CD@AB\VerticalMapExtraHeight}\def\VerticalMapExtraDepth{\CD@AB
\VerticalMapExtraDepth}\def\DiagonalLineSegments{\CD@@B\DiagonalLineSegments}%
\ifx\tenln\nullfont\CD@ZA\CD@KH{\CD@eF\space diagonal line and arrow font not
available}\else\let\CD@KH\relax\fi\def\CD@aG#1#2<#3:#4:#5#6{\begingroup\CD@PA
#3\relax\advance\CD@PA-#2\relax\ifdim.1em<\CD@PA\CD@uA#5\relax\CD@vA#6\relax
\ifnum\CD@uA<\CD@vA\count@\CD@vA\advance\count@-\CD@uA\CD@KB{#4 by \the\CD@PA
}\if#1v\let\CD@CH\CD@JK\edef\tmp{\the\CD@uA--\the\CD@vA,\the\CD@FA}\else
\advance\count@\count@\if#1l\advance\count@-\CD@A\else\if#1r\advance\count@
\CD@A\fi\fi\advance\CD@PA\CD@PA\let\CD@CH\CD@ZE\edef\tmp{\the\CD@NB,\the
\CD@uA--\the\CD@vA}\fi\divide\CD@PA\count@\ifdim\CD@CH<\CD@PA\global\CD@CH
\CD@PA\fi\fi\fi\endgroup}\CD@tG\CD@xE\CD@JD\CD@ID\CD@rG\CD@xI{See the message
above.}\CD@rG\CD@lH{Perhaps you've forgotten to end the diagram before
resuming the text, in\CD@uG which case some garbage may be added to the
diagram, but we should be ok now.\CD@uG Alternatively you've left a blank line
in the middle - TeX will now complain\CD@uG that the remaining \CD@S s are
misplaced - so please use comments for layout.}\CD@rG\CD@hD{You have already
closed too many brace pairs or environments; an \CD@HD\CD@uG command was (%
over)due.}\CD@rG\CD@hH{\CD@dC\space and \CD@HD\space commands must match.}%
\def\CD@jH{\ifnum\inputlineno=0 \else\expandafter\CD@iH\fi}\def\CD@iH{\CD@MD
\CD@GD\crcr\CD@YA{missing \CD@HD\space inserted before \CD@kH- type "h"}%
\CD@lH\enddiagram\CD@AG\CD@kH\par}\def\CD@AG#1{\edef\enddiagram{\noexpand
\CD@rD{#1\CD@W}}}\def\CD@rD#1{\CD@YA{\CD@HD\space(anticipated by #1) ignored}%
\CD@xI\let\enddiagram\CD@SG}\def\CD@SG{\CD@YA{misplaced \CD@HD\space ignored}%
\CD@hH}\def\CD@mC{\CD@YA{missing \CD@HD\space inserted.}\CD@hD\CD@AG{closing
group}}\ifx\ifx
\fi\fi

\catcode`\$=3 %% make sure that $ means maths-shift
\def\vboxtoz{\vbox to\z@}%% \z@ is in plain TeX and means 0pt

\def\scriptaxis#1{\@scriptaxis{$\scriptstyle#1$}}%%
\def\ssaxis#1{\@ssaxis{$\scriptscriptstyle#1$}}%%
\def\@scriptaxis#1{\dimen0\axisheight\advance\dimen0-\ss@axisheight\raise
\dimen0\hbox{#1}}\def\@ssaxis#1{\dimen0\axisheight\advance\dimen0-%
\ss@axisheight\raise\dimen0\hbox{#1}}

\ifx\boldmath\CD@qK%%
\let\boldscriptaxis\scriptaxis%%
\def\boldscript#1{\hbox{$\scriptstyle#1$}}%%
\else\def\boldscriptaxis#1{\@scriptaxis{\boldmath$\scriptstyle#1$}}%%
\def\boldscript#1{\hbox{\boldmath$\scriptstyle#1$}}%%
\fi

\def\raisehook#1#2#3{\hbox{\setbox3=\hbox{#1$\scriptscriptstyle#3$}%
\dimen0\ss@axisheight%% \scriptscriptstyle axis height
\dimen1\axisheight\advance\dimen1-\dimen0%% difference in axis heights
\dimen2\ht3\advance\dimen2-\dimen0%
\advance\dimen2-0.021em\advance\dimen1 #2\dimen2%
\raise\dimen1\box3}}%% print the character
\def\shifthook#1#2#3{\setbox1=\hbox{#1$\scriptscriptstyle#3$}\dimen0\wd1%
\divide\dimen0 12\CD@zH{\dimen0}%%  "u"
\dimen1\wd1\advance\dimen1-2\dimen0 \advance\dimen1-2\CD@oI\CD@zH{\dimen1}%
\kern#2\dimen1\box1}%% print

\def\@cmex{\mathchar"03}%%ascii double quote

\def\make@pbk#1{\setbox\tw@\hbox to\z@{#1}\ht\tw@\z@\dp\tw@\z@\box\tw@}\def
\CD@fH#1{\overprint{\hbox to\z@{#1}}}\def\CD@qH{\kern0.11em}\def\CD@pH{\kern0%
.35em}

\def\dblvert{\def\CD@rH{\kern.5\PileSpacing}}\def\CD@rH{}

\def\SEpbk{\make@pbk{\CD@qH\CD@rH\vrule depth 2.87ex height -2.75ex width 0.%
95em \vrule height -0.66ex depth 2.87ex width 0.05em \hss}}

\def\SWpbk{\make@pbk{\hss\vrule height -0.66ex depth 2.87ex width 0.05em
\vrule depth 2.87ex height -2.75ex width 0.95em \CD@qH\CD@rH}}

\def\NEpbk{\make@pbk{\CD@qH\CD@rH\vrule depth -3.81ex height 4.00ex width 0.%
95em \vrule height 4.00ex depth -1.72ex width 0.05em \hss}}

\def\NWpbk{\make@pbk{\hss\vrule height 4.00ex depth -1.72ex width 0.05em
\vrule depth -3.81ex height 4.00ex width 0.95em \CD@qH\CD@rH}}

\def\puncture{{\setbox0\hbox{A}\vrule height.53\ht0 depth-.47\ht0 width.35\ht
0 \kern.12\ht0 \vrule height\ht0 depth-.65\ht0 width.06\ht0 \kern-.06\ht0
\vrule height.35\ht0 depth0pt width.06\ht0 \kern.12\ht0 \vrule height.53\ht0
depth-.47\ht0 width.35\ht0 }}

\def\NEclck{\overprint{\raise2.5ex\rlap{ \CD@rH$\scriptstyle\searrow$}}}%%
\def\NEanti{\overprint{\raise2.5ex\rlap{ \CD@rH$\scriptstyle\nwarrow$}}}%%
\def\NWclck{\overprint{\raise2.5ex\llap{$\scriptstyle\nearrow$ \CD@rH}}}%%
\def\NWanti{\overprint{\raise2.5ex\llap{$\scriptstyle\swarrow$ \CD@rH}}}%%
\def\SEclck{\overprint{\lower1ex\rlap{ \CD@rH$\scriptstyle\swarrow$}}}%%
\def\SEanti{\overprint{\lower1ex\rlap{ \CD@rH$\scriptstyle\nearrow$}}}%%
\def\SWclck{\overprint{\lower1ex\llap{$\scriptstyle\nwarrow$ \CD@rH}}}%%
\def\SWanti{\overprint{\lower1ex\llap{$\scriptstyle\searrow$ \CD@rH}}}

\def\rhvee{\mkern-10mu\greaterthan}%%
\def\lhvee{\lessthan\mkern-10mu}%%
\def\dhvee{\vboxtoz{\vss\hbox{$\vee$}\kern0pt}}%%
\def\uhvee{\vboxtoz{\hbox{$\wedge$}\vss}}%%
\newarrowhead{vee}\rhvee\lhvee\dhvee\uhvee

\def\dhlvee{\vboxtoz{\vss\hbox{$\scriptstyle\vee$}\kern0pt}}%%
\def\uhlvee{\vboxtoz{\hbox{$\scriptstyle\wedge$}\vss}}%%
\newarrowhead{littlevee}{\mkern1mu\scriptaxis\rhvee}{\scriptaxis\lhvee}%
\dhlvee\uhlvee\ifx\boldmath\CD@qK%%
\newarrowhead{boldlittlevee}{\mkern1mu\scriptaxis\rhvee}{\scriptaxis\lhvee}%
\dhlvee\uhlvee\else%%
\def\dhblvee{\vboxtoz{\vss\boldscript\vee\kern0pt}}%%
\def\uhblvee{\vboxtoz{\boldscript\wedge\vss}}%%
\newarrowhead{boldlittlevee}{\mkern1mu\boldscriptaxis\rhvee}{\boldscriptaxis
\lhvee}\dhblvee\uhblvee%%
\fi

\def\rhcvee{\mkern-10mu\succ}%%
\def\lhcvee{\prec\mkern-10mu}%%
\def\dhcvee{\vboxtoz{\vss\hbox{$\curlyvee$}\kern0pt}}%%
\def\uhcvee{\vboxtoz{\hbox{$\curlywedge$}\vss}}%%
\newarrowhead{curlyvee}\rhcvee\lhcvee\dhcvee\uhcvee

\def\rhvvee{\mkern-13mu\gg}%% 24.8.92 changed 10mu to 13mu
\def\lhvvee{\ll\mkern-13mu}%% to make rule go through
\def\dhvvee{\vboxtoz{\vss\hbox{$\vee$}\kern-.6ex\hbox{$\vee$}\kern0pt}}%%
\def\uhvvee{\vboxtoz{\hbox{$\wedge$}\kern-.6ex \hbox{$\wedge$}\vss}}%%
\newarrowhead{doublevee}\rhvvee\lhvvee\dhvvee\uhvvee

\def\rhtriangle{\triangleright\mkern1.2mu}%% 29.1.93
\def\lhtriangle{\triangleleft\mkern.8mu}%%
\def\uhtriangle{\vbox{\kern-.2ex \hbox{$\scriptscriptstyle\bigtriangleup$}%
\kern-.25ex}}%%
\def\dhtriangle{\vbox{\kern-.28ex \hbox{$\scriptscriptstyle\bigtriangledown$}%
\kern-.1ex}}%% 15.1.93 from -.25ex
\def\dhblack{\vbox{\kern-.25ex\nointerlineskip\hbox{$\blacktriangledown$}}}%
\def\uhblack{\vbox{\kern-.25ex\nointerlineskip\hbox{$\blacktriangle$}}}%
\def\dhlblack{\vbox{\kern-.25ex\nointerlineskip\hbox{$\scriptstyle
\blacktriangledown$}}}%% AMS
\def\uhlblack{\vbox{\kern-.25ex\nointerlineskip\hbox{$\scriptstyle
\blacktriangle$}}}%% AMS
\newarrowhead{triangle}\rhtriangle\lhtriangle\dhtriangle\uhtriangle
\newarrowhead{blacktriangle}{\mkern-1mu\blacktriangleright\mkern.4mu}{%
\blacktriangleleft}\dhblack\uhblack\newarrowhead{littleblack}{\mkern-1mu%
\scriptaxis\blacktriangleright}{\scriptaxis\blacktriangleleft\mkern-2mu}%
\dhlblack\uhlblack

\def\rhla{\hbox{\setbox0=\lnchar55\dimen0=\wd0\kern-.6\dimen0\ht0\z@\raise
\axisheight\box0\kern.1\dimen0}}%%
\def\lhla{\hbox{\setbox0=\lnchar33\dimen0=\wd0\kern.05\dimen0\ht0\z@\raise
\axisheight\box0\kern-.5\dimen0}}%%
\def\dhla{\vboxtoz{\vss\rlap{\lnchar77}}}%%
\def\uhla{\vboxtoz{\setbox0=\lnchar66 \wd0\z@\kern-.15\ht0\box0\vss}}%% 1/93
\newarrowhead{LaTeX}\rhla\lhla\dhla\uhla

\def\lhlala{\lhla\kern.3em\lhla}%%
\def\rhlala{\rhla\kern.3em\rhla}%%
\def\uhlala{\hbox{\uhla\raise-.6ex\uhla}}%%
\def\dhlala{\hbox{\dhla\lower-.6ex\dhla}}%%
\newarrowhead{doubleLaTeX}\rhlala\lhlala\dhlala\uhlala

\def\hhO{\scriptaxis\bigcirc\mkern.4mu} \def\hho{{\circ}\mkern1.2mu}%
\newarrowhead{o}\hho\hho\circ\circ%%
\newarrowhead{O}\hhO\hhO{\scriptstyle\bigcirc}{\scriptstyle\bigcirc}%%

\def\rhtimes{\mkern-5mu{\times}\mkern-.8mu}\def\lhtimes{\mkern-.8mu{\times}%
\mkern-5mu}\def\uhtimes{\setbox0=\hbox{$\times$}\ht0\axisheight\dp0-\ht0%
\lower\ht0\box0 }\def\dhtimes{\setbox0=\hbox{$\times$}\ht0\axisheight\box0 }%
\newarrowhead{X}\rhtimes\lhtimes\dhtimes\uhtimes\newarrowhead+++++

\newarrowhead{Y}{\mkern-3mu\Yright}{\Yleft\mkern-3mu}\Ydown\Yup

\newarrowhead{->}\rightarrow\leftarrow\downarrow\uparrow

\newarrowhead{=>}\Rightarrow\Leftarrow{\@cmex7F}{\@cmex7E}

\newarrowhead{harpoon}\rightharpoonup\leftharpoonup\downharpoonleft
\upharpoonleft

\def\twoheaddownarrow{\rlap{$\downarrow$}\raise-.5ex\hbox{$\downarrow$}}%%
\def\twoheaduparrow{\rlap{$\uparrow$}\raise.5ex\hbox{$\uparrow$}}%%
\newarrowhead{->>}\twoheadrightarrow\twoheadleftarrow\twoheaddownarrow
\twoheaduparrow

\def\ltvee{\mkern-1mu{\lessthan}\mkern.4mu}%% \mkern added 15.1.93
\newarrowtail{vee}\greaterthan\ltvee\vee\wedge

\newarrowtail{littlevee}{\scriptaxis\greaterthan}{\mkern-1mu\scriptaxis
\lessthan}{\scriptstyle\vee}{\scriptstyle\wedge}\ifx\boldmath\CD@qK
\newarrowtail{boldlittlevee}{\scriptaxis\greaterthan}{\mkern-1mu\scriptaxis
\lessthan}{\scriptstyle\vee}{\scriptstyle\wedge}\else\newarrowtail{%
boldlittlevee}{\boldscriptaxis\greaterthan}{\mkern-1mu\boldscriptaxis
\lessthan}{\boldscript\vee}{\boldscript\wedge}\fi

\newarrowtail{curlyvee}\succ{\mkern-1mu{\prec}\mkern.4mu}\curlyvee\curlywedge

\def\rttriangle{\mkern1.2mu\triangleright}%% 29.1.93
\newarrowtail{triangle}\rttriangle\lhtriangle\dhtriangle\uhtriangle
\newarrowtail{blacktriangle}\blacktriangleright{\mkern-1mu\blacktriangleleft
\mkern.4mu}\dhblack\uhblack\newarrowtail{littleblack}{\scriptaxis
\blacktriangleright\mkern-2mu}{\mkern-1mu\scriptaxis\blacktriangleleft}%
\dhlblack\uhlblack

\def\rtla{\hbox{\setbox0=\lnchar55\dimen0=\wd0\kern-.5\dimen0\ht0\z@\raise
\axisheight\box0\kern-.2\dimen0}}%%
\def\ltla{\hbox{\setbox0=\lnchar33\dimen0=\wd0\kern-.15\dimen0\ht0\z@\raise
\axisheight\box0\kern-.5\dimen0}}%%
\def\dtla{\vbox{\setbox0=\rlap{\lnchar77}\dimen0=\ht0\kern-.7\dimen0\box0%
\kern-.1\dimen0}}%% 15.1.93 from -.6
\def\utla{\vbox{\setbox0=\rlap{\lnchar66}\dimen0=\ht0\kern-.1\dimen0\box0%
\kern-.6\dimen0}}%%
\newarrowtail{LaTeX}\rtla\ltla\dtla\utla

\def\rtvvee{\gg\mkern-3mu}%%
\def\ltvvee{\mkern-3mu\ll}%%
\def\dtvvee{\vbox{\hbox{$\vee$}\kern-.6ex \hbox{$\vee$}\vss}}%%
\def\utvvee{\vbox{\vss\hbox{$\wedge$}\kern-.6ex \hbox{$\wedge$}\kern\z@}}%%
\newarrowtail{doublevee}\rtvvee\ltvvee\dtvvee\utvvee

\def\ltlala{\ltla\kern.3em\ltla}%%
\def\rtlala{\rtla\kern.3em\rtla}%%
\def\utlala{\hbox{\utla\raise-.6ex\utla}}%%
\def\dtlala{\hbox{\dtla\lower-.6ex\dtla}}%%
\newarrowtail{doubleLaTeX}\rtlala\ltlala\dtlala\utlala

\def\utbar{\vrule height 0.093ex depth0pt width 0.4em}%%
\let\dtbar\utbar%%
\def\rtbar{\mkern1.5mu\vrule height 1.1ex depth.06ex width .04em\mkern1.5mu}%
\let\ltbar\rtbar%%
\newarrowtail{mapsto}\rtbar\ltbar\dtbar\utbar%%
\newarrowtail{|}\rtbar\ltbar\dtbar\utbar%%ascii vertical bar (|)

\def\rthooka{\raisehook{}+\subset\mkern-1mu}%%
\def\lthooka{\mkern-1mu\raisehook{}+\supset}%%
\def\rthookb{\raisehook{}-\subset\mkern-2mu}%%
\def\lthookb{\mkern-1mu\raisehook{}-\supset}%%

\def\dthooka{\shifthook{}+\cap}%%
\def\dthookb{\shifthook{}-\cap}%%
\def\uthooka{\shifthook{}+\cup}%%
\def\uthookb{\shifthook{}-\cup}%%

\newarrowtail{hooka}\rthooka\lthooka\dthooka\uthooka\newarrowtail{hookb}%
\rthookb\lthookb\dthookb\uthookb

\ifx\boldmath\CD@qK\newarrowtail{boldhooka}\rthooka\lthooka\dthooka\uthooka
\newarrowtail{boldhookb}\rthookb\lthookb\dthookb\uthookb\newarrowtail{%
boldhook}\rthooka\lthooka\dthookb\uthooka\else\def\rtbhooka{\raisehook
\boldmath+\subset\mkern-1mu}%%
\def\ltbhooka{\mkern-1mu\raisehook\boldmath+\supset}%%
\def\rtbhookb{\raisehook\boldmath-\subset\mkern-2mu}%%
\def\ltbhookb{\mkern-1mu\raisehook\boldmath-\supset}%%
\def\dtbhooka{\shifthook\boldmath+\cap}%%
\def\dtbhookb{\shifthook\boldmath-\cap}%%
\def\utbhooka{\shifthook\boldmath+\cup}%%
\def\utbhookb{\shifthook\boldmath-\cup}%%
\newarrowtail{boldhooka}\rtbhooka\ltbhooka\dtbhooka\utbhooka\newarrowtail{%
boldhookb}\rtbhookb\ltbhookb\dtbhookb\utbhookb\newarrowtail{boldhook}%
\rtbhooka\ltbhooka\dtbhooka\utbhooka\fi

\def\dtsqhooka{\shifthook{}+\sqcap}%%
\def\ltsqhooka{\mkern-1mu\raisehook{}+\sqsupset}%%
\def\rtsqhooka{\raisehook{}+\sqsubset\mkern-1mu}%%
\def\utsqhooka{\shifthook{}+\sqcup}%%
\newarrowtail{sqhook}\rtsqhooka\ltsqhooka\dtsqhooka\utsqhooka

\newarrowtail{hook}\rthooka\lthookb\dthooka\uthooka\newarrowtail{C}\rthooka
\lthookb\dthooka\uthooka

\newarrowtail{o}\hho\hho\circ\circ%%
\newarrowtail{O}\hhO\hhO{\scriptstyle\bigcirc}{\scriptstyle\bigcirc}%%

\newarrowtail{X}\lhtimes\rhtimes\uhtimes\dhtimes\newarrowtail+++++

\newarrowtail{Y}\Yright\Yleft\Ydown\Yup

\newarrowtail{harpoon}\leftharpoondown\rightharpoondown\upharpoonright
\downharpoonright

\newarrowtail{<=}\Leftarrow\Rightarrow{\@cmex7E}{\@cmex7F}

\newarrowfiller{=}=={\@cmex77}{\@cmex77}%% 16.1.93
\def\vfthree{\mid\!\!\!\mid\!\!\!\mid}%%ascii
\newarrowfiller{3}\equiv\equiv\vfthree\vfthree

\def\vfdashstrut{\vrule width0pt height1.3ex depth0.7ex}%%
\def\vfthedash{\vrule width\CD@LF height0.6ex depth 0pt}%%
\def\hfthedash{\CD@AJ\vrule\horizhtdp width 0.26em}%%
\def\hfdash{\mkern5.5mu\hfthedash\mkern5.5mu}%%
\def\vfdash{\vfdashstrut\vfthedash}%%
\newarrowfiller{dash}\hfdash\hfdash\vfdash\vfdash

\newarrowmiddle+++++

\iffalse%%
\newarrow{To}----{vee}%%
\newarrow{Arr}----{LaTeX}%%
\newarrow{Dotsto}....{vee}%%
\newarrow{Dotsarr}....{LaTeX}%%
\newarrow{Dashto}{}{dash}{}{dash}{vee}%%
\newarrow{Dasharr}{}{dash}{}{dash}{LaTeX}%%
\newarrow{Mapsto}{mapsto}---{vee}%%
\newarrow{Mapsarr}{mapsto}---{LaTeX}%%
\newarrow{IntoA}{hooka}---{vee}%%
\newarrow{IntoB}{hookb}---{vee}%%
\newarrow{Embed}{vee}---{vee}%%
\newarrow{Emarr}{LaTeX}---{LaTeX}%%
\newarrow{Onto}----{doublevee}%%
\newarrow{Dotsonarr}....{doubleLaTeX}%%
\newarrow{Dotsonto}....{doublevee}%%
\newarrow{Dotsonarr}....{doubleLaTeX}%%
\else%%
\newarrow{To}---->%%
\newarrow{Arr}---->%%
\newarrow{Dotsto}....>%%
\newarrow{Dotsarr}....>%%
\newarrow{Dashto}{}{dash}{}{dash}>%%
\newarrow{Dasharr}{}{dash}{}{dash}>%%
\newarrow{Mapsto}{mapsto}--->%%
\newarrow{Mapsarr}{mapsto}--->%%
\newarrow{IntoA}{hooka}--->%%
\newarrow{IntoB}{hookb}--->%%
\newarrow{Embed}>--->%%
\newarrow{Emarr}>--->%%
\newarrow{Onto}----{>>}%%
\newarrow{Dotsonarr}....{>>}%%
\newarrow{Dotsonto}....{>>}%%
\newarrow{Dotsonarr}....{>>}%%
\fi%%

\newarrow{Implies}===={=>}%% minimum cell height 9.5pt
\newarrow{Project}----{triangle}%%
\newarrow{Pto}----{harpoon}%% partial function
\newarrow{Relto}{harpoon}---{harpoon}%% binary relation

\newarrow{Eq}=====%%
\newarrow{Line}-----%%
\newarrow{Dots}.....%%
\newarrow{Dashes}{}{dash}{}{dash}{}%%

\newarrow{SquareInto}{sqhook}--->

\newarrowhead{cmexbra}{\@cmex7B}{\@cmex7C}{\@cmex3B}{\@cmex38}%%
\newarrowtail{cmexbra}{\@cmex7A}{\@cmex7D}{\@cmex39}{\@cmex3A}%%
\newarrowmiddle{cmexbra}{\braceru\bracelu}{\bracerd\braceld}{\vcenter{%
\hbox@maths{\@cmex3D\mkern-2mu}}}%% right
{\vcenter{\hbox@maths{\mkern2mu\@cmex3C}}}%% left
\newarrow{@brace}{cmexbra}-{cmexbra}-{cmexbra}%% braces
\newarrow{@parenth}{cmexbra}---{cmexbra}%% straight parentheses
\def\rightBrace{\d@brace[thick,cmex]}%%ASCII square brackets []
\def\leftBrace{\u@brace[thick,cmex]}%%ASCII square brackets []
\def\upperBrace{\r@brace[thick,cmex]}%%ASCII square brackets []
\def\lowerBrace{\l@brace[thick,cmex]}%%ASCII square brackets []
\def\rightParenth{\d@parenth[thick,cmex]}%%ASCII square brackets []
\def\leftParenth{\u@parenth[thick,cmex]}%%ASCII square brackets []
\def\upperParenth{\r@parenth[thick,cmex]}%%ASCII square brackets []
\def\lowerParenth{\l@parenth[thick,cmex]}%%ASCII square brackets []

\let\hEq\rEq%%
\let\vEq\uEq%%
\def\labelstyle{%%
\ifincommdiag%%
\textstyle%%
\else%%
\scriptstyle%%
\fi}%%
\let\objectstyle\displaystyle

\newdiagramgrid{pentagon}{0.618034,0.618034,1,1,1,1,0.618034,0.618034}{1.%
17557,1.17557,1.902113,1.902113}

\newdiagramgrid{perspective}{0.75,0.75,1.1,1.1,0.9,0.9,0.95,0.95,0.75,0.75}{0%
.75,0.75,1.1,1.1,0.9,0.9}

\diagramstyle[%%ascii open square bracket
dpi=300,%%              office laserwriters are usually 300 dots per inch
vmiddle,nobalance,%%    vertical and horizontal positioning
loose,%%                allow rows and columns to stretch
thin,%%                 line10 arrows; default rule thickness (TeXbook p447)
pilespacing=10pt,%
shortfall=4pt,%%        distance between arrowheads and their targets
]%%ascii close square bracket

\ifx\CD@qK\else\CD@PK\relax\CD@FF\CD@e\fi\fi

\CD@vE\CD@hK\else\fi\else\fi\cdrestoreat
\usepackage{amssymb}
\usepackage{mathrsfs}
\usepackage[all,arc]{xy}
\usepackage{cite}
\usepackage{tikz-cd}
\usepackage{MnSymbol}
\usepackage{a4wide}
\makeatletter
\let\@wraptoccontribs\wraptoccontribs
\makeatother

\DeclareMathOperator\C{\mathbb C}
\DeclareMathOperator\Z{\mathbb Z}

\newtheorem{theorem}{Theorem}[section]
\newtheorem{lemma}[theorem]{Lemma}

\newtheorem{prop}[theorem]{Proposition}
\newtheorem{claim}[theorem]{Claim}
\theoremstyle{definition}
\newtheorem{definition}[theorem]{Definition}

\newtheorem{case}{Case}
\theoremstyle{remark}
\newtheorem{remark}[theorem]{Remark}
\newtheorem{conjecture}[theorem]{Conjecture}

\numberwithin{equation}{section}

\newcommand{\dontprint}[1]\relax

\newcommand{\Id}{\operatorname{Id}}

\newcommand{\Gr}{\operatorname{Gr}}

\newcommand{\sm}{\setminus}
\newcommand{\sms}{\smallskip}

\newcommand{\mC}{{\mathbb C}}

\newcommand{\bo}{\omega}
\newcommand{\bl}{\lambda}
\newcommand{\bL}{\Lambda}

\newcommand{\mcF}{\mathcal F}
\newcommand{\mcG}{\mathcal G}

\newcommand{\mcL}{\mathcal L}

\newcommand{\mcS}{\mathcal S}

\newcommand{\mcW}{\mathcal W}

\newcommand{\mcY}{\mathcal Y}

\newcommand{\fh}{\mathfrak h}

\newcommand{\fn}{\mathfrak n}

\newcommand{\ft}{\mathfrak t}

\newcommand{\fz}{\mathfrak z}
\newcommand{\ti}{\tilde}

\newcommand{\ad}{\operatorname{ad}}
\newcommand{\ssl}{\mathfrak sl}

\newcommand{\can}{\operatorname{can}}

\newcommand{\fg}{{\frak g}}
\newcommand{\fb}{{\frak b}}

\newcommand{\Ad}{\operatorname{Ad}}

\newcommand{\De}{\Delta}
\newcommand{\La}{\Lambda}
\newcommand{\Ga}{\Gamma}
\newcommand{\Aut}{\operatorname{Aut}}

\newcommand{\und}{\underline}
\newcommand{\Pic}{\operatorname{Pic}}
\newcommand{\hra}{\hookrightarrow}

\renewcommand{\P}{{\mathbb P}}
\newcommand{\A}{{\mathbb A}}

\newcommand{\wt}{\widetilde}
\newcommand{\ot}{\otimes}
\newcommand{\fu}{{\mathfrak u}}

\newcommand{\Hom}{\operatorname{Hom}}

\newcommand{\HH}{{\mathcal H}}
\newcommand{\XX}{{\mathcal X}}
\newcommand{\YY}{{\mathcal Y}}

\newcommand{\CC}{{\mathcal C}}
\newcommand{\EE}{{\mathcal E}}
\renewcommand{\SS}{{\mathcal S}}
\newcommand{\FF}{{\mathcal F}}

\newcommand{\GG}{{\mathcal G}}
\newcommand{\LL}{{\mathcal L}}
\newcommand{\MM}{{\mathcal M}}
\newcommand{\OO}{{\mathcal O}}

\newcommand{\UU}{{\mathcal U}}
\newcommand{\WW}{{\mathcal W}}
\newcommand{\si}{\sigma}
\newcommand{\de}{\delta}
\newcommand{\sub}{\subset}
\newcommand{\Spec}{\operatorname{Spec}}

\newcommand{\PGL}{\operatorname{PGL}}
\newcommand{\ov}{\overline}

\newcommand{\om}{\omega}
\newcommand{\la}{\lambda}
\renewcommand{\a}{\alpha}
\renewcommand{\b}{\beta}

\newcommand{\id}{\operatorname{id}}
\newcommand{\GL}{\operatorname{GL}}
\newcommand{\SO}{\operatorname{SO}}
\newcommand{\tot}{\operatorname{tot}}
\newcommand{\G}{{\mathbb G}}

\newcommand{\ga}{\gamma}
\newcommand{\lan}{\langle}
\newcommand{\ran}{\rangle}

\newcommand{\triv}{{\operatorname{triv}}}

\newcommand{\SL}{{\operatorname{SL}}}
\newcommand{\End}{{\operatorname{End}}}
\newcommand{\Bun}{{\operatorname{Bun}}}

\newcommand{\cusp}{{\operatorname{cusp}}}
\newcommand{\fm}{{\mathfrak m}}

\newcommand{\eps}{\epsilon}

\newcommand{\SSet}{{\mathbb S}et}
\newcommand{\bSet}{{\bf S}et}
\newcommand{\VVect}{{\mathbb V}ect}
\newcommand{\Pro}{{\operatorname{Pro}}}
\newcommand{\Ind}{{\operatorname{Ind}}}
\renewcommand{\H}{{\mathbb H}}
\newcommand{\Coinv}{{\operatorname{Coinv}}}
\newcommand{\Rep}{{\operatorname{Rep}}}
\newcommand{\Quot}{{\operatorname{Quot}}}

\title[Hecke operators]{Hecke operators for curves over non-archimedean local fields and related finite rings}

\author{Alexander Braverman}
\author{David Kazhdan}
\author{Alexander Polishchuk}
\contrib[with an appendix by]{Alexander Polishchuk and Ka Fai Wong}
\thanks{D.K. is partially supported by the ERC grant No 101142781.
A.P. is partially supported by the NSF grants DMS-2001224, 
NSF grant DMS-2349388, by the Simons Travel grant MPS-TSM-00002745,
and within the framework of the HSE University Basic Research Program}

\address{Deparment of Mathematics, University of Toronto, Ontario, Canada}
\email{braval@math.toronto.edu}
\address{Einstein Institute of Mathematics,
The Hebrew University of Jerusalem,
Jerusalem 91904, Israel}
\email{kazhdan@math.huji.ac.il}
\address{
    Department of Mathematics, 
    University of Oregon,     
    Eugene, OR 97403, USA; National Research University Higher School of Economics
  }
  \email{apolish@uoregon.edu}
\address{Department of Mathematics,University of Oregon,     
    Eugene, OR 97403, USA}
\email{kwong9@uoregon.edu}

\begin{document}

\begin{abstract} We study Hecke operators associated with curves over a non-archimedean local field $K$ and over the rings $O/\fm^N$, where
$O\sub K$ is the ring of integers. Our main result is commutativity of a certain ``small" local Hecke algebra over $O/\fm^N$, associated with
a connected split reductive group $G$ such that $[G,G]$ is simply connected. The proof uses a Hecke algebra associated with $G(K(\!(t)\!))$
and a global argument involving $G$-bundles on curves. 
\end{abstract}

\maketitle

\section{Introduction}

Let $C$ be a smooth proper curve over a non-archimedean local field $K$.
In this paper we study Hecke operators on certain vector spaces related to 
the moduli space of $G$-bundles on $C$ and over the related curves over finite rings $O/\fm^n$,
under the assumption of existence of a smooth model $C_O$ over the ring of integers $O\sub K$.

We refer to \cite{BK} for a survey of conjectures and approaches to the analog of the Langlands program in the
case of curves over local fields, and to \cite{EFK1}, \cite{EFK2} for more precise conjectures and results in the archimedean case.

For a connected split reductive group $G$ we consider the moduli stack $\und{\Bun}_G$ of $G$-bundles over $C$.
Applying the construction of \cite{GK2} we consider the Schwartz space 
$$\WW:=\SS(\Bun_G,|\om|^{1/2})$$ 
of half-densities on $\Bun_G=\und{\Bun}_G(K)$. Hecke operators associated with points of $C$ generate a commutative algebra $H(C)$ of endomorphisms of the
$\C$-vector space $\WW$ (see Theorem \ref{HC-comm-thm}).

\begin{definition}
%\begin{enumerate}
%\item 
(1) For a homomorphism $ s:H (C)\to \mC $ we define 
 $$\mcW _s (C) = \{w\in \mcW (C) |\ hw= s(h)w\}.$$ 
\item 

\noindent
(2) The {\it spectrum} $S(C)$ of $H (C) $
is the set of homomorphisms $ s:H (C)\to \mC $ such that $ \mcW _s (C)\neq \{0\} $. 
%\end{enumerate} 
\end{definition}

\begin{conjecture}
$\bigoplus _{s \in S (C)} \mcW _s  (C) = \mcW (C) $.
\end{conjecture} 

\begin{remark} In the archimedean case, the analogs of Hecke operators could also be defined, and the
spectrum $S(C)$ has a conjectural description in terms of $G^\vee$-opers on $C$ (see \cite{EFK1}, \cite{EFK2}).
\end{remark}

%for a curve $C$ 
%over a non-archimedean local field $K$.

In the case when there exists a smooth proper model $C_O$ over $O$ one can look for an
approach to the above conjecture through the analysis of similar objects for $C_O$ and its reductions
over $O/\fm^N$, where $\fm\sub O$ is the maximal ideal.
From now on we fix a smooth proper model $C_O$ over $O$.

Let $\Bun^O_G\sub \Bun_G$ be the open subgroupoid of
$G$-bundles on $C$ for which there exists an extension to $C_O$ (not to be confused with the groupoid $\Bun_G(O)$ of $G$-bundles on $C_O$).
We show that in the case when $[G,G]$ is %simple and 
simply connected, 
%and the characteristic of $O/\fm$ is sufficiently large, 
$\Bun^O_G$ coincides with the subgroupoid of generically trivial $G$-bundles (see Lemma \ref{gen-tr-O-lem}).
% and \ref{principal-sl2-basic-lem}).
We also show that all Schwartz half-densities supported on $\Bun^O_G$ come from
smooth functions on $\Bun_G(O)$ and therefore depend on the reduction modulo some power of the maximal ideal $\fm^N\sub O$
(see Prop.\ \ref{adm-surj-prop}). 

We then consider the reductions $C_N$ of $C$ over $O/\fm^N$. As in \cite{BKP}, we consider
a {\it big local Hecke algebra} (noncommutative for $N>1$), defined in terms of distributions on the local group $G(O/\fm^N(\!(t)\!))$. 
For every point $v\in C_N(O/\fm^N)$, there is an action of this Hecke algebra on the space $\SS(\Bun_G(O/\fm^N))$ of finitely supported functions.
In \cite{BKP} we studied this action in the context of automorphic representations.
In this paper we define a {\it small local Hecke algebra} $\HH^{sm}_{G,O/\fm^N[\![t]\!]}$ (a subalgebra in the big Hecke algebra). Our main result,
generalizing \cite[Thm.\ 2.6]{BKP}, is that this small Hecke algebra is commutative 
provided $[G,G]$ is simply connected, and the characteristic of $O/\fm$ is sufficiently large (see Theorem \ref{small-hecke-thm}).

Due to the compatibility between Hecke operators over $K$ and $O/\fm^N$ (see Sec.\ \ref{hecke-K-O-sec}), one can therefore reduce the construction of $H(C)$-eigenvectors
in $\WW$ to the similar problem for the curves $C_N$ over $O/\fm^N$. The latter problem is much closer to the classical Langlands program for curves over finite fields
(see \cite{BKP} for some partial results).

The key tool we use in this paper is the theory of representations of the ind-pro-group $\G:=G(K(\!(t)\!))$ in pro-vector spaces developed in \cite{GK1}, \cite{GK2}.
More precisely, similarly to \cite[Sec.\ 3]{GK2}, we realize local Hecke algebras as endomorphisms of functors of $G[\![t]\!]$-coinvariants.
% certain induced representations and identify 
Such an interpretation allows to define and analyze actions of these algebras on spaces related to $G$-bundles.
%as natural actions on the spaces of coinvariants of certain representations in the pro-vector spaces (associated with a point in $C$, $C_O$ or $C_N$, respectively).

\begin{remark}
The local Hecke algebras we consider in this paper are associated with the embedding of groups $G(K[\![t]\!])\sub G(K(\!(t)\!))$ (resp., $G(O/\fm^N[\![t]\!])\sub G(O/\fm^N(\!(t)\!))$,
in a version over $O/\fm^N$). One can exchange the roles of the uniformizer in $O$ and of formal variable $t$, and consider instead the pair $G(O(\!(t)\!))\sub G(K(\!(t)\!))$ 
(resp., $G(O[t]/(t^N))\sub G(K[t]/(t^N))$). For example, the work \cite{K-YD} (and most of \cite{BK-Hecke}) deals with the Hecke algebras coming from pairs of the second kind.

In the functional case $K=k(\!(x)\!)$, when $O=k[\![x]\!]$ and $O/\fm^N=k[x]/(x^N)$, we can identify the pair of the first kind associated
to $K$ with the pair of the second kind associated to $K'=k(\!(t)\!)$ (with $O'=k[\![t]\!]$):
$$O/\fm^N(\!(t)\!)=K'[x]/(x^N), \ \ O/\fm^N[\![t]\!]=O'[x]/(x^N).$$
%Thus, the corresponding Hecke algebras are isomorphic.

We conjecture that in the case when the characteristic of $K$ is zero, small Hecke algebras associated with pairs $G(O[t]/(t^N))\sub G(K[t]/(t^N))$ are still commutative
but we do not address this problem in this paper (the case $N=2$ is proved in \cite{K-YD}).
\end{remark}

The paper is organized as follows. In Sec.\ \ref{geom-sec} we collect some general results about Schwartz spaces of varieties and stacks over $K$ and over $O$.
In Sec.\ \ref{hecke-sec} we discuss local Hecke algebras related to the loop group of $G$ over $K$ and over related rings.
First, we discuss the Hecke algebra over $K$, denoted by $\HH(\hat{\G},\H)_c$ (where $c$ is the level), in the framework of $G(K(\!(t)\!))$-representations in pro-vector spaces.
In Sec.\ \ref{crit-level-sec} we recall the commutative subalgebra in the Hecke algebra
$\HH(\G,\H)_{crit}$ at the critical level constructed in \cite{BK-Hecke}. 
In Sec.\ \ref{fin-hecke-sec} we consider versions of this algebra 
over $O$ and $O/\fm^N$ denoted by $\HH(\G_O,\H_O)$ and $\HH(\G_N,\H_N)$, respectively.
%In Sec.\ \ref{hecke-def-sec} and \ref{coinv-sec} we define and study these algebras in the context 
%In particular, we show that they act on the spaces of $G[\![t]\!]$-coinvariants.
We show that the algebra $\HH(\G_N,\H_N)$ is isomorphic to 
the algebra of $G(O/\fm^N[\![t]\!])$-biinvariant distributions  with compact support on $G(O/\fm^N(\!(t)\!))$.
In Sec.\ \ref{hecke-hom-sec} we define algebra homomorphisms between the Hecke algebras,
$$\HH(\hat{\G},\H)_c\to \HH(G_O,\H_O)\to \HH(\G_N,\H_N),$$
and show their compatibility with the action on $G[\![t]\!]$-coinvariants. 
In Sec.\ \ref{coord-free-sec} we discuss coordinate-free versions of our local Hecke algebras
and define the small Hecke algebra $\HH^{sm}_{G,O/\fm^N[\![t]\!]}$.

In Sec.\ \ref{G-bun-sec} we present auxiliary results on $G$-bundles.
% over finite fields and over rings $O$ and $O/\fm^N$.
%In Sec.\ \ref{prin-sl2-sec} we recall facts about the principal $\ssl_2$-subalgebra.
In Sec.\ \ref{nice-G-sec} 
%and \ref{nice-bun-constr-sec} 
we introduce 
%and study 
the notion of a {\it nice $G$-bundle} on a curve with a fixed point $p$: these are $G$-bundles for which we can control automorphisms over $C-p$
with a fixed order $n$ of pole at $p$. In the appendix we prove existence of curves over arbitrary fields (with a mild restriction on the characteristic)
with such nice $G$-bundles for arbitrarily large $n$.
In Sec.\ \ref{gen-triv-O-sec} we relate generic triviality of $G$-bundles over a curve $C$ over $K$
with existence of an extension to $C_O$ (a smooth proper model over $O$). 

In Sec.\ \ref{hecke-bunG-sec} we study Hecke operators on the Schwartz space $\WW$ of half-densities on $\Bun_G$ for a curve $C$ over $K$, and use global arguments
to prove our main result Theorem \ref{small-hecke-thm}.
In Sec.\ \ref{hecke-bunG-K-sec} we relate the operators on $\WW$ arising from the representation theory of $G(K(\!(t)\!))$ in pro-vector spaces with those given by the Hecke
correspondences. We also prove commutation of Hecke operators in $H(C)$ corresponding to different $K$-points of a curve $C$ (see Theorem \ref{HC-comm-thm}).
In Sec.\ \ref{hecke-K-O-sec} we discuss Hecke operators over $O$ and $O/\fm^N$ arising from elements of $H(C)$. Note that we can explicitly describe these Hecke operators
only for minuscule coweights $\la$. Finally in Sec.\ \ref{comm-hecke-sec} we give a global proof of Theorem \ref{small-hecke-thm} on commutativity of the small local Hecke algebra over $O/\fm^N$,
using Theorem \ref{HC-comm-thm} and nice $G$-bundles.

\medskip

\noindent
{\it Notation}:

\noindent
$K$ is a non-archimedian local field, $O \subset K$ its ring of integers, $\fm \subset O$ the maximal ideal, $k= O/\fm$.

\noindent
For a connected split reductive group $G$, $\bL$ denotes the lattice of coweights of $G$, $\bL ^+\subset \bL$ the semigroup of dominant coweights,
% and $R= \mC [\bL ^+]$,
 $\Gr_G=G(\!(t)\!)/G[\![t]\!]$ the affine Grassmannian.

\section{Geometry over $K$ and $O$}\label{geom-sec}

\subsection{Varieties over $K$ and $O$}\label{varieties-sec}

Let us recall some standard general constructions. 

\begin{definition}\label{variety} Let $X$ be a scheme of finite type over $K$.

\begin{enumerate}
\item $\mC (X(K))$ is the space of locally constant 
$\mC$-valued functions on $X(K)$.
% \footnote{If $F$ is archimedian then  $\mcS (X)$ is a Fréchet space and if $F$ is non-archimedian  $\mcS (X)$ consists of 
%locally constant functions.} 

\item $\mcS (X(K))\subset \mC (X(K))$ is the subspace of 
compactly supported functions.
\item For a line bundle
 $\mcL$ on $X$ and a character $c:K^*\to \C^*$ (i.e., a continuous homomorphism), we denote by $\mcL_c$
 the associated complex local system on $X(K)$ (defined using push-out with respect to $c$).
In particular, for a complex number $z$, we denote by $|\mcL|^z:=\mcL_{|\cdot |^z}$ the local system associated with
 the homomorphism $|\cdot|^z:K^*\to \C^*$.

\item For any $\C$-local system $L$ on $X(K)$, 
since transition functions 
%$\mcL_c$ on $X(K)$
are locally constant, we can define the space  $\mC(X(K),L)$
% $\mC (X(K),\mcL_c) $ 
of locally constant sections 
%of the complex line bundle $\mcL_c$ on $X(K)$ 
and  the subspace $\mcS (X(K),L)\subset \mC (X(K),L) $ of compactly supported sections. 
In particular, for a line bundle $\mcL$ on $X$, and a character $c:K^*\to \C^*$, this gives a complex vector space $\mcS(X(K),\mcL_c)$.
\footnote{See Section $2.1$ of \cite{BK}.}
\end{enumerate} \end{definition} 

Now consider the case when $X$ is scheme of finite type over $O$, with the associated scheme $X_K$ over $K$.  
In this case, the space $X(O)$ is an open subset of $X(K)=X_K(K)$.
Furthermore, for any line bundle $\mcL$ on $X$, the
complex line bundle $ |\mcL|^z$ is canonically trivialized over $X(O)$.
Hence, the space $\mcS(X(O),|\mcL|^z)$ of compactly supported locally constant sections of $|\mcL|^z$ can be identified with the space $\mcS(X(O))$ of
locally constant functions on $X(O)$.

Thus, for any line bundle $\mcL$ on $X$, and every $n\ge 1$, we have a natural map
\begin{equation}\label{r-n-X-OK-map}
E_{n,|\mcL|^z}:\SS(X(O/\fm^n))\to \SS(X(O))\hra \SS(X(K),|\mcL|^z),
\end{equation}
which sends a finitely supported function on $X(O/\fm^n)$ to the corresponding locally constant compactly supported function on $X(O)$, 
which is then viewed as a locally constant section of $|\mcL|^z$ on $X(O)$\footnote{The map $E_{n,|\mcL|^z}$ can be viewed as an analog of the parabolic Eisenstein series.}. Note that we have
$$\SS(X(O))=\varinjlim \SS(X(O/\fm^n)).$$

Assume now that $X$ is smooth over $K$, and let $\bo _X$ be the canonical line bundle. Then $\mC (X(K), |\bo _X|)$ is the space of 
smooth complex valued measures on $X(K)$. A nowhere vanishing section $\eta\in \Ga (X,\bo _X)$  defines a smooth measure  $|\eta|$ on $X(K)$.
For $\mu\in \SS(X(K),|\bo_X|)$, the integral $\int _{X(K)} \mu$ is well defined. Similarly, for a smooth morphism $f:X\to Y$, a line bundle $\mcL$ on $Y$, and a character 
$c:K^*\to \C^*$,
we have a well-defined
push-forward map
\begin{equation}\label{f!-var-eq}
f_!:\SS(X(K),(f^*\mcL)_c\ot|\om_{X/Y}|)\to \SS(Y(K),\mcL_c).
\end{equation}

More generally, if $X$ is Gorenstein and has rational singularities then there is a well defined integration map
$$\SS(X(K),|\bo_X|)\to \C: \mu \mapsto \int_{X^{sm}(K)}\mu,$$
where the integral is absolutely convergent (see \cite[Sec.\ 3.4]{AA}). 
Now consider a proper Gorenstein morphism $f:X\to Y$ of integral schemes of finite type over $K$. Assume that there exists a proper birational map $\pi:\wt{X}\to X$ with the following
properties:
\begin{itemize}
\item The map $\pi f:\wt{X}\to Y$ is smooth;
\item there is an isomorphism $\pi^*\om_{X/Y}\simeq \om_{\wt{X}/Y}(-D)$ for some effective Cartier divisor $D$ on $\wt{X}$;
\item locally in smooth topology over $Y$, there is an isomorphism $(\wt{X},D)\simeq (Y\times F, Y\times E)$, where $F$ is smooth and $E\sub F$ is a Cartier divisor.
\end{itemize}
Then for any line bundle $\LL$ on $Y$ and a character $c:K^*\to \C^*$, we can define the push-forward map
$$f_!:\SS(X(K),(f^*\mcL)_c\ot|\om_{X/Y}|)\to \SS(Y(K),\mcL_c)$$ 
as the composition
$$\SS(X(K),(f^*\mcL)_c\ot|\om_{X/Y}|)\to \SS(\wt{X}(K),(\pi^*f^*\LL)_c\ot |\om_{\wt{X}/Y}(-D)|)\to \SS(Y(K),\mcL_c),$$
where the second arrow is given by integration in fibers of $\wt{X}\to Y$. The fact that the resulting section of $\mcL_c$ is locally constant follows from
our assumption on the local structure of $(\wt{X},D)$. The independence on the choice of the relative resolution $\wt{X}$ follows from the fact that the values of $f_!\varphi$
are given by convergent integrals over the smooth loci in the fibers of $f$. 

The following result goes back to Weil \cite{Weil}.

\begin{lemma}\label{Weil-lem}
Let $X$ be a smooth variety over $O$. Then for any function $\phi\in \SS(X(O/\fm^n))$ one has
$$\int_{X(K)} E_{n,|\bo_X|}(\phi)=\frac{1}{|O/\fm^n|^{\dim X}}\sum_{x\in X(O/\fm^n)} \phi(x).$$
\end{lemma}
 
\begin{proof} 
This is essentially in \cite[Sec.\ 2.2]{Weil}. It is enough to consider the case when $\phi$ is the delta-function of a point $x_0\in X(O/\fm^n)$.
Then we have the corresponding compact open neighborhood
$$V(x_0,n):=\{x\in X(O) \ | \ x\equiv x_0 \mod \fm^n\},$$
and our statement reduces to the statement that
$$\mu(V_{x_0,n})=\frac{1}{|O/\fm^n|^{\dim X}},$$
where $\mu$ is the measure on $X(O)$ obtained from the canonical trivialization of $|\om_X|$ over $X(O)$.
We can choose near $x_0$ an \'etale map $f:X\to \A^d_O$ sending $x_0$ to the origin, where $d=\dim X$. Then by Hensel's lemma, $f$ maps
$V_{x_0,n}$ bijectively onto $\A^d(\fm^n)\sub \A^d(O)$, and the result follows. 
\end{proof}

\subsection{Stacks over $K$ and $O$}\label{KO-stacks-sec}
 
We refer to \cite{GK2} for details and proofs of the results below on stacks over $K$. 
 
 \begin{definition}\label{stack}

\begin{enumerate} 
\item For a smooth stack $\mcY$ over $K$ we denote by $\bo _\mcY$ the canonical bundle on $\mcY$.

\item For a smooth representable map $q: X\to \mcY$ 
we denote by $\bo _q=\bo_{X/\mcY}$
% =\bo  _X \otimes q^\ast (\bo _\mcY ^{-1})$ 
the relative canonical bundle.

\item We say that a smooth stack  $\mcY$ of finite type over $K$ is {\it admissible} if it can be represented as a quotient $[X/H]$ where $X$ is a smooth variety over $K$ and $H$ is a linear algebraic group acting on $X$.

\item A stack $\mcY$ over $K$ is called {\it admissible}\footnote{This notion is slightly stronger than that of \cite{GK2}.} if it is the union $\cup_{i\ge 1}\mcY_i$ of an increasing sequence $\mcY_1\sub \mcY_2\sub\ldots$ of open admissible substacks of finite type over $K$. 
\end{enumerate} \end{definition}

\begin{claim} Let $\mcY$ be an admissible stack of finite type over $K$. 
\begin{enumerate}

\item A line bundle $\mcL$ on $\mcY=[X/H]$ is canonically represented by an $H$-equivariant line bundle  $\tilde \mcL$ on $X$.
\item The canonical bundle $\bo _\mcY$ on $\mcY=[X/H]$ is represented by $\bo _X\otimes \bL ^d(\fh)$ where $\fh$ is the Lie algebra of $H$ and $d=dim(\fh)$ (we use the adjoint action of
$H$ in defining the $H$-equivariant structure).
\item Any admissible stack $\mcY$ of finite type can be represented as a quotient $[X/\GL _N]$.
\item The topological groupoid $X(K)/\GL _N(K)$ does not depend on a choice of a presentation $\mcY=[X/\GL_N]$. 
\end{enumerate}
\end{claim}

\begin{proof} Most of the assertions are straightforward. Replacing a presentation $[X/H]$ by one with $H=\GL_N$ is achieved using an embedding $H\to \GL_N$
(see \cite[Lem.\ 6.7]{GK2}). The last assertion follows from Hilbert's theorem 90, which implies triviality of $\GL_N$-torsors over $K$.
\end{proof}

\begin{definition}\label{stackF}

\begin{enumerate}
\item For a line bundle $\mcL$ on an admissible stack of finite type over $K$, $\mcY = [X/\GL_N]$, and a character $c:K^*\to \C^*$, 
we denote by   $\mcS (\mcY(K), \mcL_c)$ (where $z\in \C$)
the space of coinvariants 
$$\mcS (\mcY(K), \mcL_c):= \mcS (X(K), \ti \mcL_c \otimes |\bo _{X/\mcY}|) _{\GL_N(K)},$$ 
where 
 $\ti \mcL$ is   the  $\GL_N$-equivariant line bundle  on $X$ which represents $\mcL$.

\item If  $\mcL$ is   a line bundle on an admissible 
stack $\mcY $ which is the union of increasing open 
substacks $\mcY _i$ of finite type we define
  $\mcS (\mcY(K), \mcL_c)= \varinjlim  
\mcS (\mcY _i(K), (\mcL|_{\mcY _i})_c)$. 

%\item For an open subgroupoid $U\sub \mcY(K)$ we set $\mcS_U(\mcY(K), |\mcL|)=\varinjlim \mcS_{U}(\mcY_i(K), |\mcL|_{\mcY _i}|)$,
%where $\mcS_{U}(\mcY_i(K), |\mcL|_{\mcY _i}|)\sub \mcS (\mcY _i(K), |\mcL|_{\mcY _i}|)$ consists of classes that can be represented by sections
%supported on the preimage of $U$.
\end{enumerate} \end{definition}

\begin{claim} The space $\mcS (\mcY(K), \mcL_c)$ for an admissible stack of finite type does not depend on a presentation of $\mcY$ as a quotient $[X/\GL_N]$. 
The space $\mcS(\mcY(K),\mcL_c)$ for an admissible stack does not depend on a choice of open admissible substacks of finite type $\mcY_i$ such that $\mcY=\cup \mcY_i$.
\end{claim} 

%\subsubsection{Stacks over $\mcO$}

We define {\it admissible stacks over $O$} in the same way as above, replacing $K$ by $O$.
 
 \begin{lemma}\label{eq-triv-O-lem}
 Let $X$ be a scheme over $O$ equipped with an action of an algebraic group $H$ (defined over $O$), and let $\MM$ be
an $H$-equivariant line bundle on $X$. Then the natural trivialization of $|\MM|$ over $X(O)$ is compatible with the $H(O)$-action.
\end{lemma}

\begin{proof}
Let $s_i:\OO_{U_i}\rTo{\sim} \MM|_{U_i}$ be trivializations of $\MM$ over an open covering $(U_i)$ of $X$. Consider the covering of $H\times X$ by
the open subsets
$$V_{ij}=\{(h,x)\in X\times H \ | x\in U_i, hx\in U_j\}.$$
Over each $V_{ij}$ we have an equality
$$s_j(hx)=f_{ij}(h,x)\cdot s_i(x),$$
for some $f_{ij}\in \OO^*(V_{ij})$ (we use the equivariant structure to view both sides as sections of $p_2^*\MM$).
Hence, taking absolute values we get the equality
$$|s_j(hx)|=|s_i(x)|$$
on $V_{ij}(O)$, which implies our assertion.
\end{proof}
 
\begin{definition}\label{OL} Let $\mcY =[X/\GL_N]$  be an admissible stack of finite type defined over $O$. 
%$\mcL$ a line bundle on $\mcY$, represented by a $\GL_N$-equivariant bundle $\ti \mcL$ on $X$.

\begin{enumerate}

 \item  We set $\mcS(\mcY(O)):=\mcS(X(O))_{\GL_N(O)}$.
 %$\mcS(\mcY(O), |\mcL|^z)= \mcS (X(O), |\ti \mcL|^z \otimes |\bo _{X/\mcY}|) _{\GL_N(O)}=\mcS(X(O))_{\GL_N(O)}$.

\item  If $\mcY $ is the union of increasing open 
substacks $\mcY _i$ of finite type we define
  $\mcS (\mcY(O))= \varinjlim  
\mcS (\mcY _i(O))$. 
%\item  $r^ \mcY _\ast : \mcS ^O (\mcY, |\mcL|) \to \mcS (\mcY, |\mcL|) $ is the pushforward. \footnote{ $r^ \mcY _\ast $ is well defined since $ r^ \mcY $ is a local isomorphism.}

\end{enumerate} \end{definition}

\begin{remark} By Lemma \ref{eq-triv-O-lem}, we have $\mcS(X(O))_{\GL_N(O)}=\mcS(X(O),|\om_{X/\YY}|)_{\GL_N(O)}$.
The space $\mcS (\mcY(O))$ does not depend on a representation of $\mcY$ as a quotient $[X/\GL_N]$.
% \begin{enumerate} If  $F$ is a local non-archimedian  field, then 
%\item If an admissible smooth stack  $\mcY =X/GL_ n $ of finite type and a line bundle $\mcL$ on $\mcY$
%are defined over $\mcO$ then the  line bundle $(r^\mcY) ^\ast (|\mcL|)$ on $\mcY (\mcO)$ is canonically trivialized.
\end{remark}

If $\YY=[X/\GL_N]$ is an admissible stack of finite type over $O$, $\LL$ a line bundle on $\YY$, represented by a $\GL_N$-equivariant line bundle
$\wt{\LL}$ on $X$, then by Lemma \ref{eq-triv-O-lem}, the identification 
$$\SS(X(O),|\wt{\LL}|^z\ot |\om_{X/\YY}|)\simeq \SS(X(O))$$
%\simeq \SS(X(O),|\om_{X/\YY}|)$$ 
is compatible with the $\GL_N(O)$-actions. Hence, we get a natural map
$$i^{\mcY,|\LL|^z}_*:\SS(\YY(O))=\SS(X(O))_{\GL_N(O)}\to \SS(X(K),|\wt{\LL}|^z\ot |\om_{X/\YY}|)_{\GL_N(K)}=\SS(\YY_K(K),|\LL|^z),$$
where $\YY_K$ is corresponding stack $[X_K/\GL_N]$ over $K$. 
 
More generally, if $\mcY$ is an admissible stack over $O$, $\mcY_K$ the corresponding admissible stack over $K$,
then for any line bundle $\mcL$ on $\mcY$, we get a natural map
\begin{equation}\label{r*-stacks-eq}
i^{\mcY,|\mcL|^z}_*:\mcS(\mcY(O))\to \mcS(\mcY_K(K),|\mcL|^z)
%\mcS_{r^{\YY}(\YY(O))}(\mcY_K(K),|\mcL|)
\end{equation}
defined as the limit of similar maps for $\mcY_i=[X/\GL_{N}]$,
%$$r^{\mcY_i}_*:\mcS(X(O), |\ti \mcL|^z \otimes |\bo _{X/\mcY}|) _{\GL_N(O)}\to \mcS(X(K), |\ti \mcL|^z \otimes |\bo _{X/\mcY}|)_{\GL_N(K)}.$$

\begin{definition}\label{O}

%\begin{enumerate}
%\item 
For a stack $\mcY$ over $O$ we denote by $i^\mcY $ the natural map $\mcY (O)\to \mcY (K)$ of topological groupoids.

%\item We say that an $O$-stack $\mcY $ is {\it almost proper} if the map $r^\mcY $ is surjective. \end{enumerate} 
\end{definition} 

%\begin{remark} If $\mcY$ is proper over $O$ then it is almost proper.
%But as we will see, there exist almost proper stacks which are not of finite type over $O$ (and so, not proper).
%\end{remark} 

\begin{prop}\label{adm-surj-prop} 
%\begin{enumerate}
%\item The map $r^\mcY $ is a local isomorphism. 
%\item The map $r^\mcY $ is onto is $\mcY$ is almost proper. 
%\end{enumerate}
%Assume $\mcY$ is an admissible almost proper stack over $O$.
The image of the map $i^{\mcY,|\mcL|^z}_*$ (see \eqref{r*-stacks-eq}) consists of densities supported on the open subgroupoid 
$i^\mcY(\mcY (O))$.
%is surjective.
\end{prop}

 \begin{proof}
 Let $\mcY_i=[X_i/\GL_{N_i}]$. For each $j\ge i$, let us consider the cartesian square
 \begin{diagram}
 X_{i,j}&\rTo{\wt{f}_{i,j}}& X_j\\
 \dTo{p_{i,j}}&&\dTo{}\\
 X_i&\rTo{f_{i,j}}& \mcY_j
 \end{diagram}
 where $f_{i,j}$ is the composition of the projection $X_i\to \mcY_i$ with the open embedding $\mcY_i\to \mcY_j$.
 Note that $X_{i,j}$ is a $GL_N$-torsor over $X_i$.
 
\noindent
{\bf Step 1}.
We claim that for every point $x\in X_i(K)$ in the preimage of $r^\YY(\YY(\OO))\sub \YY(K)$, there exists $j\ge i$ and a point $\wt{x}\in X_{i,j}(K)$ over $x$ such that $\wt{f}_{i,j}(\wt{x})\in X_j(O)$.
Indeed, by definition
%the assumption that $\mcY$ is almost proper implies that 
there exists $j\ge i$ such that the image of $x$ in $\mcY_j(K)$ comes from a point in $X_j(O)$.
Thus, if $\wt{x}\in X_{i,j}(K)$ is any point over $x$ then the $\GL_{N_j}(K)$-orbit of $\wt{f}_{i,j}(x')$ contains a point in $X_j(O)$.
Hence, changing $\wt{x}$ to $g\wt{x}$ for an appropriate $g\in \GL_{N_j}(K)$, we can achieve that $\wt{f}_{i,j}(\wt{x})\in X_j(O)$.

\noindent
{\bf Step 2}.
Next, we claim that for any compact open subset $A$ in the preimage of $r^\YY(\YY(\OO))$ in $X_i(K)$, there exists $j\ge i$ and a compact open $\wt{A}\sub X_{i,j}(K)$ such that $p_{i,j}(\wt{A})=A$
and $\wt{f}_{i,j}(\wt{A})\sub X_j(O)$.
Indeed, by Step 1, for any point $x\in A$ we can find $j\ge i$ and a point $\wt{x}\in X_{i,j}(K)$ over $x$ such that $\wt{f}_{i,j}(\wt{x})\in X_j(O)$.
Furthermore, there exists an open compact neighborhood $U$ of $\wt{x}$ in $X_{i,j}(K)$ such that $\wt{f}_{i,j}(U)\sub X_j(O)$. By compactness of $A$,
finitely many open sets of the form $p_{i,j}(U)$ cover $A$, which implies our claim.

\noindent
{\bf Step 3}. 
Given an element $\phi\in \mcS(X_i(K),|\mcL \otimes \bo_{X_i/\mcY}|)$, supported on a compact open subset $A\sub X_i(K)$ contained in the preimage of $i^\YY(\YY(\OO))$, we can find $j\ge i$ and
$\wt{A}\sub X_{i,j}(K)$ as in Step 2. Since the projection $\wt{A}\to A$ is smooth and surjective, as in the proof of \cite[Prop.\ 6.2]{GK2}, we can find 
$\wt{\phi}\in \mcS(X_{i,j}, |\mcL \otimes \bo_{X_{i,j}/\mcY}|)$ supported on $\wt{A}$, such that $p_{i,j,!}(\wt{\phi})=\phi$ 
(note that $\bo_{X_{i,j}/\mcY}\simeq \bo_{X_i/\mcY}\ot \bo_{X_{i,j}/X_i}$). Now $\wt{f}_{i,j,!}(\wt{\phi})\in \mcS(X_j,|\mcL \otimes \bo_{X_j/\mcY}|)$
maps to the same element in $\mcS(\mcY(K),|\mcL|)$ as $\phi$. Since $\wt{f}_{i,j,!}(\wt{\phi})$ is supported on $\wt{f}_{i,j}(\wt{A})\sub X_j(O)$,
our assertion follows.
\end{proof}

% Let $F$ be a local field, $R\sub F$ its ring of integers, $\fm\sub R$ the maximal ideal.
  
%For a nice stack $\XX$ defined over $F$ ($\XX$ is locally of the form $X/\GL_n$, where $X$ is smooth over $F$) 
%and a line bundle $\LL$ over $\XX$, we define the space $\SS(\XX,|\LL|)$ as in \cite{GK}.

%Let $\YY=[X/\GL_N]$ be an admissible stack of finite type over $O$, $\LL$ a line bundle on $\YY$, $\ti{\LL}$ corresponding line bundle on $X$. Then the
%natural trivialization of $|\ti{\LL}|$ on $X(O)$ is compatible with the action of $\GL_n(O)$. Hence, we get an identification
%$\SS(\YY(O),|\mcL|)=\SS(\YY(O),|\mcO|)$. We also have similar identifications for admissible stacks not of finite type.

For a smooth representable morphism $f:\XX\to \YY$ of admissible stacks over $K$, a line bundle $\LL$ over $\YY$, and a character $c:K^*\to \C^*$, we define
the push-forward maps
$$f_!:\SS(\XX(K),(f^*\LL)_c\ot |\om_{\XX/\YY}|)\to \SS(\YY(K),\LL_c)$$
using the corresponding maps \eqref{f!-var-eq} for varieties (see also \cite[Sec.\ 6.8]{GK2}).

Let $\YY$ be an admissible stack of finite type over $O$, $\LL$ a line bundle over $\YY$.
For each $n>0$, we have the groupoid $\YY(O/\fm^n)$ of $O/\fm^n$-points of $\YY$, and the corresponding space $\SS(\YY(O/\fm^n))$ of finitely
supported functions.
Furthermore, 
%any line bundle $\LL$ on $\YY$ 
we have a natural map
$$E_{n,|\LL|^z}:\SS(\YY(O/\fm^n))\to \SS(\YY(O)) \rTo{i^{\YY}_*} \SS(\YY(K),|\LL|^z)$$
%\SS(\YY(O),|\LL|)$$
induced by the map \eqref{r-n-X-OK-map}.
%By passing to limits we can also define this map for $\YY$ not of finite type.
As in the case of varieties, we see that $i^{\YY}_*(\SS(\YY(O)))$ is the union of the subspaces $E_{n,|\LL|^z}(\SS(\YY(O/\fm^n)))$.
%constructed as follows.
%Namely, we have a natural trivialization of $|\LL|$ over $\XX(R)$.
%It is enough to define $r_{n,\LL}$ on the delta-function of a point $x_0\in \XX(R/\fm^n)$.
%Locally near this point we have a presentation $\XX=X/G$, where $G=\GL_n$.
%We have a natural trivialization of $|\LL|$ over $X(R)$ and therefore a natural map
%$$\SS(X(R))\to \SS(X,|\LL|)\to \SS(\XX,|\LL|).$$
%Let $x'_0\in X(R/\fm^n)$ be a point mapping to $x_0$. Then we have the corresponding compact open neighborhood
%$$V(x'_0,n):=\{x\in X(R) \ | \ x\equiv x'_0 \mod \fm^n\},$$
%and we define the image of $\de_{x_0}$ as the image of the characteristic function $\de_{V(x'_0,n)}$ under the above map.
  
\begin{prop}\label{int-point-integration-prop}
Let $f:\XX\to \YY$ be a representable smooth morphism of admissible stacks over $O$, and let $f_{O/\fm^n}:\XX(O/\fm^n)\to \YY(O/\fm^n)$ denote the corresponding
functor between groupoids of $O/\fm^n$-points. Then for any line bundle $\LL$ over $\YY$, and any $\phi\in \SS(\XX(O/\fm^n))$, we have
$$f_!E_{n,|f^*\LL|^z\ot|\om_{\XX/\YY}|}(\phi)=\frac{1}{|O/\fm^n|^{\dim \XX-\dim \YY}}\cdot E_{n,|\LL|^z}f_{O/\fm^n,*}\phi.$$
Here we use the natural push-forward map $f_{O/\fm^n,*}$ for finitely supported functions on groupoids. 
%(see Sec.\ \ref{push-forward-sec}).
%integration map
%$$f_!:\SS(\XX,|f^*\LL\ot\om_{\XX/\YY}|)\to \SS(\YY,|\LL|)$$
%defined as in \cite{GK}.
\end{prop}
  
\begin{proof}
%[Proof of Proposition \ref{int-point-integration-prop}] 
We can assume that $\XX=X/G$, $\YY=Y/G$, where $G=\GL_N$, 
%$\LL$ is trivial. 
and the morphism $\XX\to \YY$ 
%is representable, we have $\XX=X/G$, where $X=Y\times_{\YY} \XX$, and we have a 
is induced by a smooth $G$-equivariant map $f:X\to Y$. The commutative diagrams
\begin{diagram}
\SS(X(O),|\om_{X/Y}|)&\rTo{f_!}& \SS(Y(O))\\
\dTo{}&&\dTo{}\\
\SS(\XX(K),|f^*\LL|^z\ot|\om_{\XX/\YY}|)&\rTo{f_!}&\SS(\YY(K),|\LL|^z)
\end{diagram}
\begin{diagram}
\SS(X(O/\fm^n))&\rTo{f_{O/\fm^n,*}}& \SS(Y(O/\fm^n))\\
\dTo{}&&\dTo{}\\
\SS(\XX(O/\fm^n))&\rTo{f_{O/\fm^n,*}}&\SS(\YY(O/\fm^n))
\end{diagram}
show that it is enough to prove our assertion with $\XX\to \YY$ replaced by $X\to Y$. We can also assume $\LL$ to be trivial.

Let us start with $\phi\in \SS(X(O/\fm^n))$ then $f_!E_{n,|\om_{X/Y}|}(\phi)$ is supported on $Y(O)$ and we need to compute its value at $y\in Y(O)$.
Let $X_y\sub X$ be the fiber over $y$ (which is a smooth variety over $O$). Then
$$f_!E_{n,|\om_{X/Y}|}(\phi)(y)=\int E_{n,|\om_{X_y}|}(\phi|_{X_y(O/\fm^n)}).$$
%Then it is enough to check that the formula holds for a function on $X(O/\fm^n)$ of the form $\phi=\de_{x_0}$ for some $x_0\in X(O/\fm_n)$ 
%We have $f_{O/\fm^n,*}\de_{x_0}=\de_{y_0}$, where $y_0\in Y(O/\fm^n)$ is the image of $x_0$.
%Hence, $r_{n,\OO}f_*\de_{x_0}$ is the characteristic function of $V_{y_0,n}$.
%On the other hand, $r_{n,\om_{X/Y}}$ corresponds to the characteristic function of $V_{x_0,n}$ for some $x_0\in X(R/\fm^n)$. The image of $V_{x_0,n}$ under the map $X(R)\to Y(R)$
%is exactly the set $V_{y_0,n}$. Furthermore, for $y\in V_{y_0,n}$, the value of $f_!r_{n,\om_{X/Y}}$ is exactly the canonical measure of $V_{x_0,n}\cap f^{-1}(y)$.
Applying Lemma \ref{Weil-lem}, we can rewrite the integral in terms of the summation of values of $\phi$ over $X_y(O/\fm^n)$, which gives the result.
\end{proof}

\section{Local Hecke algebras}\label{hecke-sec}

In this section we discuss local Hecke algebras associated with a group $G$ and a $2$-dimensional local field $K(\!(t)\!)$, as
well as the related algebras over $O$ and $O/\fm^N$.  

First, in Sec.\ \ref{hecke-def-sec}, \ref{coinv-sec} and \ref{crit-level-sec} we discuss 
Hecke algebras associated with $K(\!(t)\!)$ in the context of representations of $G(K(\!(t)\!))$ in pro-vector spaces.
Then in Sec.\ \ref{fin-hecke-sec}, we consider analogous algebras over $O$ and $O/\fm^N$. In Sec.\ \ref{hecke-hom-sec}
we construct homomorphisms connecting the three kinds of algebras. Finally, in Sec.\ \ref{coord-free-sec}, we use
these homomorphisms to define the {\it small Hecke algebra} over $O/\fm^N$, and we formulate our main theorem about
its commutativity (Theorem \ref{small-hecke-thm}).

Note that all constructions of this section are purely local and do not involve $G$-bundles on curves.

\subsection{Representations on pro-vector spaces and the local Hecke algebra over $K$}\label{hecke-def-sec}

Following \cite{GK1}, \cite{GK2}, 
we work in the framework of groups in $\SSet=\Ind(\Pro(\Ind(\Pro(Set_0))))$ (where $Set_0$ is the category of finite sets) and their representations in pro-vector spaces,
i.e., objects of $\VVect=\Pro(Vect)=\Pro(\Ind(Vect_0))$ (where $Vect_0$ is the category of finite-dimensional $\C$-vector spaces). 
Following \cite{GK1}, we write pro-objects of a category $\CC$ as $``\varprojlim" C_j$, where $C_j\in \CC$, to distinguish them from the projective limits taken in the category $\CC$.

As before, $G$ is a split connected reductive group over $\Z$.

Let $K$ be a local non-archimedean field. 
%$O\sub K$ the ring of integers, $\fm\sub O$ the maximal ideal.
Then there is a natural group in $\SSet$ which we denote as $\G=G(K(\!(t)\!))$. Namely, 
$K(\!(t)\!)$ can be viewed as the ind-object $(t^{-n}K[\![t]\!])$, while $K[\![t]\!]$ is the pro-object $(K[t]/(t^n))$,
where each $K[t]/t^n$ is an object of $\bSet=\Ind(\Pro(Set_0))$.

Note that $\G$ comes from a group-ind-scheme $G(\!(t)\!)$ defined over $\Z$ (see \cite[Sec.\ 2.12]{GK1}).
We also fix a central extension $\hat{G}$ of $G(\!(t)\!)$ by $\G_m$ in the category of group-ind-schemes, equipped with
a splitting over $G[\![t]\!]$, and denote by 
$\hat{\G}$ the corresponding extension of $\G$ by $K^*$ (see \cite[Sec.\ 2.14]{GK1}). 
%Note that $\hat{G}$ is also a group in $\SSet$.

We consider the subgroup $\H:=G(K[\![t]\!])$ of $\G$ (also in $\SSet$), and the natural congruence subgroups
$\G^i\sub\H$, $i\ge 0$ (see \cite[Sec.\ 2.12]{GK1}).
Note that $\H$ is a group object of $\Pro(\bSet)$, in fact, it is given
by the projective system $(G(K[t]/(t^n)))$ of groups in $\bSet$. Note that $\H$ is a {\it thick } subgroup of $\G$ in the sense of \cite[2.12]{GK1},
and we have a splitting of the central extension $\hat{\G}$ over $\H$. 

Let $\Rep_c(\hat{\G})$ denote the category of $\hat{\G}$-representations in pro-vector spaces at level $c$, where $c:K^*\to \C^*$ is a character (see \cite[2.14]{GK1}). 
The quotient $\G/\H\in \Ind(\bSet)$ is {\it ind-compact}, i.e., a direct system of compact
objects of $\bSet$ (see \cite[3.4]{GK1}). In this case there is a natural induction functor 
$$i_{\H}^{\hat{\G}}:\Rep(\H)\to \Rep_c(\hat{\G}),$$ 
forming an adjoint pair $(r_{\H}^{\hat{\G}},i_{\H}^{\hat{\G}})$ 
with the corresponding restriction functors (see \cite[Prop.\ 3.5]{GK1}).
%This group is not {\it quasi pro-unipotent} in terminology of \cite[Sec.\ 2.6]{GK1}

We are interested in the induced representation $i_{\H}^{\hat{\G}}(\C)$,
where $\C$ is the trivial representation of $\H$. 
%We also have similar induced representations $i_{\H_O}^{\G_O}(\C)$ and $i_{\H_N}^{\G_N}(\C)$ of $\G_O$ and $\G_N$,
%respectively.

\begin{definition}\label{hecke-alg-K-def}
We define the local Hecke algebra at the level $c$ by
$$\HH(\hat{\G},\H)_c:=\End_{\Rep_c(\hat{\G})}(i_\H^{\hat{\G}}(\C))^{op}$$
\end{definition}

Using the adjunction of the restriction and the induction functors, we can rewrite this definition as
$$\HH(\hat{\G},\H)_c=\Hom_{\H}(i_\H^{\hat{\G}}(\C),\C).$$

More concretely, let us consider the affine Grassmannian $\Gr_G=G(\!(t)\!)/G[\![t]\!]$ (defined over $\Z$), and represent it as the union of finite-dimensional proper schemes $\ov{\Gr}_\la$
(where $\la$ runs over dominant coweights).
The central extension $\hat{G}\to G(\!(t)\!)$ induces a $\G_m$-torsor $\hat{G}/G[\![t]\!]$ over $\Gr_G$ (equipped with a $G[\![t]\!]$-equivariant structure), which we denote by $\LL$.
Taking the push-out with respect to the character $c:K^*\to \C^*$ we get a complex line bundle $\LL_c$ on $\Gr_G(K)$,
%for the critical level $c=crit$ 
and we can identify $i_\H^{\hat{\G}}(\C)$ with the projective limit  
$$i_\H^{\hat{\G}}(\C)=``\varprojlim" \SS(\ov{\Gr}_\la(K),\LL_c).$$
Hence, as a vector space, $\HH(\hat{\G},\H)_c$ can be identified with the inductive limit,
$$\HH(\hat{\G},\H)_c=\varinjlim \Hom_{\H}(\SS(\ov{\Gr}_\la(K),\LL_c),\C).$$
%G(K[\![t]\!])
%\HH(\G,\H):=\HH(\G,\H)_{crit}

%There is a natural functor of coinvariants
%$$\Coinv_{\H}:\Rep_c(\hat{\G})\to \VVect.$$
%
%???
%We will define the following two homomorphisms
%$$\HH(G(K(\!(t)\!)),G(K[\![t]\!]))\to \HH(G(O(\!(t)\!)),G(O[\![t]\!]))\to \HH(G(O/\fm^N(\!(t)\!)),G(O/\fm^N[\![t]\!])).$$

\subsection{Endomorphisms of the functor of coinvariants}\label{coinv-sec}

By \cite[Prop.\ 2.5]{GK1}, we have the functor of coinvariants,
$$\Coinv_{\H}:\Rep(\H)\to \VVect,$$
left adjoint to the functor $\triv_{\H}:\VVect\to \Rep(\H)$ of the trivial representations. 
We want to identify the endomorphisms of the composed functor
$$\Coinv_{\H}r^{\hat{\G}}_{\H}:\Rep_c(\hat{\G})\to \VVect$$
with the opposite algebra to $\HH(\hat{\G},\H)_c$.

First, we observe that the functor $\Coinv_{\H}r^{\hat{\G}}_{\H}$ is left adjoint to the composition 
$$i^{\hat{\G}}_{\H}\triv_{\H}:\VVect\to \Rep_c(\hat{\G}).$$
Hence, we have a natural identification of algebras
$$\End(\Coinv_{\H}r^{\hat{\G}}_{\H})\simeq \End(i^{\hat{\G}}_{\H}\triv_{\H})^{op}.$$
Now we consider the natural evaluation map
\begin{equation}\label{end-hecke-map}
\End(i^{\hat{\G}}_{\H}\triv_{\H})^{op}\to \End_{\Rep_c(\hat{\G})}(i^{\hat{\G}}_{\H}(\C))^{op}=\HH(\hat{\G},\H)_c.
\end{equation}

\begin{prop}
The map \eqref{end-hecke-map} is an isomorphism.
\end{prop}

\begin{proof}
We will construct a map in the opposite direction.
Let us set $F:=i^{\hat{\G}}_{\H}\triv_{\H}$. First, we claim that every $h\in \End_{\Rep_c(\hat{\G})}(F(\C))$
induces an endomorphism $h_V$ of $F(V)$ for every (usual) vector space $V$ that is uniquely characterized
by the property that it is compatible with $h$ via any linear map $\C\to V$.

To construct $h_V$ we choose a basis in $V$, so $V=\bigoplus_i \C e_i$. Let us consider the corresponding embedding
$V\hra \prod_i \C e_i$. Since $F$ is a right adjoint functor, it commutes with products and is left exact, hence, we get the induced embedding
$$F(V)\hra \prod_i F(\C e_i).$$ 
Now the endomorphism $h$ induces an endomorphism $\prod_i h$ of $\prod_i F(\C e_i)$. We claim that it preserves the subobject $F(V)$.
The pro-vector space underlying $F(V)$ is given by 
$$F(V)=``\varprojlim" (\SS_\la\ot V)=``\varprojlim" (\bigoplus_i \SS_\la\ot \C e_i).$$
where we set $\SS_\la:=\SS(\ov{\Gr}_\la(K),\LL_c)$.
On the other hand, as a pro-vector space, $\prod_i F(\C e_i)$ is given by
$$\prod_i F(\C e_i)=``\varprojlim" (\prod_i \SS_\la\ot \C e_i).$$

Now by definition, we have
$$\Hom_{\VVect}(``\varprojlim" \SS_\mu, ``\varprojlim" \SS_\la)=\varprojlim_\la \varinjlim_\mu \Hom(\SS_\mu,\SS_\la).$$
Thus, $h$ is given by a collection of linear maps $h_{\la}:\SS_{\phi(\la)}\to \SS_\la$, for some function $\phi:\La\to \La$,
that are compatible in the following sense: for any $\la'<\la$ there exists a sufficiently large $\mu$ such that the compositions
$$\SS_\mu\to \SS_{\phi(\la)}\rTo{h_\la} \SS_\la\to \SS_{\la'} \ \text{ and}$$
$$\SS_\mu\to \SS_{\phi(\la')}\rTo{h_{\la'}} \SS_{\la'}$$
are the same.
Now the
diagonal endomorphism $\prod_i h$ of $\prod_i F(\C e_i)$ is given by the collection of maps
$$\prod_i h_{\la}: \prod_i \SS_{\phi(\la)}\ot \C e_i\to \prod_i \SS_\la \ot \C e_i.$$
It is clear that these maps preserve the subspaces obtained by replacing direct products by direct sums.
The corresponding endomorphism of the pro-vector space underlying $F(V)$ is given by
$$h_\la\ot \id:\SS_{\phi(\la)}\ot V\to \SS_\la\ot V.$$
Our argument shows that it is in fact compatible with the $\hat{\G}$-action.

It is easy to check that the constructed $h_V$ is compatible with arbitrary linear maps $V'\to V$, and in particular, is uniquely determined.
Since $F$ commutes with projective limits, we can now construct $h_V\in \End(F(V))$ for any pro-vector space $V$.
One also easily checks that the obtained map is inverse to \eqref{end-hecke-map}.
\end{proof}

Next, we will give an explicit formula for the action of the Hecke algebra on coinvariants of some $\hat{\G}$-representations.
Recall that $V\in \Rep_c(\hat{\G})$ is called {\it admissible} if for every congruence-subgroup $\G^i\sub \H$, the pro-vector space of coinvariants $V_{\G^i}$ is a usual
vector space.
The action map $(g,v)\mapsto g^{-1}v$ of $\hat{\G}$ induces a collection of weight-$c$ maps
\begin{equation}\label{Gr-action-maps}
\a_{\la}:\tot_{\Gr_\la}(\LL)\times V_{\G_{i(\la)}}\to V_{\H},
\end{equation}
for some function $i(\la)$, where $\tot_{\Gr_\la}(\LL)$ is the total space of the $\G_m$-bundle corresponding to $\LL$.

\begin{lemma}\label{coinv-action-lem}
For any admissible $V\in \Rep_{c}(\G)$, consider the unit for the adjoint pair $(\Coinv_\H r^{\hat{\G}}_{\H},i^{\hat{\G}}_{\H}\triv_{\H})$,
%there is a natural map of pro-vector spaces
$$u_V:V\to i^{\hat{\G}}_{\H}(V_{\H})\simeq ``\varprojlim" (\SS(\ov{\Gr}_\la,\LL_{c})\ot V_{\H}).$$
Then for any $h\in \HH(\G,\H)_{c}$ and any admissible $V\in \Rep_{c}(\G)$, the corresponding endomorphism $h_V\in \End(V_{\H})$ is determined from the
commutative diagram
\begin{diagram}
V &\rTo{}& V_{\H}\\
\dTo{u_V}&&\dTo{h_V}\\
``\varprojlim" (\SS(\ov{\Gr}_\la,\LL_{c})\ot V_{\H})&\rTo{h_0\ot\id}& V_\H.
\end{diagram}
where $h_0\in \Hom_{\H}(i^{\hat{\G}}_{\H}(\C),\C)$ corresponds to $h$.

More explicitly, if $h$ comes from $h_0\in \Hom_{\H}(\SS(\ov{\Gr}_\la,\LL_{c}),\C)$, then $h_V$ is induced by the composition
$$V_{\G_{i(\la)}}\rTo{\a_\la^*} \SS(\ov{\Gr}_\la,\LL_{c})\ot V_{\H}\rTo{h_0\ot\id} V_\H.$$
\end{lemma}

\begin{proof}
Recall that $h$, viewed as an endomorphism of $i^{\hat{\G}}_{\H}(\C)$ in $\Rep_{c}(\G)$, can be represented by a collection of linear maps
$h_\la:\SS_{\phi(\la)}\to \SS_\la$, where $\SS_\la=\SS(\ov{\Gr}_\la,\LL_{c})$, and for any $W\in Vect$, the corresponding endomorphism $h_W$ of
$i^{\hat{\G}}_{\H}(W)$ is given by the maps $(h_{\la}\ot\id_W)$.

Recall that we use the isomorphism $\End(\Coinv_{\H}r^{\hat{\G}}_{\H})\simeq \End(i^{\hat{\G}}_{\H}\triv_{\H})^{op}$ to define the action of $\HH(\G,\H)_{c}$ on $\Coinv_{\H}$.
Hence, for any $V\in\Rep_{c}(\G)$, the endomorphism $h_V:V_{\H}\to V_{\H}$ corresponds by adjunction to the composed map
$$V\rTo{u_V} i^{\hat{\G}}_{\H}(V_{\H})\rTo{h_{V_{\H}}} i^{\hat{\G}}_{\H}(V_{\H}).$$
This means that $h_V$ is the composition
$$V_{\H}\rTo{\Coinv_\H(u_V)} (i^{\hat{\G}}_{\H}(V_{\H}))_\H \rTo{(h_{V_{\H}})_\H} (i^{\hat{\G}}_{\H}(V_{\H}))_\H\rTo{\eps_{V_{\H}}} V_{\H},$$
where $\eps_W$ is the counit of adjunction.

It is easy to see that for any $W\in Vect$, the composition
$$i^{\hat{\G}}_{\H}(W)\to (i^{\hat{\G}}_{\H}(W))_\H \rTo{(h_W)_\H} (i^{\hat{\G}}_{\H}(W))_\H\rTo{\eps_{W}} W$$
is given by $h_0\ot \id_W$. Now the assertion follows immediately from the commutative diagram
\begin{diagram}
V &\rTo{}& V_{\H}\\
\dTo{u_V}&&\dTo_{\Coinv_{\H}(u_V)}\\
i^{\hat{\G}}_{\H}(V_{\H})&\rTo{}& (i^{\hat{\G}}_{\H}(V_{\H}))_\H
\end{diagram}
\end{proof}

%\begin{prop} For $V\in \Rep_{crit}(\G)$ such that $\Coinv_{\H}(V)\in Vect\sub\VVect$, 
%one has
%$$h_\la\cdot v=\int_{g\in\Gr_\la(K)} g\cdot v.$$
%\end{prop}

\subsection{Critical level and integration}\label{crit-level-sec}

Assuming that the commutator subgroup $[G,G]$ is simply connected, let us consider the central extension $\hat{G}_{crit}\to G(\!(t)\!)$ at the {\it critical level}. It corresponds to
the invariant form $-B/2$ on the Lie algebra $\fg$ of $G$, where $B$ is the Killing form, and is characterized by the
fact that the corresponding $G[\![t]\!]$-equivariant line bundle $\LL_{crit}$ on the affine Grassmannian $\Gr_G$ satisfies 
%definition of $\LL_c$ and its relation to the central extension $\hat{\G}\to \G$. ???
%Recall that we have an isomorphism
\begin{equation}\label{L-crit-isom}
\LL_{crit}|_{\Gr_\la}\simeq \om_{\Gr_\la} \ot L_\la,
\end{equation}
where $L_\la$ is a $1$-dimensional vector space depending multiplicatively on $\la$ (see \cite{BD}, \cite[Thm.\ 2.4]{BK-Hecke}, \cite[Thm.\ 5.1]{BK}).
%In the semisimple simply connected case, $\hat{G}_{crit}$ is the central extension corresponding to $-B/2$, where $B$ is the Killing form.???
By definition, the action of $G(\!(t)\!)$ lifts to an action of $\hat{G}_{crit}$ on $\LL_{crit}$ such that $\G_m$ acts by the identity character.

Recall that the variety $\ov{\Gr}_\la$ is Gorenstein and has rational singularities (see \cite{Faltings}, \cite[Thm.\ 2.2]{BK-Hecke}). This implies that the isomorphism \eqref{L-crit-isom}
extends to a similar isomorphism on $\ov{\Gr}_\la$ (see \cite[Thm.\ 2.5]{BK-Hecke}).
%In fact, there exists a $G[\![t]\!]$-equivariant resolution???
%It follows that for any $

We denote by $\Rep_{crit}(\G):=\Rep_{|\cdot|}(\hat{\G}_{crit})$ the category of representations of the central extension $\hat{\G}=\hat{G}_{crit}(K)$ in pro-vector spaces such that $K^*$ acts via $|\cdot |:K^*\to \C^*$, and by $\HH(\G,\H)_{crit}:=\HH(\hat{G}_{crit},\H)_{|\cdot|}$ the corresponding local Hecke algebra.

\begin{definition} For $\la\in \La_+$, we denote by 
$$h^\la\in \Hom_{\H}(\SS(\ov{\Gr}_\la(K),|\LL_{crit}|),\C)\ot L_\la\sub \HH(\G,\H)_{crit}\ot L_\la$$
the element given by the absolutely convergent integral
$$\mu\mapsto \int_{\Gr_\la(K)}\mu,$$
where we view elements of $\SS(\Gr_\la(K),|\LL_{crit}|)$ as smooth $L_\la$-valued measures on $\Gr_\la(K)$.
\end{definition}

The absolute convergence in the above definition follows from the fact that $\ov{\Gr}_\la$ has rational singularities
(see \cite[Sec.\ 3.4]{AA}).
%is proved in \cite{BK-Hecke}. 
The next theorem is proved in \cite{BK-Hecke}.

\begin{theorem}\label{BK-thm}
The elements $(h^\la)$ commute in $\HH(\G,\H)_{crit}$. 
\end{theorem}

\subsection{Local Hecke algebras over $O$ and $O/\fm^N$}\label{fin-hecke-sec}

The analogs of the constructions and results of Sections \ref{hecke-def-sec} and \ref{coinv-sec} also hold if we replace $K$ with $O$ or with $O/\fm^N$.
%with appropriate simplifications for the algebras $\HH(\G_O,\HH_O)$ and $\HH(\G_N,\H_N)$.

Namely, we consider the group $\G_O:=G(O(\!(t)\!))$ in $\SSet$ and its subgroup $\H_O:=G(O[\![t]\!])$, as well as the group
$\G_N:=G(O/\fm^N(\!(t)\!))$ with the subgroup $\H_N:=G(O/\fm^N[\![t]\!])$.
(Note that we do not consider central extensions of $\G_O$ and $\G_N$, just the usual categories of representations.)

Similarly to Definition \ref{hecke-alg-K-def}, we define local Hecke algebras
$$\HH(\G_O,\H_O):=\End_{\Rep(\G_O)}(i_{\H_O}^{\G_O}(\C))^{op}, \ \ \HH(\G_N,\H_N):=\End_{\Rep(\G_N)}(i_{\H_N}^{\G_N}(\C))^{op}.$$
We still have the identifications
$$\HH(\G_O,\H_O)=\Hom_{\H_O}(i_{\H_O}^{\G_O}(\C),\C)=\varinjlim \Hom_{\H_O}(\SS(\ov{\Gr}_\la(O)),\C),$$
%G(O[\![t]\!])
$$\HH(\G_N,\H_N)=\Hom_{\H_N}(i_{\H_N}^{\G_N}(\C),\C)=\varinjlim \Hom_{\H_N}(\SS(\ov{\Gr}_\la(O/\fm^N)),\C).$$
%G(O/\fm^N[\![t]\!])

%\subsection{Local Hecke algebras over $O/\fm^N$}\label{fin-hecke-sec}

Next, we will show how to identify $\HH(\G_N,\H_N)$ with the algebra $\HH_{G,O/\fm^N}$ of $G(O/\fm^N[\![t]\!])$-biinvariant distributions  with compact support on 
$G(O/\fm^N(\!(t)\!))$. Indeed, 
the group $G(O/\fm^N[\![t]\!])$ acts transitively on the fibers of the map $G(O/\fm^N(\!(t)\!))\to \Gr(O/\fm^N)$. Hence,
the orbits of $G(O/\fm^N[\![t]\!])$ on $\Gr(O/\fm^N)$ are in bijection with the double $G(O/\fm^N[\![t]\!])$-cosets on $G(O/\fm^N(\!(t)\!))$. 
Since the delta-functions of the former orbits form a basis in $\HH(\G_N,\H_N)$, while the delta-functions of the latter double cosets form a basis in $\HH_{G,O/\fm^N}$,
we get an identification
$$\nu:\HH(\G_N,\H_N)\rTo{\sim} \HH_{G,O/\fm^N}.$$
More explicitly, given a $G(O/\fm^N[\![t]\!])$-invariant functional $\de$ on $\SS(\ov{\Gr}_\la(O/\fm^N))$, we construct a $G(O/\fm^N[\![t]\!])$-biinvariant distribution on
$G(O/\fm^N(\!(t)\!))$ as follows: starting with a smooth function on $G(O/\fm^N(\!(t)\!))$ we restrict it to
$G(O/\fm^N(\!(t)\!))_{\le \la}$, the preimage of $\ov{\Gr}_\la(O/\fm^N)$, then integrate along the fibers of the map
$G(O/\fm^N(\!(t)\!))_{\le \la}\to \ov{\Gr}_\la(O/\fm^N)$, and finally apply $\de$ to the resulting function.

\begin{prop} The map $\nu$ is an isomorphism of algebras. 
\end{prop}

\begin{proof}
To understand multiplication on $\HH(\G_N,\H_N)$ let us look again at the isomorphism
$$\End_{\Rep(\G_N)}(i_{\H_N}^{\G_N}(\C))\rTo{\sim} \varinjlim \Hom_{\H_N}(\SS_{\la,N},\C),$$
where $\SS_{\la,N}=\SS(\ov{\Gr}_\la(O/\fm^N))$ (these are finite dimensional $\H_N$-representations). 
We have $i_{\H_N}^{\G_N}(\C)=``\varprojlim" \SS_{\la,N}$,
and an endomorphism $h$ of this pro-vector space is given by a compatible collection 
$$h_\la:\SS_{\phi(\la),N}\to \SS_{\la,N}.$$

The fact that $h$ commutes with the $\G_N$-action allows us to recover $(h_\la)$ from $h_0:\SS_{\phi(0),N}\to \C$.
Namely, for each $\la,\mu$, the action map $(g,\phi)\mapsto (g^{-1}\phi)(x)=\phi(gx)$, induces a map
$$\a:\ov{\Gr}_{\la}(O/\fm^N)\times \SS_{\psi(\la,\mu),N}\to (\SS_{\mu,N})_{\H_N},$$
%where $\G_{N,\la}\sub \G_N$ is the preimage of $\ov{\Gr}_\la(O/\fm^N)$, which descends
Hence, for each $\la$, the composition
$$\ov{\Gr}_{\la}(O/\fm^N)\times \SS_{\psi(\la,\phi(0)),N}\to (\SS_{\phi(0),N})_{\H_N}\rTo{h_0} \C$$
can be viewed as a map
$$\SS_{\psi(\la,\phi(0)),N}\to \SS_\la$$
which induces $h_\la$.

Now, for $h,h'\in \HH(\G_N,\H_N)$, the map $(h'h)_0$ is given as the composition
$$\SS_{\la'}\rTo{h_\la}\SS_{\la}\rTo{h'_0} \C.$$
As we have seen above, $h_\la(\phi)(g)=h_0(g^{-1}\phi)$, where $g^{-1}\phi(x)=\phi(gx)$.
Thus, $(h'h)_0(\phi)$ is obtained by applying $h'_0\ot h_0$ to $\phi(gx)\in \SS(\G_{N,\la'})\ot \SS_{\la,N}$,
where $\G_{N,\la'}\sub \G_N$ is the preimage of $\ov{\Gr}_{\la'}(O/\fm^N)$.
The latter definition is compatible with the usual convolution of distributions on $\G_N$.
\end{proof}

\subsection{Homomorphisms between local Hecke algebras}\label{hecke-hom-sec}

As in Sec.\ \ref{hecke-def-sec} and \ref{coinv-sec}, we continue to work with any central extension $\hat{G}$ and a level $c:K^*\to \C^*$.

First, we want to construct a natural homomorphism of algebras
$$\nu_{K,O}:\HH(\hat{\G},\H)_c\to \HH(\G_O,\H_O).$$

\begin{lemma}
One has a natural isomorphism of functors $\Rep(\H)\to \Rep(\G_O)$,
\begin{equation}\label{ind-K-O-isom}
r^{\hat{\G}}_{\G_O} i_{\H}^{\hat{\G}}\rTo{\sim} i_{\H_O}^{\G_O} r_{\H_O}^\H.
\end{equation}
Furthermore, the following diagram of functors $\Rep(\H)\to \Rep(\H_O)$ is commutative:
\begin{equation}\label{ind-K-O-diagram}
\begin{diagram}
%i_{\H}^{\hat{\G}}\triv_{\H} W &\rTo{}& 
r_{\H_O}^{\H}r_{\H}^{\hat{\G}}i_{\H}^{\hat{\G}}&\rTo{r_{\H_O}^{\H}(\can)}& r_{\H_O}^{\H}\\
\dTo{\sim}&&\dTo{\id}\\
r_{\H_O}^{\G_O}i_{\H_O}^{\G_O}r_{\H_O}^{\H}&\rTo{\can\circ r_{\H_O}^{\H}}& r_{\H_O}^{\H}
\end{diagram}
\end{equation}
where the left vertical arrow is induced by \eqref{ind-K-O-isom}.
\end{lemma}

\begin{proof}
The isomorphism \eqref{ind-K-O-isom} is induced by the identifications
$$\ov{\Gr}_\la(O)\rTo{\sim}\ov{\Gr}_\la(K)$$
and by the trivialization of $\LL_c$ over $\Gr(O)$.
%follows from ind-compactness of the affine Grassmannian $\Gr=\G/\H$.
The commutativity of the diagram follows from the fact that the adjunction maps $r_{\H}^{\hat{\G}}i_{\H}^{\hat{\G}}(V)\to V$ 
and $r_{\H_O}^{\G_O}i_{\H_O}^{\G_O}(W)\to W$ are both given by the evaluation at $1$.
\end{proof}

Now we define the homomorphism $\nu_{K,O}$ as the composition
\begin{align*}
&\HH(\hat{\G},\H)_c\simeq \End(i^{\hat{\G}}_{\H}\triv_{\H})^{op}\to \End(r^{\hat{\G}}_{\G_O}i^{\hat{\G}}_{\H}\triv_{\H})^{op}\simeq 
\End(i_{\H_O}^{\G_O}r_{\H_O}^{\H}\triv_{\H})^{op}\\
&\simeq \End(i_{\H_O}^{\G_O}\triv_{\H_O})^{op}\simeq \HH(\G_O,\H_O),
\end{align*}
where we use the isomorphism \eqref{ind-K-O-isom}, as well as the isomorphism \eqref{end-hecke-map} and a similar isomorphism for 
$\HH(\G_O,\H_O)$.

Note that the isomorphism of $\G_O$-representations $r^{\hat{\G}}_{\G_O}i_{\H}^{\hat{\G}}(\C)\rTo{\sim} i_{\H_O}^{\G_O}(\C)$
comes from the identifications $\SS(\ov{\Gr}_\la(K),|\LL_c|)\simeq\SS(\ov{\Gr}_\la(O))$, and $\nu_{K,O}$ is simply the induced map
$$\varinjlim \Hom_{\H}(\SS(\ov{\Gr}_\la(K),|\LL_c|),\C)\to \varinjlim \Hom_{\H_O}(\SS(\ov{\Gr}_\la(O)),\C).$$
%$\HH(\hat{\G},\H)_c$ with

The homomorphism $\nu_{K,O}$ is compatible with the action on the functors of coinvariants as follows.
Note that we have a natural morphism of functors from $\Rep(\H)$ to $\VVect$,
\begin{equation}\label{coinv-fun-map}
\Coinv_{\H_O}r_{\H_O}^{\H}\to \Coinv_{\H}
\end{equation}
obtained using adjunction from the natural isomorphism $r_{\H_O}^{\H}\triv_{\H}\rTo{\sim} \triv_{\H_O}$.
Namely, the map
$$\Hom(\Coinv_{\H}V,W)\to \Hom(\Coinv_{\H_O}r_{\H_O}^{\H},W)$$
corresponding to \eqref{coinv-fun-map} is given as the composition
\begin{align*}
&\Hom(\Coinv_{\H}V,W)\simeq\Hom(V,\triv_{\H}W)\rTo{r_{\H_O}^{\H}}\Hom(r_{\H_O}^{\H}V,r_{\H_O}^{\H}\triv_{\H}W)\simeq\\
&\Hom(r_{\H_O}^{\H}V,\triv_{\H_O}W)\simeq\Hom(\Coinv_{\H_O}r_{\H_O}^{\H}V,W).
\end{align*}

\begin{lemma}\label{nu-KO-comp-lem}
Suppose we have $V\in \Rep(\G)$, $V_O\in \Rep(\G_O)$, and a morphism $f:V_O\to V$ compatible with the action of $\G_O$.
Then for any $h\in \HH(\hat{\G},\H)_c$, the action of $h$ on $\Coinv_{\H}(V)$ is compatible with the action of $\nu_{K,O}(h)$ on $\Coinv_{\H_O}(V_O)$, i.e.,
the following diagram is commutative
\begin{diagram}
\Coinv_{\H_O}(V_O)&\rTo{\nu_{K,O}(h)_{V_O}}&\Coinv_{\H_O}(V_O)\\
\dTo{} &&\dTo{}\\
\Coinv_{\H}(V)&\rTo{h_V}&\Coinv_{\H}(V)
\end{diagram}
where the vertical maps are induced by $f$ and by the morphism \eqref{coinv-fun-map}.
\end{lemma}

\begin{proof}
{\bf Step 1}. First, we claim that for any map $f:\Coinv_{\H}(V)\to W$, where $V\in \Rep_c(\hat{\G})$, the morphism 
$f':V\to i_{\H}^{\hat{\G}}\triv_{\H} W$ in $\Rep_c(\hat{\G})$ corresponding to $f$ by adjunction, viewed as a morphism in $\Rep(\G_O)$,
coincides with the morphism obtained by adjunction from the composed morphism
$$\Coinv_{\H_O}(V)\to \Coinv_{\H}(V)\to W,$$
where the first arrow is given by the canonical morphism \eqref{coinv-fun-map}.
Indeed, this is equivalent to checking the commutativity of the diagram of functors
\begin{equation}
\begin{diagram}
\Coinv_{\H_O}r_{\H_O}^{\G_O}r_{\G_O}^{\G}i_{\H}^{\hat{\G}}\triv_{\H}&\rTo{\ga}&
\Coinv_{\H_O}r_{\H_O}^{\G_O}i_{\H_O}^{\G_O}\triv_{\H_O}\\
\dTo{\b}&&\dTo{\can}\\
\Coinv_{\H}r_{\H}^{\hat{\G}}i_{\H}^{\hat{\G}}\triv_{\H}&\rTo{\can}&\Id_{\VVect}
\end{diagram}
\end{equation}
where $\ga$ is induced by the isomorphism \eqref{ind-K-O-isom}, $\b$ is induced by \eqref{coinv-fun-map},
and the remaining two arrows are counits of adjunction.
We claim that this follows from the commutativity of the diagram \eqref{ind-K-O-diagram}.
Indeed, $\can\circ\ga$ is equal to the composition 
$$\Coinv_{\H_O}r_{\H_O}^{\G_O}r_{\G_O}^{\G}i_{\H}^{\hat{\G}}\triv_{\H}\rTo{\ga}\Coinv_{\H_O}r_{\H_O}^{\G_O}i_{\H_O}^{\G_O}\triv_{\H_O}\to
\Coinv_{\H_O}\triv_{\H_O}\to \Id_{\VVect}.$$
The composition of the first two arrows is induced (by post-composing with $\Coinv_{\H_O}$ and pre-composing with $\triv_{\H}$) by the map
$$r_{\H_O}^{\G_O}r_{\G_O}^{\G}i_{\H}^{\hat{\G}}\to r_{\H_O}^{\G_O}i_{\H_O}^{\G_O}r_{\H_O}^{\H}\to r_{\H_O}^{\H},$$
which by commutativity of \eqref{ind-K-O-diagram}, is equal to the map
$$r_{\H_O}^{\G_O}r_{\G_O}^{\G}i_{\H}^{\hat{\G}}\simeq r_{\H_O}^{\H}r_{\H}^{\hat{\G}}i_{\H}^{\hat{\G}}\to r_{\H_O}^{\H}$$
induced by the adjunction. Hence, $\can\circ\ga$ is equal to the composition
$$\Coinv_{\H_O}r_{\H_O}^{\G_O}r_{\G_O}^{\G}i_{\H}^{\hat{\G}}\triv_{\H}\simeq \Coinv_{\H_O}r_{\H_O}^{\H}r_{\H}^{\hat{\G}}i_{\H}^{\hat{\G}}\triv_{\H}\to
\Coinv_{\H_O}r_{\H_O}^{\H}\triv_{\H}\simeq \Coinv_{\H_O}\triv_{\H_O}\to\id_{\VVect},$$
induced by adjunctions. Now the fact that this is equal to $\can\circ \b$ follows from the commutative diagram
\begin{diagram}
\Coinv_{\H_O}r_{\H_O}^{\H}r_{\H}^{\hat{\G}}i_{\H}^{\hat{\G}}\triv_{\H}&\rTo{}&\Coinv_{\H_O}r_{\H_O}^{\H}\triv_{\H}&\rTo{}&
\Coinv_{\H_O}\triv_{\H_O}\\
\dTo{}&&\dTo{}&&\dTo{}\\
\Coinv_{\H}r_{\H}^{\hat{\G}}i_{\H}^{\hat{\G}}\triv_{\H}&\rTo{}&\Coinv_{\H}\triv_{\H}
&\rTo{}&\Id_{\VVect}
\end{diagram}
in which commutativity of the right square follows from the definition of the morphism \eqref{ind-K-O-isom}.
%$\can\b$

\noindent
{\bf Step 2}. Step 1 implies that for any $V\in\Rep_c(\hat{\G})$ and $W\in \VVect$, the following square (in which the horizontal arrows
are isomorphisms) is commutative
\begin{equation}\label{Hom-Coinv-diagram}
\begin{diagram}
\Hom(\Coinv_{\H}(V),W)&\rTo{\a}& \Hom_{\G}(V,i_{\H}^{\hat{\G}}\triv_{\H}W)\\
\dTo{\b}&&\dTo{r^{\hat{\G}}_{\G_O}}\\
\Hom(\Coinv_{\H_O}(V),W)&\rTo{\ga\circ\a_O}& \Hom_{\G_O}(r^{\hat{\G}}_{\G_O}V,r^{\hat{\G}}_{\G_O}i_{\H}^{\hat{\G}}\triv_{\H}W)
\end{diagram}
\end{equation}
where $\a$ is given by the adjunction, $\b$ is induced by \eqref{coinv-fun-map}, and the bottom horizontal arrow is the composition of the map
$$\a_O:\Hom(\Coinv_{\H_O}(V),W)\rTo{\sim}\Hom_{\G_O}(V,i_{\H_O}^{\G_O}\triv_{\H_O}W)$$
given the adjunction with the map induced by the isomorphism $\ga:i_{\H_O}^{\G_O}\triv_{\H_O}\rTo{\sim} r_{\G_O}^{\G}i_{\H}^{\hat{\G}}\triv_{\H}$ (see \eqref{ind-K-O-isom}).

For $h\in \HH_{\G,\H}=\End(i_{\H}^{\hat{\G}}\triv_{\H})^{op}$, let $h'$ be the corresponding element of $\End(\Coinv_{\H}r_{\H}^{\hat{\G}})$.
By definition, for any $f\in \Hom(\Coinv_{\H}(V),W)$, we have
\begin{equation}\label{h-a-f-id}
h_W\circ \a(f)=\a(f\circ h'_V).
\end{equation}
Similarly, for $h_O\in \HH_{\G_O,\H_O}=\End(i_{\H_O}^{\G_O}\triv_{\H_O})^{op}$ let $h'_O\in \End(\Coinv_{\H_O}r_{\H_O}^{\G_O})$ be the corresponding element. 
Then we have
\begin{equation}\label{h-aO-fO-id}
h_{O,W}\circ \a_O(f_O)=\a_O(f_O\circ h'_{O,V_O}),
\end{equation}
where $f_O\in \Hom(\Coinv_{\H_O}(V_O),W)$, with $V_O\in \Rep(\G_O)$.

Now let $h_O=\nu_{K,O}(h)$. Note that we have an endomorphism of every vertex of the square \eqref{Hom-Coinv-diagram}:
for the two right vertices they are induced by $h_W$ and $r^{\hat{\G}}_{\G_O}(h_W)$, while for the two left vertices they are induced by $h'$ and $h'_O$.
We want to check that the map $\b$ is compatible with these endomorphisms, i.e., 
$$\b(f\circ h'_V)=\b(f)\circ h'_{O,V_O}.$$
It is enough to check instead that the three other edges in the square are compatible with the endomorphisms.
For the arrow $r^{\hat{\G}}_{\G_O}$, this is clear.
For $\a$ and $\a_O$ the compatibility is given by \eqref{h-a-f-id} and \eqref{h-aO-fO-id}.
Finally, for $\ga$ this corresponds to the identity
$$\ga_W h_{O,W}= r_{\G_O}^{\G}(h_W) \ga_W$$
which follows from the definition of $\nu_{K,O}$.
\end{proof}

Next, we will construct a homomorphism
$$\nu_{O,O/\fm^N}:\HH(\G_O,\H_O)\to \HH(\G_N,\H_N).$$
%We observe that we have a morphism of functors 
%$$i^{\G_N}_{\H_N}r^{\H_O}_{\H_N}\to r^{\G_O}_{\G_N}i^{\hat{\G}}_{\H},$$
%induced by the reduction maps
%$$\ov{\Gr}_\la(O)\to \ov{\Gr}_\la(O/\fm^N).$$
As a map of vector spaces it is induced by the natural maps
$$\Hom_{\H_O}(\SS(\ov{\Gr}_\la(O)),\C)\to \Hom_{\H_N}(\SS(\ov{\Gr}_\la(O/\fm^N)),\C)$$
dual to the pull-back maps $\SS(\ov{\Gr}_\la(O/\fm^N))\to \SS(\ov{\Gr}_\la(O))$.
We claim that it is a homomorphism of algebras.

For this, we use the interpretation of both algebras as endomorphisms of the functor of coinvariants.
Namely, an element $h\in\HH(\G_O,\H_O)$ induces an endomorphism $h_V\in \End(V_{\H_O})$ for any $V\in \G_O$.
Now given a representation $V_N$ of $\G_N$, we can view it as a representation of $\G_O$ via the surjective homomorphism $\G_O\to \G_N$.
Then we observe that
$(V_N)_{\H_O}=(V_N)_{\H_N}$ since $\H_O$ surjects onto $\H_N$. Hence, we can view $h_{V_N}$ as an endomorphism of $(V_N)_{\H_N}$.
Clearly, this construction is compatible with the algebra structure. The fact that it coincides with the map $\nu_{O,O/\fm^N}$
follows easily from the analogs of Lemma \ref{coinv-action-lem} for $\G_O$- and $\G_N$-representations.

%\subsection{Action on $\Bun_G$}

%Discuss homomorphisms of local Hecke algebras and compatibility with actions on $\Bun_G$.???
%Starting from a distribution on $\ov{\Gr}_{\le \la}$ and an element of $\mcS (\Bun_G,|\bo ^{1/2}|)$, we
%want to use this correspondence to produce a new element of $\mcS (\Bun_G,|\bo ^{1/2}|)$.
%???

\subsection{Coordinate-free Hecke algebras and the small Hecke algebra over $O/\fm^N$}\label{coord-free-sec}

All the notions related to the groups $G(\!(t)\!)$, $G[\![t]\!]$ over a local field $K$, including the Hecke algebras $\HH(\hat{\G},\H)_c$ and
their actions on the space of $\H$-coinvariants, can be formulated starting with a local complete $K$-algebra $A_K$, isomorphic to $K[\![t]\!]$ (then the field of quotients of $A_K$
will be isomorphic to $K(\!(t)\!)$).
We denote the corresponding Hecke algebra as $\HH_{\hat{G},A_K,c}$ (which is isomorphic to $\HH(\hat{\G},\H)_c$).
For example, we can take $A_K=\hat{\OO}_{C,v}$, the completion of a local ring of a $K$-point on smooth curve over $K$.

The situation is slightly different with the $O$-integral version. Here, we need to start with an $O$-algebra $A_O$, together with a principal ideal $I\sub A_O$ such that
$A_O/I\simeq O$ and $A_O$ is $I$-adically complete, such that there exists an isomorphism $A_O\simeq O[\![t]\!]$ sending $I$ to $t O[\![t]\!]$.
Then since $I$ is free of rank $1$ as an $A_O$-module, we can define powers $I^{-n}$ for $n>0$, and consider the ring $\varinjlim I^{-n}$ as a replacement of $O(\!(t)\!)$
(to which it is isomorphic). This allows to define the corresponding Hecke algebra $\HH_{G,A_O,I}$ which is isomorphic to $\HH(\G_O,\H_O)$.
For example, we can take $A_O=\hat{\OO}_{C_O,v_O}$, the completion of a smooth curve $C_O$ over $O$ along an $O$-point $v_O:\Spec(O)\hra C_O$,
where the ideal $I\sub A_O$ is the ideal of $v_O(\Spec(O))$.

Finally, for the $O/\fm^N$-version, we start with a local complete $O/\fm^N$-algebra $A_{O/\fm^N}$, isomorphic to $O/\fm^N[\![t]\!]$. Note that for such an algebra
we can consider the complete ring of quotients $Q A_{O/\fm^N}$ (by inverting every non-zero divisor), which is easily seen to be isomorphic to $O/\fm^N(\!(t)\!)$.
Then we can define the corresponding Hecke algebra $\HH_{G,A_{O/\fm^N}}$.
For example, if $C_N$ is a smooth curve over $O/\fm^N$ and $\ov{v}\in \ov{C}(k)$ is a $k$-point of the reduction $\ov{C}=(C_N)_k$, then we can take
$A_{O/\fm^N}=\hat{\OO}_{C_N,\ov{v}}$. Note that if $v_N\in C_N(O/\fm^N)$ is a lifting of $\ov{v}$ then the completion of $C_{O/\fm^N}$ along $v_N$ gives the same algebra.

Now all the constructions of the previous sections can be adjusted so that they would work with the above definitions.
For example, starting with the data $(A_O,I)$ as above, we can define $A_K$ as the completion of $K\ot A_O$ and set $A_{O/\fm^N}=A_O\ot_O O/\fm^N$.
Then we will have natural homomorphisms, compatible with the action on spaces of coinvariants,
\begin{equation}\label{coord-free-hecke-hom-eq}
\nu_{K,O/\fm^N;I}:\HH_{G,A_K,crit}\rTo{\nu_{K,O}} \HH_{G,A_O,I}\rTo{\nu_{O,O/\fm^N}} \HH_{G,A_{O/\fm^N}},
\end{equation}
where $\HH_{G,A_K,crit}:=\HH_{\hat{G}_{crit},A_K,|\cdot|}$.
Note that this homomorphism depends on a choice of an ideal $I\sub A_O$. Namely, we can assume that $A_O=O[\![t]\!]$ and consider ideals of the form
$I=(t+x)$ where $x\in \fm$. The corresponding Hecke operators over $O/\fm^N[\![t]\!]$ obtained from homomorphisms \eqref{coord-free-hecke-hom-eq} depend on
$x\in \fm$.

\begin{definition}
Let $A_{O/\fm^N}$ be an $O/\fm^N$-algebra as above. We define the {\it small Hecke algebra} $\HH^{sm}_{G,A_{O/\fm^N}}$ as the
subalgebra of $\HH_{G,A_{O/\fm^N}}$ generated by the elements $\nu_{K,O/\fm^N;I}(h^\bl)$, where $\bl\in \La^+$, and 
we consider all pairs $(A_O,I)$ as above equipped with an isomorphism $A_O\ot_O O/\fm^N\simeq A_{O/\fm^N}$,
and use 
%$\nu_{K,O/\fm^N}$ is 
the corresponding homomorphisms
\eqref{coord-free-hecke-hom-eq}.
%the images of homomorphisms $\HH_{G,A_O,I}\rTo{\nu_{O,O/\fm^N}} \HH_{G,A_{O/\fm^N}}$,
\end{definition}

Note that by Theorem \ref{BK-thm}, the elements $\nu_{K,O/\fm^N;I}(h^\bl)$ commute for varying $\bl\in\La^+$ and a fixed ideal $I\sub A_O=O[\![t]\!]$ 
(and a fixed identification $A_O\ot_O O/\fm^N\simeq A_{O/\fm^N}$). Our main result, Theorem \ref{small-hecke-thm} below, is that they also commute
for different choices of $I$ (under some assumptions on $G$).

For a connected split reductive group $G$ over a field $k$ we denote by $Z_G\sub G$ the center of $G$, and by $Z^0_G\sub Z_G$ its neutral component 
(and as usual, $\fg$ denotes the Lie algebra of $G$). 
We will consider the following condition on $(G,k)$, which is satisfied if the characteristic of $k$ avoids some primes depending on $G$:

\medskip

\noindent
$\bf{Char_G}$: the center $\fz$ of $\fg$ coincides with the Lie algebra of $Z^0_G$, and the order of the finite group scheme $Z_G/Z^0_G$ is invertible in $k$.

\medskip

For example, this condition always holds for $G=\GL_n$, and it holds for $G=\SL_n$ if and only if $n$ is invertible in $k$.

\begin{theorem}\label{small-hecke-thm}
Assume that 
%\begin{itemize}
%\item
$G$ is a split reductive group over $\Z$, such that 
%$[G,G]$ is simple and 
its commutator subgroup $[G,G]$ is simply connected, and the pair $(G,k=O/\fm)$ satisfies condition $\bf{Char_G}$.
%\item
%the characteristic $p$ of $k=O/\fm$ satisfies $p>\max(a_{\fg'},4(h_{\fg'}-1))$, where $\fg'=[\fg,\fg]$.
%\end{itemize}
Then the small Hecke algebra $\HH^{sm}_{G,O/\fm^N[\![t]\!]}$ is commutative.
\end{theorem}

We will give a proof of this theorem in Section \ref{comm-hecke-sec} using the action of the Hecke algebras on the spaces associated with $\Bun_G$.
A purely local proof of commutation of some of the elements in $\HH^{sm}_{G,O/\fm^N[\![t]\!]}$ for $G=\GL_2$ was given in \cite{BKP}.

\section{Some results on $G$-bundles}\label{G-bun-sec}

In this section we prove some auxiliary results about $G$-bundles on curves. 
%First, we recall some facts on principal $\ssl_2$ subalgebra in Sec.\ \ref{prin-sl2-sec}, working with a group $G$ over any field $k$ whose characteristic is either zero
%or sufficiently large. Then 
In Sec.\ \ref{nice-G-sec}
%, still working over any such field, 
we introduce the a notion of a {\it nice $G$-bundle} of level $n$ with respect to a point on a smooth projective curve over $k$ (nice $G$-bundles exhibit a sufficiently generic cohomological behavior). We refer to the appendix for the proof of existence of curves with such $G$-bundles.
%In Sec.\ \ref{nice-bun-constr-sec} we specialize to the case of a finite field and give a construction of such $G$-bundles, while 
%In Sec.\ \ref{nice-red-sec} 
We also discuss $G$-bundles over $O/\fm^N$ whose reduction over $k=O/\fm$ is nice. 
Then in Sec.\ \ref{gen-triv-O-sec} we consider curves over $O$ and $K$ and
discuss generic triviality of $G$-bundles on them.

\subsection{Nice $G$-bundles}\label{nice-G-sec}

%Now let $G$ be a simple connected group, and let $i:\SL_2\sub G$ be a principal $\SL_2$-subgroup.
%Assume that $G$ is not of type $D_{\ell}$. Then we have a decomposition

\subsubsection{Definition and existence}

Let $G$ be a connected split reductive group over a field $k$,
%such that $[G,G]$ is simple.???
%We also assume that the characteristic of $k$ is either $0$ or $p>\max(a_{\fg'},4(h_{\fg'}-1))$, where $\fg'=[\fg,\fg]$.
%Note that we have a decomposition of the Lie algebra $\fg=[\fg,\fg]\oplus \fz$, where $\fz$ is the center of $\fg$.
%Also, we have a surjective homomorphism $Z_{[G,G]}\to Z_G/Z^0_G$, and if the characteristic of $k$ is $p\neq 0$
%then the orders of both groups are relatively prime to $p$.
%In addition we will always assume that 
satisfying condition $\bf{Char_G}$.
% is satisfied for the pair $(G,k)$.
%make the following assumption (which is satisfied if the characteristic of $k$ avoids some primes depending on $G$).

%\medskip

%Applying the construction of Sec.\ \ref{prin-sl2-sec} to the split simple group $[G,G]$, we get a homomorphism
%$i:\SL_2\to [G,G]\sub G$, corresponding to a principal $\ssl_2$ in $[\fg,\fg]\sub\fg$. Note that the statement of Lemma \ref{principal-sl2-lem} still holds in this situation
%(since we can apply the said Lemma to $G/Z_G$).

%In the case when $[G,G]$ is of type $D_\ell$ with even $\ell$, we denote by $W'$ the representation of $G/Z_G$ defined in Sec.\ \ref{prin-sl2-sec}.

%We also denote by $A$ the standard weight-$1$ representation of the split torus $G/[G,G]$.

Let $C$ be a smooth projective curve over $k$ such that $H^0(C,\OO)=k$.
Note that for every $G$-bundle $P$ on $C$ we have the induced vector bundle $\fg_P$, 
%(resp., $A_P$), 
so any automorphism $\phi$ of $P$ over an open subset $U\sub C$ induces an automorphism
$\phi_\fg$ of the vector bundle $\fg_P|_U$. 
%(resp., $\phi_A$ of $A_P|_U$).
%In the case when $G$ is of type $D_{\ell}$ with even $\ell$, for any $G$-bundle $P$  we have the induced bundle $W'_P$, so an automorphism $\phi$ of $P|_U$ induces an %automorphism
%$\phi_{W'}$ of $W'_P|_U$.

\begin{definition}\label{G-bun-order-pole-def}
Let $P$ and $P'$ be $G$-bundles on $C$ and let $\phi:P|_{C-p}\to P'_{C-p}$ be an isomorphism of $G$-bundles over $C-p$. 
We say that $\phi$ has a pole of order $\le n$ at $p$, if this holds for
the induced homomorphism of the adjoint vector bundles $\phi_{\fg}:\fg_P|_{C-p}\to \fg_{P'}|_{C-p}$ over $C-p$.
%(resp., for the homomorphism $\phi_{W'}:W'_P|_{C_0-p}\to W'_{P'}|_{C_0-p}$, if $G$ is of type $D_{\ell}$ with even $\ell$).
% and for the induced endomorphism $\phi_A$ of $A_P$ over $C_0-p$.
\end{definition}

Note that since $G$ acts trivially on $\fz\sub \fg$, we always have an embedding $\fz\sub H^0(C,\fg_P)$.

\begin{definition}\label{nice-G-bun-def}
Let $p\in C(k)$ be a point. We say that a $G$-bundle $P$ on $C$ is {\it nice of level $n$} at $p$ if
\begin{itemize}
\item $H^0(C,\fg_P(np))=\fz(k)$;
\item every automorphism of $P|_{C-p}$ with a pole of order $\le n$ at $p$, is given by an element of the center $Z_G(k)$.
\end{itemize}
\end{definition}

%Note that since $\fg_P=[\fg,\fg]_P\oplus \fz\ot \OO_{C_0}$, the condition $h^0(C_0,\fg_P(np))=0$ is a consequence of the following two conditions:
%\begin{equation}\label{nice-infinitseimal-part}
%h^0(C_0,[\fg,\fg]_P(np))=0, \ \ H^0(C_0,\OO(np))=k
%\end{equation}
%(the second condition is only needed when $\fz\neq 0$).

The following result will be proved in the appendix (see Proposition \ref{nice-G-bun-prop-bis}).

\begin{prop}\label{nice-G-bun-prop}
Let $G$ be a connected split reductive group over a field $k$ satisfying condition $\bf{Char_G}$. Then
for any $n\ge 1$, there exists a smooth geometrically irreducible projective curve $C$ of genus $g\ge n+2$ over $k$, a point $p\in C(k)$,
and a $G$-bundle $P$ such that $G$ is nice of level $n$ at $p$.
\end{prop}

\subsubsection{$G$-bundles with nice reduction}\label{nice-red-sec}

Now let $O$ be the ring of integers in a local field, and let $k=O/\fm$ be the residue field.
Assume that $G$ is a connected split reductive group over $\Z$, such that 
%$[G,G]$ is simple,
%As before, we assume that the characteristic of $k$ satisfies $p>\max(a_{\fg},4(h-1))$.
condition $\bf{Char_G}$ is satisfied for the pair $(G,k)$.

Let $C$ be a smooth proper curve over $O/\fm^N$, $P$ a $G$-bundle over $C$.
We use an obvious analog of Definition \ref{G-bun-order-pole-def} to define the order of a pole of an automorphism of $P$ along an $O/\fm^N$-point of $C$.
For $i<N$, we set $C_i:=C\times\Spec(O/\fm^{i+1}).$

\begin{lemma}\label{G-nice-fin-ring-lem} 
%Assume that $Z_G$ is a smooth group over $\Z$.
Let $C$ be a smooth proper curve over $O/\fm^N$, $p\in C(O/\fm^N)$, and let $C_0$ be the corresponding curve over $k=O/\fm$, $p_0=p\mod \fm$.
Assume that $P$ is a $G$-bundle on $C$ such that $P|_{C_0}$ is nice of level $n$ at $p_0$. Then any automorphism of $P|_{C-p}$ with a pole of order $\le n$ at $p$ comes
from an element of $Z_G(O/\fm^N)$.
\end{lemma}

\begin{proof}
We will prove by induction on $i$ that the assertion holds for the restriction $P|_{C_i}$. In the case $i=0$ this is a consequence of the assumption that $P|_{C_0}$ is nice.

Let $d$ denote the order of $Z_G/Z^0_G$. Then the group $Z_G$ is smooth over $\Z[d^{-1}]$. Since $d$ is invertible in $O/\fm$, we deduce that each map
$Z_G(O/\fm^{i+1})\to Z_G(O/\fm^i)$ is surjective.

Let $\phi$ be an automorphism of $P|_{C_i-p}$ with a pole of order $\le n$ at $p$, such that the induced automorphism of $P|_{C_{i-1}-p}$ comes from an element $\ga\in Z_G(O/\fm^i)$.
Lifting $\ga$ to an element of $Z_G(O/\fm^{i+1})$ and modifying $\phi$ we can assume that $\phi$ is identity modulo $\fm^i$. Hence, $\phi$ corresponds to a section
of $H^0(C_0,\fg_P(np)|_{C_0})\ot \fm^i/\fm^{i+1}$. It remains to use the equality $H^0(C_0,\fg_P(np)|_{P_0})=\fz(k)$ and the identification of
$\fz$ with the Lie algebra of $Z^0_G$, which gives an embedding 
$$\fz(k)\ot \fm^i/\fm^{i+1}\sub Z^0_G(O/\fm^{i+1})\sub Z_G(O/\fm^{i+1}).$$
\end{proof}

\subsection{Generic triviality and $O$-structures}\label{gen-triv-O-sec}

Let $C$ be a smooth complete irreducible curve over a local non-archimedian field $K$.
Let $G$ be a connected split reductive group over $\Z$. We denote by
$\und{\Bun}_G$ the stack of $G$-bundles on $C$, and set $\Bun_G:=\und{\Bun}_G(K)$. It is known that the stack $\und{\Bun}_G$ is admissible (see \cite[Sec.\ 7.1]{GK2}).

%Let us consider the following condition on $G$:

%($\star$) $G$ is connected split reductive, $[G,G]$ is simply connected, and there exists a {\it principal} homomorphism $i:\SL_2\to G$,  
%a split torus in $G$, 
%defined over $k=O/\fm$.

%Here by a principal homomorphism $\SL_2\to G$ we mean a homomorphism inducing a principal ${\frak sl}_2$-subalgebra in $\fg$.
Let $B\sub G$ be a Borel subgroup containing a split maximal torus $T$. 

%Note that by Lemma \ref{principal-sl2-basic-lem}, condition ($\star$) is satisfied if $[G,G]$ is simply connected and the characteristic of $k$ is sufficiently large.

\begin{lemma}\label{gen-tr-O-lem}
Assume that $C$ has a smooth model $C_O$ over $O$, and let $\Bun^O_G\sub \Bun_G$ denote the subgroupoid of $G$-bundles on $C$ that extend to $G$-bundles
over $C_O$. Let also $\Bun^{gt}_G\sub \Bun_G$ denote the subgroupoid of $G$-bundles that are trivial at the general point of $C$. 

\noindent
(i) We have an inclusion $\Bun^{gt}_G\sub \Bun^O_G$. 

\noindent
(ii) For every $n\ge 0$, set $C_n=C_O\times_{\Spec(O)}\Spec(O/\fm^{n+1})$. Assume the commutator subgroup $[G,G]$ is simply connected.
%satisfies condition ($\star$). 
Then every $G$-bundle over $C_O$ (resp., over $C_n$) admits a $B$-structure and is trivial at the general point. 
In particular, we have $\Bun^{gt}_G=\Bun^O_G$.
\end{lemma}

\begin{proof}
%Let us first prove that $\Bun_G$ is almost proper for $G=GL_N$, i.e., for the stack of vector bundles. We use induction on rank. The assertion clearly holds for line bundles.
%Given a vector bundle $V$ defined over $K$, we can find a line subbundle $L\sub V$ (over $K$). By the induction assumption,
%both $L$ and $V/L$ can be defined over $O$. If $V$ is a trivial extension of $V/L$ by $L$ then we can extend $V$ as a direct sum.
%Otherwise, we have $K$-point of the corresponding projective space $\P \Ext^1(V/L,L)$. Since this projective space is defined over $O$, this
%$K$-point actually corresponds to an $O$-point. The corresponding extension over $O$ gives the required bundle defined over $O$. 
(i) First, let us show that any $G$-bundle $P$ on $C$, trivial at the general point, can be extended to a $G$-bundle on $C_O$.
%By our assumptions on $G$, 
Let $t_U$ be a trivialization of $P$ on an open subset $U=C\setminus D$, where $D\sub C$ is an effective divisor.
We can extend $D$ to a divisor $D_O\sub C_O$. Let $D_k\sub C_k$ be the corresponding divisor on the special fiber. Then using our trivialization $t_U$
we can glue $P$ and the trivial bundle on $C_O\setminus D_O$ into a $G$-bundle $P'$ over $C_O\setminus D_k$.

Let us realize $G$ as a closed subgroup in some $\GL_N$, and let $V'$ be the $\GL_N$-bundle over $C_O\setminus D_k$ associated with $P'$.
Then it is well known that $V'$ extends to a $\GL_N$-bundle $V_O$ over $C_O$ (using the fact that reflexive coherent sheaves on $C_O$ are locally free).
Let $\pi:X_O\to C_O$ be the $\GL_N/G$-fibration classifying reductions of $V_\OO$ to a $G$-bundle. The $G$-bundle $P'$ corresponds to a section $\si$ of $\pi$ over
$C_O\setminus D_k$. Since $\GL_N/G$ is affine, $\si$ extends regularly to a section of $\pi$ on the entire $C_O$. This gives the required extension of $P$ to $C_O$.
%For a general reductive group $G$ we can find an embedding $G\to \GL_N$ as a closed subgroup. For a $G$-bundle $P$ defined over $K$, let $V$ be the associated
%$\GL_N$-bundle. We can find a $\GL_N$-bundle $V_\OO$ defined over $C_O$. Let $X\to C_O$ be the scheme classifying reductions of $V_\OO$ to a $G$-bundle.
%Since $\GL_N/G$ is affine, $X$ is affine over $C_O$, and  $P$ provides a section of the induced morphism $X_K\to C$. ???

(ii) %By (i), we need to show that if $G$ satisfies ($\star$) then every $G$-bundle over $C_O$ (resp., $C_n$)
%admits a $B$-structure and is trivial at the general point. 
This follows from the results of \cite{GPS}.
We will give a somewhat different proof following closely the arguments from the work \cite{DS}.

Since every $B$-bundle is trivial at the general point, it is enough to 
prove that every $G$-bundle $P$ over $C_O$ (resp., $C_n$) admits a $B$-structure.  
Note that a $B$-structure over $C_O$ (resp., $C_n$) is a section of a $G/B$-bundle $P/B$ over $C_O$ (resp., over $C_n$).
Hence, by Grothendieck's existence theorem (see \cite[Cor.\ 8.4.7]{FGA}), it is enough to construct a compatible family of $B$-structures on the restrictions $P_n$ of $P$ to $C_n$.
We do this by first constructing a $B$-structure of special type on $P_0$, and then showing that it extends to a required compatible family. 

The starting point is the fact that for $G$ such that the commutator subgroup $[G,G]$ is simply connected,
%satisfying condition $(\star)$ 
any $G$-bundle on the curve $C_0$ over finite field $k$, is trivial at the general point. Indeed,
for $G$ simply connected and semisimple, this is a theorem of Harder \cite{Harder} on the vanishing of the Galois cohomology $H^1(k(C_0),G)$.
In general, %let $[G,G]$ be the commutator subgroup in $G$. Then 
the result follows from the vanishing of  $H^1(k(C_0),[G,G])$ and $H^1(k(C_0),G/[G,G])$ (note that $G/[G,G]$ is a split torus).

Let $\De$ denote the set of simple roots with respect to $(T,B)$. Recall that for every root $\a$, we can define the degree $\deg_\a(F)\in\Z$ of a $B$-bundle $F$,
as the degree of the line bundle associated with $F$ and the homomorphism $B\to T\to \G_m$ given by $\a$. We claim that for every $N>0$ and every $G$-bundle $P_0$ on $C_0$
there exists a $B$-bundle $F_0$ inducing $P_0$ such that $\deg_\a(F_0)\le -N$ for every simple root $\a$. We prove this as in \cite[Prop.\ 3]{DS}. First, we observe that
if $P'_0$ and $P_0$ are isomorphic over an open subset then the assertions for $P_0$ and for $P'_0$ are equivalent. Indeed, we can choose an isomorphism 
of $P_0$ and $P'_0$ over $C_0\setminus S$, for a finite set of points $S$. Then any $B$-structure on $P_0$ induces one on $P'_0$ and the difference between 
the corresponding degrees $\deg_\a$ is bounded by a constant depending only on the isomorphism $P_0|_{C_0\setminus S}\simeq P'_0|_{C_0\setminus S}$.
Since any $P_0$ is trivial at the general point, it is enough to construct a $B$-structure with sufficiently negative degrees for the trivial $G$-bundle on $C_0$. 
Furthermore, we claim that it is enough to do this for the trivial $G$-bundle on $\P^1$. Indeed,
we can choose a finite morphism $f:C_0\to \P^1$ over $k$ and take the pull-backs of a $B$-structure under $f$.
In the case of the trivial $G$-bundle on $\P^1$, the existence of the needed $B$-structure is proved in Proposition \ref{B-str-prop} in the appendix.
%using the existence of a principal $\SL_2$ in $G$, we reduce to the case $G=\SL_2$ (see \cite[Sec.\ 5]{DS}). In this case we use the $B$-structure corresponding to the subbundle
%$\OO(-N)\sub \OO^2$ generated by $(x_0^N,x_1^N)$, where $(x_0,x_1)$ is the standard basis of $H^0(\P^1,\OO(1))$.

Next, starting with a $B$-structure $F_0$ on $P_0$, with $\deg_\a(F_0)$ sufficiently small (in fact we need them $<2-2g$), we claim that it extends to a collection
of compatible $B$-structures $F_n$ on $P_n$, for all $n\ge 0$. Indeed, we can think of such $B$-structures as sections $\si_n:C_n\to P_n/B$. We just need
to check that each $\si_n$ extends to a section $\si_{n+1}$. As explained in \cite[Prop.\ 1]{DS}, the obstruction to such an extension lies in $H^1(C_0,V)$,
where $V$ is the vector bundle on $C_0$ associated with the $B$-bundle $F_0$ and with the $B$-representation $\fg/\fb$. Now the assumption that $\deg_\a(F_0)$
are sufficiently small for $\a\in\De$ implies that $H^1(C_0,V)=0$.
\end{proof}

\section{Hecke operators on $\Bun_G$}\label{hecke-bunG-sec}

In this section we will use Hecke operators on $\Bun_G$ to prove commutativity of the small Hecke algebra over $O/\fm^N$.

First, in Sec.\ \ref{hecke-bunG-K-sec} we connect standard Hecke operators for $G$-bundles on a curve $C$ over $K$ with the action of
the local Hecke algebra on $\H$-coinvariants of $\G$-representations described in Sec.\ \ref{coinv-sec}. We also establish commutativity of
a natural global Hecke algebra $H(C)$ associated with $C$ (see Theorem \ref{HC-comm-thm}): the case of two Hecke operators at the same point goes back to \cite{BK-Hecke},
while the case of two Hecke operators at different points is new (but not difficult).

Then in Sec.\ \ref{hecke-K-O-sec}, we consider actions of local Hecke algebras on $G$-bundles for curves over $O$ and $O/\fm^N$.
% and show compatibility of these actions with homomorphisms between different types of Hecke algebras constructed in \ref{hecke-hom-sec}.
Note that we do not have explicit descriptions of the Hecke operators over $O$ or over $O/\fm^N$ associated with non-minuscule weights.

Finally, in Sec.\ \ref{comm-hecke-sec}, we give a global proof of Theorem \ref{small-hecke-thm}, using the connection with global Hecke operators over $K$
and nice $G$-bundles.

\subsection{Hecke operators on the Schwartz space of half-densities on $\Bun_G$}\label{hecke-bunG-K-sec}

We refer to \cite{BK} for more details on some constructions and assertions sketched below.

As before, $G$ is a connected split reductive group over $\Z$, 
$C$ is a smooth complete irreducible curve over  $K$.
%We assume that $G$ is ???

We write $\bo$ instead of  $\bo _{\Bun_G}$ and fix 
a square root $\bo ^{1/2}$ of $\bo$ and a square root $\bo _C^{1/2}$ of $\bo _C$ (which we assume to exist over $K$). 

Let us consider the space 
$$\mcW := \mcS (\Bun_G,|\bo ^{1/2}|).$$

\sms

For each $i\ge 0$, let us consider the 
$G[t]/(t^i)$-torsor $\Bun^{i,v,t}_G\to \Bun_G$, defined, over $O$, classifying $G$-bundles together with a trivialization on the $i$-th infinitesimal neighborhood of $v$ given by the ideal $(t^i)$
(we will often write $\Bun^{i,v}_G=\Bun^{i,v,t}_G$ for brevity).
For any open substack of finite type $\YY\sub \Bun_G$, we have the corrresponding torsor $\YY^{i,v}\to \YY$, which is a smooth scheme of finite type for $i\gg 0$.
We denote by $\YY^{\infty,v}$ the corresponding object of $\Pro(Sch^{ft})$, where $Sch^{ft}$ is the category of schemes of finite type.
Let us denote by $\Bun_G^{\infty,v}$ the object $(\YY^{\infty,v})_{\YY}$ of $\Ind(\Pro(Sch^{ft}))$. 

Abusing the notation, we denote by $\bo_{\Bun_G}^{1/2}$ the pull-back of the square root of the canonical bundle on $\Bun_G$ to $\Bun_G^{\infty,v}$.
We have a natural action of $G(\!(t)\!)$, viewed as a group in $\Ind(\Pro(Sch^{ft}))$, on $\Bun_G^{\infty,v}$ (see \cite[Sec.\ 7]{GK2}). This action does not change a $G$-bundle away from $v$,
and multiplies the transition function on the formal punctured neighborhood of $v$ by an element of $G(\!(t)\!)$.

Furthermore, we have a compatible action of the central extension $\hat{G}_{crit}$ at the critical level on the line bundle $\bo_{\Bun_G}^{1/2}$ over $\Bun_G^{\infty,v}$,
such that the center $\G_m$ acts with weight $1$. Note that the action of $G(\!(t)\!)$ on $\Bun_G^{\infty,v}$ fits into a commutative diagram,
where the horizontal maps are given by $(g,x)\mapsto g^{-1}x$,
\begin{equation}\label{hecke-action-diagram}
\begin{diagram}
G(\!(t)\!)\times \Bun_G^{\infty,v}&\rTo{}& \Bun_G^{\infty,v}\\
\dTo{}&&\dTo{}\\
\Gr_G\times \Bun_G^{\infty,v}&\rTo{\b}& \Bun_G
\end{diagram}
\end{equation}
and the above statement about the weight-$1$ action of $\hat{G}_{crit}$ on $\bo_{\Bun_G}$ corresponds to an isomorphism
%This is compatible with the identification 
\begin{equation}\label{hecke-action-line}
\b^*\bo_{\Bun_G}^{1/2}\simeq p^*\LL_{crit}^{-1}\ot q^*\bo_{\Bun_G}^{1/2},
\end{equation}
where $p$ and $q$ are the projections onto the factors $\Gr_G$ and $\Bun_G^{\infty,v}$ (the appearance of $\LL_{crit}^{-1}$ is due to the fact that we use the action of $g^{-1}$).
The latter isomorphism descends to an isomorphism of line bundles on the Hecke correspondence that we will use later (see \eqref{hecke-line-isom}).

Following \cite{GK2} we set $W^i:=\varinjlim_{\YY} \SS(\YY^{i,v}(K),|\bo_{\Bun_G}|^{1/2})$.
Due to the above action of $\hat{G}_{crit}$ on the line bundle $\bo_{\Bun_G}^{1/2}$, 
the pro-vector space
$$W:=``\varprojlim_i" W^i$$ 
acquires a structure of an object of $\Rep_{crit}(\G)$, such that
$W_{\G^i}\simeq W^i$. In particular, 
$$\WW:=W^0=\SS(\Bun_G(K),|\bo|^{1/2})\simeq W_{\H}.$$ 
Thus, we get an action of the local Hecke algebra $\HH(\GG,\HH)_{crit}$ on $\SS(\Bun_G(K),|\bo|^{1/2})$.

Note that we can also define this action without a choice of a formal parameter $t$ at $v$, by replacing $K[\![t]\!]$ with the algebra $\hat{\OO}_{C,v}$
(see Sec.\ \ref{coord-free-sec}).

%\begin{lemma}
%The action of 
%\end{lemma}

The action of the elements $h^\la\in \HH(\GG,\HH)_{crit}$ on $\SS(\Bun_G(K),|\bo|^{1/2})$, for $\bl \in \bL ^+$ and $v\in C(K)$, can be described in terms of the Hecke correspondences.

%For $\bl \in \bL ^+$ and $v\in C(K)$ we 
Let us denote by $Z_{\bl ,v}$ the Hecke correspondence which is the stack of triples 
$(\mcF , \mcG ,j)$ where $ \mcF , \mcG \in \Bun_G$ and $j: \mcF _{C\sm v}\to \mcG _{C\sm v} $ is an isomorphism which is in the position $\mu\le\bl$ at $v$. The natural projections $p_1,p_2: Z_{\bl ,v} \to \Bun_G$ are representable proper maps. If $\bl$ is a minuscule coweight then these projections are also smooth.
%Let $\bo^{1/2}_C$ be a theta characteristic on $C$ (assume there exists one defined over $K$).

The map
%action of $\G$ on $\Bun_G^{\infty,v}$ induces $G[\![t]\!]$-invariant map
$\b:\Gr_G\times \Bun_G^{\infty,v}\to \Bun_G$ (see \eqref{hecke-action-diagram})
%which in turn induces $G[\![t]\!]$-invariant maps
%$$\ga_\la:\ov{\Gr}_{\le \la}\times \Bun_G^{\infty,v}\to \Bun_G,$$
%for each $\la\in \bL^+$.
%On the other hand, we have a $G[\![t]\!]$-invariant projection $q:\ov{\Gr}_{\le \la}\times \Bun_G^{\infty,v}\to \Bun_G$.
%\begin{lemma}\label{action-Gr-lem}
%The action of $G(\!(t)\!)$ on $\Bun_G^{\infty,v}$ (with $g$ replaced by $g^{-1}$) 
gives a collection of maps
$$Z^i_{\bl,v}:=\ov{\Gr}_{\la}\times \Bun_G^{i,v}\rTo{\b_\la} \Bun_G,$$
with $i=i(\la)$,
The maps $\b_\la$ induce the action maps \eqref{Gr-action-maps} for $W\in \Rep_{crit}(\G)$. Furthermore, $\b_\la$ factors as a composition
$$\b_\la:Z^i_{\bl,v}\rTo{\pi_\la} Z_{\bl,v}\rTo{p_2} \Bun_G,$$
where $\pi_\la$ is a $\H/\G^i$-torsor. Also, the natural
projection $Z^i_{\bl,v}\to \Bun_G^{i,v}\to \Bun_G$ coincides with $p_1\pi_\la$.
%\end{lemma}
%\begin{proof}
%The fact that the maps $\b_\la$ induce the action maps for $\hat{\G}$-representation $W$ follows from the definition of the latter $\hat{\G}$-representation in \cite[Sec.\ 7]{GK2}.
%???
%\end{proof}
%The following result was proved in \cite{BD}. This result corresponds to the fact that $W$ is in $\Rep_{crit}(\G)$.

%\begin{claim} There exists a canonical isomorphism of line bundles on $Z_{\bl,v}$,
The isomorphism \eqref{hecke-action-line} descends to an isomorphism of line bundles on $Z_{\bl,v}$,
$$p_2^*\bo^{1/2}\simeq p_1^*\om^{1/2}\ot p^*\LL_{crit}^{-1},$$
where $p$ is the projection to the quotient stack $[\Gr_G/G[\![t]\!]]$.
Using the isomorphism \eqref{L-crit-isom}, we can rewrite this as
\begin{equation}\label{hecke-line-isom}
p_1^* \bo ^{1/2}\rTo{\sim} p_2^* \bo ^{1/2} \otimes \bo _{p_2}\otimes L_{\bl,v},
\end{equation}
where $L_{\bl,v}$ is a $1$-dimensional space depending on $\bl$ and $v$
(in fact $L_{\bl,v}= \bo_C|_v^{- <\bl ,\rho ^\vee>}$, where 
%$ T_C(v) $ is the tangent space to $C$ at $v$ and 
$\rho ^\vee$ is the sum of positive coroots).
%\end{claim}

Thus, in the case when $\bl$ is minuscule we have a well defined composition
\begin{align*}
&T_v^\bl: \SS(\Bun_G(K),|\bo^{1/2}|)\rTo{p_1^\ast} \SS(Z_{\bl,v}(K),|p_1^\ast(\bo^{1/2})|)\simeq \SS(Z_{\bl,v}(K),|p_2^\ast (\bo ^{1/2}) \otimes \bo _{p_2}|)\otimes |L_{\bl,v}|\\
&\rTo{p_{2,!}} \SS(\Bun_G(K),|\bo^{1/2}|)\ot |L_{\bl,v}|.
\end{align*}
Trivializing $L_{\bl,v}$, we can view $T_v^\bl$ as an operator on $\SS(\Bun_G(K),|\bo^{1/2}|)$ defined up to rescaling. 
%for any $w\in \mcW$ and any point $\mcG \in \Bun_G  $ the restriction of 
%$p_1^\ast (w)$ on the fiber $Z_\mcG (F)= p_2^{-1}(\mcG)(F)$ is a measure  
%$\nu (w,\mcG)$ 
%with values in $|\bo ^{1/2}|(\mcG) \otimes |L_v| $.

In the case when $\bl$ is not necessarily minuscule, we can still define the Hecke operators using a resolution of $Z_{\bl,v}$.
Namely, 
%let us use the following alternative description of $Z_{\bl,v}$. 
%Let $\wt{\Bun}_G$ be the stack of $G$-bundles equipped with a trivialization in the formal neighborhood $D_v$ of $v$.
%It is a $G(K[\![t]\!])$-torsor over $\Bun_G$. Let $\Gr_G$ be the affine Grassmannian, and let $\ov{\Om}_{\bl}\sub \Gr_G$ denote the closure of the $G(K[\![t]\!])$-orbit corresponding to
%$\bl$. Then we have an identification
%$$Z_{\bl,v}\simeq \ov{\Om}_{\bl} \times_{G(K[\![t]\!])} \wt{\Bun}_G,$$
%so that the projection $p_2$ corresponds to the obvious projection to $\Bun_G$, while $p_1$ sends $(g,P)$ to the $G$-bundle $\wt{P}$ with the same restrictions to $C\setminus v$ and
%$D_v$ as $P$ but with the transition function multiplied by $g$. 
the variety $\ov{\Gr}_{\bl}$ admits a $G(K[\![t]\!])$-equivariant smooth resolution
$$\rho:\wt{\Gr}_{\bl}\to \ov{\Gr}_{\bl}$$
such that $\rho^*\om_{\ov{\Gr}_{\bl}}\simeq \om_{\wt{\Gr}_{\bl}}(-E)$ for some effective divisor $E$ (see \cite{Faltings}). 
%Furthermore, we can assume that $E$ is a normal crossing divisor.???
Let us set
$$\wt{Z}_{\bl,v}=\wt{\Gr}_{\bl}\times_{G[\![t]\!]/G_i} \Bun_G^{i,v}, \ \ D:=E\times_{G[\![t]\!]/G_i} \Bun_G^{i,v}.$$
We have a natural projection $\rho:\wt{Z}_{\bl,v}\to Z_{\bl,v}$, which is a simultaneous resolution of singularities of the fibers of $p_2$.

Furthermore, locally on a smooth covering $S$ of $\Bun_G$, the map $p_2\pi$ has the form $\wt{\Gr}_{\bl}\times S\to S$, so that the divisor
$D$ corresponds to $E\times S$.
%$p_2\pi$ is smooth and the divisor $D$ induces a normal crossing divisor in all fibers of $p_2\pi$, 
Thus, as was explained in Sec.\ \ref{varieties-sec}, we have a well defined push-forward map
%a well defined map
$$(p_2)_!:\SS(Z_{\bl,v}(K),|p_2^*\om^{1/2}\ot \om_{p_2}|)\to \SS(\Bun_G(K),|\bo^{1/2}|).$$
%$$(p_2\pi)_!:\SS(\wt{Z}_{\bl,v}(K),|(p_2\pi)^*\om^{1/2}\ot \om_{p_2\pi}(-D)|)\to \SS(\Bun_G(K),|\bo^{1/2}|).$$

%The canonical isomorphism \eqref{hecke-line-isom} 
%together with an isomorphism
%$\pi^*\om_{p_2}\simeq \om_{p_2\pi}(-D)$ gives an isomorphism
%$$(p_1\pi)^*\om^{1/2}\rTo{\sim} (p_2\pi)^*\om^{1/2}\ot \om_{p_2\pi}(-D).$$
%Since, the morphism $p_1\pi$ is proper, we get a map
%$$(p_1\pi)^*:\SS(\Bun_G(K),|\bo^{1/2}|)\to \SS(\wt{Z}_{\bl,v}(K),|(p_2\pi)^*\om^{1/2}\ot \om_{p_2\pi}(-D)|).$$

Hence, using the isomorphism \eqref{hecke-line-isom}, 
we can still define the operator $T_v^\bl$ on $\SS(\Bun_G(K),|\bo^{1/2}|)$ 
as the composition $(p_2)_!p_1^*$. It is easy to see that this definition does not depend on a choice of resolution of $\ov{\Gr}_{\bl}$.
%the elements of $p_1^\ast(\SS(\Bun_G(K),|\bo^{1/2}|))$ do not necessarily have compact support (since $p_1$ is not proper). Nevertheless,
%as shown in \cite{BK-Hecke}, the composition $T_v^\bl=p_{2,!}p_1^\ast$ is still well defined (the relevant integrals are absolutely convergent).

\begin{lemma}
Under the identification $\WW=W_\H$, the action of the element $h^\la\in \HH(\G,\H)_{crit}$ on $W_\H$ coincides (up to rescaling) with the operator $T_v^\bl$ defined above.
\end{lemma}

\begin{proof} This follows from Lemma \ref{coinv-action-lem} and from the relation between the Hecke correspondence and $G(\!(t)\!)$-action on $\Bun_G^{\infty,v}$ described above.
% and Lemma \ref{action-Gr-lem}. ???
\end{proof}

\begin{definition} 
We denote by $H(C)\subset \End (\mcW) $ the subalgebra generated by the operators $ T_v^\bl$ for 
$v\in C(K),\bl \in \bL^+$. 
\end{definition}

\begin{theorem}\label{HC-comm-thm} 
The algebra $H(C)$ is commutative.
\end{theorem}

\begin{proof}
For a fixed point $v\in C$, the operators $T_v^\bl$ and $T_v^{\mu}$ commute by Theorem \ref{BK-thm}.
Let us now show that the operators $T_v^\bl$ and $T_{v'}^\mu$ for distinct points $v,v'\in C(K)$ commute.
Since we want to prove the equality of some convergent integrals, we can work over open parts of our correspondences.
The composition $T_v^\bl\circ T_{v'}^\mu$ is given by the composed correspondence
$$Z(v',v):=Z_{\mu,v'}\times_{\Bun_G} Z_{\bl,v}$$
equipped with projections $p_1,p_2$ to $\Bun_G$, and by the isomorphism
$$\a_{v',v}:p_1^*\om^{1/2}\rTo{\sim} p_2^*\om^{1/2}\ot \om_{p_2}\ot L_{\bl,v}\ot L_{\mu,v'}$$
induced by \eqref{hecke-line-isom}.
The composition $T_{v'}^\mu\circ T_v^\bl$ is given by the correspondence $Z(v,v')$ which is naturally identified with $Z(v',v)$ and
some isomorphism $\a_{v,v'}$ between the same line bundles.
Note that the composed operators depend only on the absolute value of the isomorphisms $\a_{v,v'}$ and $\a_{v',v}$. We claim that in fact 
$$\a_{v',v}=\pm \a_{v,v'},$$
which implies the required commutativity.

Indeed, to prove this we can work over an algebraically closed field.
A priori we have $\a_{v',v}=f_{v',v}\cdot \a_{v,v'}$ for some invertible function $f_{v,v'}$ on $Z(v,v')=Z(v',v)$. Note that by construction, we have
\begin{equation}\label{fvv'-sym-eq}
f_{v,v'}\cdot f_{v',v}=1.
\end{equation}
Let us work over a fixed connected component of $\Bun_G$ via $p_2$.
Then all global functions on this component are constant, hence, all global functions on the corresponding component of $Z(v,v')$ are constant,
so the restriction of $f_{v,v'}$ is constant. Now let us vary the points $v$ and $v'$, so that $f_{v,v'}$ will become an invertible function on the complement of the diagonal in 
$C\times C$. The relation \eqref{fvv'-sym-eq} shows that in fact, $f_{v,v'}$ is regular on $C\times C$, hence constant. Now the same relation shows that $f_{v,v'}=\pm 1$.
\end{proof}

%and  $\mC (|\mcL|) $ is the spaces of locally constant  $\mC$-valued functions on $X(F)$. For a top-form 
%$s\in \GG (\bo _X)$  
%we denote by $\bar s$ the function on $X(F)$ corresponding to the section $|s|\in \GG (X(F),|\bo _X|)$.
%\sms

%Let $\kk _r: X(\mcO)\to X( \mcO/ \fm ^r)$ be the reduction map.  We denote by $\mu$ be the measure on $ X(\mcO) $ such $\int _{\kk _r^{-1}} (a)= q^{-rdim (X)}$.

%\begin{claim}
%If $X$ is proper then $X(F)= X(\mcO)$ and
% $\int _{X(F)} |s| = \int _{X(\mcO)}\bar s\mu $ for any top-form $s$ on $X$.
%\end{claim}

\subsection{Hecke operators over $K$, over $O$, and over $O/\fm^n$}\label{hecke-K-O-sec}

Now, assume $C$ is a smooth proper curve over $O$. Then the stack $\Bun_G$ is also defined over $O$, and 
we can consider the spaces 
$$\mcW_O  := \mcS  (\Bun_G(O),|\bo ^{1/2}|)=\mcS (\Bun_G(O)), \ \ \mcW_N :=\mcS(\Bun_G(O/\fm^N)),$$
where in the second case we consider finitely supported functions on the isomorphism classes of $G$-bundles on 
$C_n:=C\times_{\Spec(O)}\Spec(O/\fm^n)$.
%=\varinjlim \mcS(\Bun_G(O/\fm^n)) .$$
We have natural maps constructed in Sec.\ \ref{KO-stacks-sec},
\begin{equation}\label{W-maps-eq}
\mcW_N\to \mcW_O\to \mcW,
\end{equation}
where the image of the map $\mcW_O\to \mcW$ consists of densities supported on $\Bun^O_G$ (see Prop.\ \ref{adm-surj-prop}).

%For a fixed point $v\in C_O(O)=C(K)$ and a formal parameter $t$ along $v(\Spec(O))$, 
%we will consider actions of our three local Hecke algebras on $\SS(\Bun_G(K), |\bo|^{1/2})$, $\SS(\Bun_G(O), |\bo|^{1/2})$ and $\SS(\Bun_G(O/\fm^N))$,
%respectively. We will check that these actions are compatible via the homomorphisms $\nu_{K,O}$, $\nu_{O,O/\fm^N}$ and the maps 
%$r_*^{\Bun_G}$ and $r_{n,\LL_{crit}}$ (see Sec.\ \ref{KO-stacks-sec}).

Given a point $v\in C(O)=C_K(K)$, we have homomorphisms of local Hecke algebras
\begin{equation}\label{Hecke-alg-maps-eq}
\HH_{G,\hat{\OO}_{C_K,v},crit}\rTo{\nu_{K,O}} \HH_{G,\hat{\OO}_{C,v},I_v}\rTo{\nu_{O,O/\fm^N}} \HH_{G,\hat{\OO}_{C_n,\ov{v}}},
%\HH(\G,\H)_{crit}\rTo{\nu_{K,O}} \HH(\G_O,\H_O)\rTo{\nu_{O,O/\fm^N}} \HH(\G_N,\H_N),
\end{equation}
where $\ov{v}$ is the reduction of $v$ modulo $\fm$. We also have three algebra actions:
%and a formal parameter $t$ on $C$ at $v$, we have three algebra actions:
\begin{itemize}
\item $\HH_{G,\hat{\OO}_{C_K,v},crit}$-action on $\mcW$;
%$\HH(\G,\H)_{crit}$-action on $\mcW$, 
\item $\HH_{G,\hat{\OO}_{C,v},I_v}$-action on $\mcW_O$;
%$\HH(\G_O,\H_O)$-action on $\mcW_O$,
\item $\HH_{G,\hat{\OO}_{C_n,\ov{v}}}$-action  on $\mcW_N$.
%$\HH(\G_N,\H_N)$-action on $\mcW_N$.
\end{itemize}
We claim that the maps \eqref{W-maps-eq} are compatible with these Hecke actions via homomorphisms \eqref{Hecke-alg-maps-eq}.
Indeed, for $\nu_{K,O}$ this follows immediately from Lemma \ref{nu-KO-comp-lem}.
For $\nu_{O,O/\fm^N}$ the compatibility of the Hecke actions on coinvariants was explained in Sec.\ \ref{hecke-hom-sec}, as part of the definition.

\begin{definition}
For $\bl\in \La^+$, let us set 
$$h^\la_O:=\nu_{K,O}(h^\la), \ \  h^\la_N:=\nu_{O,O/\fm^N}\nu_{K,O}(h^\la).$$
We denote by $T_{v,O}^{\bl}$ (resp., $T_{v,O/\fm^N}^{\bl}$) the operator on $\WW_O$ (resp., $\WW_N$) given by the action of
$h^\la_O$ (resp., $h^\la_N$) associated with $v\in C(O)$. 
\end{definition}

\begin{lemma}
Assume $\bl$ is minuscule. Then for any point $v\in C(O)$, the operator $T_{v,O/\fm^N}^{\bl}$
coincides (up to rescaling) with the operators on $\SS(\Bun_G(O/\fm^n))$ coming from
the Hecke correspondence $Z_{\bl,v}$. 
%defines Hecke operators
%$T_v^{\bl}$ on $\SS(\Bun_G(O/\fm^n))$. 
\end{lemma}

\begin{proof}
This follows from Proposition \ref{int-point-integration-prop}.
%the maps $r_{n,\bo^{1/2}}:\SS(\Bun_G(O/\fm^n))\to \mcW$ are compatible with the Hecke operators (up to rescaling).
\end{proof}

%Finally, for each $n>0$, we have an identification of the groupoid $\Bun_G(O/\fm^n)$ with the groupoid of $G$-bundles over 
%the curve $C_n:=C\times_{\Spec(O)}\Spec(O/\fm^n)$ over the finite ring $O/\fm^n$.
%Thus, at least in the case of minuscule $\bl$, to understand the action of the operators $T_v^{\bl}$ on $\mcW$, it is enough to understand the action of
%similar operators on the spaces of automorphic functions for the curves $C_n$ over $O/\fm^n$.

%\begin{remark} For an arbitrary $\bl\in\La^+$ one can similarly check that for $v\in C(O)$ the operators $T_v^\bl$ are compatible
%with the operators on $\SS(\Bun_G(O/\fm^n))$ given by some function on the Hecke correspondence $Z_{\le \la,v}=\cup_{\mu\le \la} Z_{\mu,v}$.
%This function is determined by the singularities of the corresponding orbit closure $\ov{\Om}_{\la}$ in the affine Grassmannian.
%\end{remark}

%Theorem B suggests the following conjecture.

%\bigskip

%\noindent
%{\bf Conjecture.} {\it Assume that $C_O$ is a smooth proper curve over $O$ such that $C_k$ is geometrically irreducible. Then the Hecke operators on $\Bun_G(O/\fm^n)$
%associated with $\la\in\La^+$ and simple divisors $c\in C$, commute with each other.
%}

%\bigskip

%Below we will prove this conjecture under certain extra technical assumptions.???

\subsection{Filtration on the stack of $G$-bundles}

Here $G$ is a linear algebraic group, $C$ a smooth projective curve over a Noetherian base scheme $S$.
We denote by $\Bun_G(C)$ the moduli stack of $G$-bundles on $C$: to an $S$-scheme $T\to S$ it associates the groupoid
of $G$-bundles over $C_T=C\times_S T$.  
If $D\sub C$ is a positive relative Cartier divisor, then we denote by $\Bun_G(C,nD)$ the fibration over $\Bun_G$, corresponding to
choices of a trivialization of a $G$-bundle over $nD$.

\begin{prop}\label{filtration-prop}
There exists an exhaustive filtration $\UU_N\sub \UU_{N+1}\sub\ldots $ of $\Bun_G(C)$ by open substacks such that 
for each $N$ there exists $n_0$ such that for any positive relative Cartier divisor $D\sub C$ and any $n\ge n_0$, the fibered product 
$$\UU_N(nD):=\UU_N\times_{\Bun_G(C)}\Bun_G(C,nD)$$
is a separated scheme of finite type over $S$.

If $G$ is reductive, then in addition, for each $N$, every $S$-point $p\sub C$, 
and every dominant coweight $\la$ there exists $N'$, such that all Hecke transforms of coweight $\le \la$ at $p$
take $\UU_N$ to $\UU_{N'}$. 
\end{prop}

\begin{proof}
{\bf Step 1}.
First, assume $G=\GL_r$. For every vector bundle $V$ over $C_s$, where $s\in S$ is a point, let us denote by $\mu_+(V)$ and $\mu_-(V)$ the maximum and minimum slopes of the Harder-Narasimhan subquotients of $V$. Then we define $\UU_N$ to be the substack of vector bundles $V$ over $C_T$ (where $T$ is an $S$-scheme) 
such that $\mu_+(V_t)<N$ and $\mu_-(V_t)>-N$ for all $t\in T$. 
There is finitely many possible degrees for $V$ in $\UU_N$, so it is enough to prove our assertion for the component corresponding to a fixed degree $d$.

\noindent
{\bf Step 2}. Let us fix $N$. It is a standard fact that there exists $n_1$ (depending on $N$) such that for $n\ge n_1$, for all $V$ in $\UU_N(T)$ one has
$H^1(C_t,V_t(nD))=0$, $H^0(C_t,V_t(-nD))=0$ and $V_t(nD)$ is generated by global sections (see \cite[Lem.\ 5.2]{Newstead}). We set $n_0=2n_1$.

It follows that for $n\ge 2n_1$, the restriction map 
$$r_V: H^0(C_t,V_t(n_1D))\to H^0(V_t(n_1D)|_{nD_t})$$ 
is injective, while the corresponding map of bundles over $T$ is an embedding as a subbundle.

We can argue locally over $S$, so
we can assume $S$ to be affine and $H^0(C,\OO(n_1D)|_{nD})$ to be a free module over $\OO(S)$. Let us choose a basis in this $\OO(S)$-module.
Then for all $V\in \UU_N(T)$ equipped with a trivialization at $nD$, we get a basis of $H^0(C_t,V_t(n_1D|_{nD_t}))$,
so the map $r_V$ gives a map to the Grassmannian $T\to G(k,nr)\times S$, where $k=\chi(C,V(n_1D))=d+(n_1\deg(D)-g+1)r$.
Thus, we get a morphism
$$\pi:\UU_N(nD)\to G(k,nr)\times S.$$

\noindent
{\bf Step 3}. We claim that the map $\pi$ is representable and separated of finite type, which implies that $\UU_N(nD)$ is itself a separated scheme.
It is enough to prove that $\pi^{-1}(U_I\times S)$ is a separated scheme of finte type, where $U_I\sub G(k,nr)$ is the open affine cell associated with a subset
$I\sub \{1,\ldots,nr\}$ of size $k$. To this end we observe that the universal subbundle on the Grassmannian has a natural trivialization over $U_I$.
Hence, over $\pi^{-1}(U_I\times S)$, the spaces $H^0(C_t,V_t(n_1D))$ are equipped with a basis (i.e., the corresponding vector bundle over $\pi^{-1}(U_I\times S)$ is trivialized).
Since $V_t(n_1D)$ is a quotient of $H^0(C_t,V_t(n_1D))\ot \OO$, we get a natural morphism
$$j: \pi^{-1}(U_I\times S)\to \Quot(C/S,\OO^k)$$
to the relative Quot scheme of quotients of $\OO^k$. By the standard methods one checks that $j$ is a locally closed embedding (see \cite[Thm.\ 5.3]{Newstead}).

\noindent
{\bf Step 4}. Let $V$ be in $\UU_N$. Then for any $V'$ such that $V(-np)\sub V'\sub V(np)$, one has $V'$ is in $\UU_{N+n\deg(p)}$.
This implies the assertion about the Hecke transforms.

\noindent
{\bf Step 5}. Now we will consider the case of a general linear group $G$. Let us choose an embedding $G\hra \GL_r$ as a closed subgroup.
The corresponding morphism 
$$h:\Bun_G(C)\to \Bun_{\GL_r}(C)$$
is known to be representable and quasi-projective. More precisely, choices of a reduction of a structure group of a $\GL_r$-bundle $P$ from $\GL_r$ to $G$ correspond
to sections of the associated $\GL_r/G$-fibration $P/G$ over $C$, which is quasi-projective over $C$ (by Chevalley's theorem, see \cite[Sec.\ 3.6.7]{Sorger}). Thus, if $P_S$ if a family of $\GL_r$-bundles over 
$S$ then $h^{-1}(S)$ is representable by the scheme of sections of the $\GL_r/G$-fibration over $C\times S$ associated with $P_S$. The latter scheme is quasi-projective over $S$
(as follows from theory of Hilbert schemes, see \cite[Thm.\ 5.23]{FGA}).

Similarly, the morphism
$$h(nD):\Bun_G(C,nD)\to \Bun_{\GL_r}(C,nD)$$
is quasi-projective. The only change to make in the above argument is that now the $G/\GL_r$-fibration associated with an object of $\Bun_{\GL_r}(C,nD)$
is trivialized over $nD$, and we consider sections compatible with this trivialization. 

Let $\UU_{\GL_r,N}\sub \Bun_{\GL_r}$ denote the substacks defined in Step 1.
We claim that the filtration of $\Bun_G(C)$ given by the open substacks $h^{-1}(\UU_{\GL_r,N})$ has the required properties.
Indeed, we need to check that $h(nD)^{-1}(\UU_{\GL_r,N}(nD))$ are separated schemes. But this follows from the fact that $h(nD)$ is quasi-projective and the
fact that $\UU_{\GL_r,N}(nD)$ are separated schemes.

Finally, the compatibility with Hecke transforms (for reductive $G$) follows from the case of $\GL_r$.
\end{proof}

\subsection{Commuting Hecke operators}\label{comm-hecke-sec}

Now assume that $G$ is a split reductive group over $\Z$, such that its commutator subgroup $[G,G]$ is %simple and 
simply connected.
We also assume that %the characteristic $p$ of 
the pair $(G,k=O/\fm)$ satisfies the assumption $\bf{Char_G}$ (see Sec.\ \ref{nice-G-sec}).
%$p>\max(a_{\fg'},4(h_{\fg'}-1))$, where $\fg'=[\fg,\fg]$.

Let $C_{O/\fm^N}$ be a smooth proper curve over $O/\fm^N$, $C_0$ the corresponding curve over $k=O/\fm$.
For a $G$-bundle  $P$ over $C_{O/\fm^N}$, and a point $p\in C_{O/\fm^N}(O/\fm^N)$,
%, and let $\Gr_G$ denote the affine Grassmannian. 
we have a map
$$h_p:\Gr_G(O/\fm^N)\to \Bun_G(O/\fm^N):x\mapsto P(x)$$
associating with a point of $\Gr_G$ the corresponding Hecke transform $P(x)$ at $p$ (to be precise, this map depends on a trivialization of $P$ in the formal neighhborhood of $p$). 
Note that for every point $x\in \Gr_G$, we have an isomorphism $\a_x:P\to P(x)$ on $C_N-p$.
Let us denote by $\Gr_n\sub \Gr_G$ the subscheme of $x$ such that $\a_x$ and $\a_x^{-1}$ have poles of order $\le n$ at $p$
(in the sense of Definition \ref{G-bun-order-pole-def}).

\begin{lemma}\label{Gr-inj-lem} 
Let $P$ be a $G$-bundle over $C_{O/\fm^N}$, such that the corresponding $G$-bundle over $C_0$ is nice of level $2n$ at $p\mod \fm$. Then
%\begin{enumerate} \item 
the map $h_p|_{\Gr_n}:\Gr_n(O/\fm^N)\to \Bun_G(O/\fm^N)$ is injective;
%\item for every $x\in \Gr_n(O/\fm^N)$, one has $\Aut(P(x))=Z_G(O/\fm^N)$.
%\end{enumerate}
\end{lemma} 

\begin{proof}
%(1) 
For every pair $x,x'\in \Gr_n$, the automorphism $\a_x^{-1}\a_{x'}$ of $P|_{C_N-p}$ has a pole of order $\le 2n$ at $p$.
Then by Lemma \ref{G-nice-fin-ring-lem}, there exists an element $z\in Z_G(O/\fm^N)$ such that
$\a_x^{-1}\a_{x'}=z$. Hence, $\a_{x'}=z\a_{x}=\a_{zx}=\a_x$, where we use the fact that $Z_G(O/\fm^N)\sub Z_G(O/\fm^N[\![t]\!])$ acts trivially on $\Gr_G$. Hence, $x'=x$.
%\noindent
%(2) Let $\phi$ be an automorphism of $P(x)$. Then $\a_x^{-1}\phi\a_x$ is an automorphism of $P|_{C-p}$ with a pole of order $\le 2n$ at $p$.
%Hence, $\a_x^{-1}\phi\a_x=z\in Z_G(O/\fm^N)$, which implies that $\phi=\a_x z\a_x^{-1}=z$.
\end{proof}

%Discuss the construction of Hecke operators over $O/\fm^N$??? Explain local construction and its compatibility with the global construction.

Now let 
%$C_O$ be a smooth proper curve over $O$, $C$ (resp., $C_0$) the corresponding curve over $K$ (resp., $k$).
%Let $v_1,v_2\in C(K)=C_O(O)$ be points such that 
%$$v_1\mod \fm=v_2 \mod \fm=v_0\in C(k),$$ and let 
%$T_1$ and $T_2$ be the Hecke operators on $\SS(\Bun_G(O/\fm^N))$ associated with $v_1$ and $v_2$
%and some dominant coweights $\la_1, \la_2$. 
$C_O$ be a smooth proper curve over $O$, with $C_k$ the corresponding curve over $k$.
Let $v_1,v_2\in C_O(K)=C_O(O)$ be distinct points with the same reduction $v_0\in C_k(k)$. 
We fix a pair of dominant coweights $\la_1,\la_2$, and also fix a relative Cartier divisor $D\sub C_O$ disjoint from $v_0$.
We consider the corresponding Hecke operators $T_1$ and $T_2$ on $\SS(\Bun_G(C_O,nD)(O/\fm^N))$, coming 
from the homomorphisms of Hecke algebras $\nu_{K,O/\fm^N}$ for $v_1$ applied to $h^{\la_1}$ and for $v_2$ applied to $h^{\la_2}$.

Note that by Lemma \ref{gen-tr-O-lem}(ii), every $G$-bundle over $C_O$ admits a $B$-structure, and so is Zariski locally trivial.

\begin{lemma}\label{main-lem}
Let $\UU_m\sub\Bun_G(C_O)$ denote one of the open substacks from Proposition \ref{filtration-prop} (defined for the curve $C_O$ over $\Spec(O)$).
Then there exists $n_0$, such that for $n\ge n_0$, one has
$T_1T_2=T_2T_1$ on $\SS(\UU_m(nD)(O/\fm^N))\sub \SS(\Bun_G(C_O,nD)(O/\fm^N))$.
\end{lemma}

\begin{proof}
\noindent
{\bf Step 1}. 
%Let us choose a positive relative Cartier divisor $D\sub C_O$ disjoint from $v_0\in C_k$.
Recall that by Proposition \ref{filtration-prop}, $\UU_m(nD)$ is a separated scheme of finite type over $O$, provided $n$ be sufficiently large.
%Let $\UU_K\sub \Bun_G(C_K)$ (resp., $\UU_{O/\fm^N}\sub\Bun_G(C_{O/\fm^N})$) be the corresponding open substack obtained by the base change.
Hence, for such $n$ the natural map $\UU_m(nD)(O)\to \UU_m(nD)(K)$ is injective, by the valuative criterion of separatedness.
Therefore, the map
$$E_{N,|\om|^{1/2}}:\SS(\UU_m(nD)(O/\fm^N))\to \SS(\UU_m(nD)(K),|\om|^{1/2})$$
is injective. 

\medskip

\noindent
{\bf Step 2}. 
For a moment let us work over $K$, and let $C=C_K$. By abuse of notation we denote $D_K$ simply by $D$.
% and write $C=C_K$, $\UU=\UU_K$.
Recall that for a point $v\in C(K)$ and a dominant coweight $\la$, we have the proper corrrespondence $Z_{\la,v}\rTo{p_1,p_2} \Bun_G(C)^2$.
If $v$ is disjoint from $D$ then we have a natural version of this correspondence with trivializations over $nD$:
$$Z_{\la,v}(nD)\rTo{p_1,p_2} \Bun_G(C,nD)^2.$$

By Proposition \ref{filtration-prop}, for every $m$, there exists $m'$ such that we have an inclusion
$$Z_{\la,v}^{\UU_m}(nD):=p_1^{-1}(\UU_m(nD))\sub p_2^{-1}(\UU_{m'})\sub Z_{\la,v}(nD).$$
%Let $\UU(v,\la)\sub \Bun(C)$ denote the open substack of $x$ such that $p_1^{-1}(x)\sub p_2^{-1}(\UU)$.
%In other words, $\UU(v,\la)$ corresponds to $G$-bundles whose Hecke transforms at $v$ using points in $\Gr_{\la}$ are all in $\UU$.
%The openness of $\UU(v,\la)$ follows from the fact that $p_1$ is proper. 
Thus, we can view $Z_{\la,v}^{\UU_m}(nD)$ as a correspondence
% from $\UU_m(nD)$ to $\UU_{m'}(nD)$:
%Let $Z^{\UU}_{\la,v}\sub Z_{\la,v}$ denote the preimage of $\UU(v,\la)$ under the projection $p_1$. Then the correspondence
$$\UU_m(nD)\lTo{p_1}Z^{\UU_m}_{\la,v}(nD)\rTo{p_2} \UU_{m'}(nD)$$
with $p_1$ proper. Together with the isomorphism of relevant line bundles (see Eq.\ (5.3)), this gives an operator 
$$T_v^{\la,\UU_m}:\SS(\UU_m(nD)(K),|\om|^{1/2})\to \SS(\UU_{m'}(nD)(K),|\om|^{1/2}).$$

Next, for a pair of points $v_1\neq v_2$, disjoint from $D$, and dominant coweights $\la_1,\la_2$, for each $m$, we can choose $m'$ and $m''$, so that the Hecke correspondences corresponding
to $(v_1,\la_1)$ and $(v_2,\la_2)$ give operators
%we denote by $\UU(v_1,v_2,\la_1,\la_2)\sub \Bun(C)$ the open substack of $G$-bundles
%whose Hecke transforms at $v_2$ using points in $\Gr_{\la_2}$ are all in $\UU(v_1,\la_1)$. Using the Hecke correspondence we get an operator
%$$T_{v_2}^{\la_2,\UU(v_1,\la_1)}:\SS(\UU(v_1,v_2,\la_1,\la_2)(K),|\om|^{1/2})\to \SS(\UU(v_1,\la_1)(K),|\om|^{1/2}),$$
%so that for any open substack $\UU'\sub \UU(v_1,v_2,\la_1,\la_2)$,  
$$\SS(\UU_m(nD)(K),|\om|^{1/2})\rTo{T_{v_2}^{\la_2,\UU_m}} \SS(\UU_{m'}(nD)(K),|\om|^{1/2})\rTo{T_{v_1}^{\la_1,\UU_{m'}}}\SS(\UU_{m''}(nD)(K),|\om|^{1/2}).$$
The composed operator 
%$$T_{v_1}^{\la_1,\UU}\circ T_{v_2}^{\la_2,\UU(v_1,\la_1)}:\SS(\UU'(K),|\om|^{1/2})\to \SS(\UU(K),|\om|^{1/2})$$
is induced by the correspondence $Z(v_1,v_2)^{\UU_m}\sub Z(v_1,v_2)$, defined as the preimage of $\UU_m$ under the first projection.
%, and an isomorphism ??? (see ???).

For each $m$, we can choose $m'$ and $m''$ sufficiently large, so that the above assertion holds for the compositions of Hecke operators at $v_1$ and $v_2$
in both orders.
Then, %taking $\UU':=\UU(v_1,v_2,\la_1,\la_2)\cap \UU(v_2,v_1,\la_2,\la_1)$, 
using the equality of correspondences $Z(v_1,v_2)=Z(v_2,v_1)$ and compatibility of isomorphisms of line bundles as in the proof of Theorem 5.3,
we get the equality of operators
$$T_{v_1}^{\la_1,\UU_{m'}}\circ T_{v_2}^{\la_2,\UU_m}=T_{v_2}^{\la_2,\UU_{m'}}\circ T_{v_1}^{\la_1,\UU_m}:\SS(\UU_m(nD)(K),|\om|^{1/2})\to \SS(\UU_{m''}(nD)(K),|\om|^{1/2}).$$

\medskip

\noindent
{\bf Step 3}. 
Now we consider the situation in the formulation, so $v_1$ and $v_2$ are distinct $O$-points 
with the same reduction $v_0$ modulo $\fm$, and $T_1$ and $T_2$ are Hecke operators on 
$\SS(\Bun_G(nD)(O/\fm^N))$ (where $T_i$ is associated with $v_i$ and $\la_i$).
%Recall that we also have Hecke operators

Note that $T_2$ (resp., $T_1$) sends $\SS(\UU_m(nD)(O/\fm^N))$ to $\SS(\UU_{m'}(nD)(O/\fm^N))$ (resp., $\SS(\UU_{m'}(nD)(O/\fm^N))$ to $\SS(\UU_{m''}(nD)(O/\fm^N))$, and
it is easy to check that the following diagram is commutative:
\begin{diagram}
\SS(\UU_m(nD)(O/\fm^N))&\rTo{T_2}&\SS(\UU_{m'}(nD)(O/\fm^N))&\rTo{T_1}&\SS(\UU_{m''}(nD)(O/\fm^N))\\
\dTo{E_N}&&\dTo{E_N}&&\dTo{E_N}\\
\SS(\UU_m(nD)(K),|\om|^{1/2})&\rTo{T_{v_2}^{\la_2,\UU_m}}&\SS(\UU_{m'}(nD)(K),|\om|^{1/2})&\rTo{T_{v_1}^{\la_1,\UU_{m'}}}&\SS(\UU_{m''}(nD)(K),|\om|^{1/2})
\end{diagram}
where $E_N=E_{N,|\om|^{1/2}}$. We also have a similar commutative diagram with $T_1$ and $T_2$ swapped.
Using Step 2 and injectivity of $E_N$ on $\SS(\UU_{m''}(nD)(O/\fm^N))$ (by Step 1 applied to $m''$), we deduce that  for sufficiently large $n$, we have
$T_1T_2=T_2T_1$ on $\SS(\UU_m(nD)(O/\fm^N))$. 
\end{proof}

\medskip

\begin{proof}[Proof of Theorem 3.10]
Consider the elements $h_1,h_2$ in the local Hecke algebra $\HH_{G,O/\fm^N[\![t]\!]}$ associated with 
a pair of dominant coweights $\la_1,\la_2$ and a pair of ideals $I_1,I_2\sub O[\![t]\!]$ complementary to $O$.

We can choose a curve $C_0$ over $k$ with a point $v_0\in C_0(k)$ and a nice $G$-bundle $P_0$ over $C_0$ of sufficiently
large level with respect to $v_0$ (using Proposition \ref{nice-G-bun-prop}). We can lift the pair $(C_0,P_0)$ to a pair $(C,P)$, consisting of a curve and a $G$-bundle over $O/\fm^N$.
Finally, we can find points $v_1,v_2\in C(O)$ lifting $v_0$, corresponding to the ideals $I_1,I_2$. 

Then we have $P\in \UU_m$ for some $m$. Let us choose some positive relative Cartier divisor $D\sub C_O$, disjoint from $v_0$, and choose a trivialization of $P$ at $nD$
(where $n$ is sufficiently large). Then the corresponding delta-function $\de_P$ is in $\SS(\UU_m(nD)(O/\fm^N))$, so by Lemma \ref{main-lem}, 
\begin{equation}\label{comm-P-eq}
T_1T_2\de_P=T_2T_1\de_P.
\end{equation}

Since the operators $T_i$ correspond to the elements $h_i$ in the local Hecke algebra, and the compositions $T_1T_2$ and $T_2T_1$ come from the corresponding compositions
$h_1h_2$ and $h_2h_1$ in the local Hecke algebra, the equality \eqref{comm-P-eq} together with Lemma \ref{Gr-inj-lem} imply that $h_1h_2=h_2h_1$.
%Finally, let $P_0$ be a nice bundle on $C_k$ which is nice of level $\gg n$ at $v_0$, and let $V_{P_0}\sub \SS(\Bun_G(O/\fm^N))$
%denote the subspace of functions supported on $G$-bundles $P$ over $O/\fm^N$ which reduce to $P_0$. Then by Lemma 5.6, we have $V_{P_0}\sub \SS(\UU'(O/\fm^N))$.
%Hence, $T_1T_2=T_2T_1$ on $V_{P_0}$. 
\end{proof}

\appendix

\section{Constructions of $G$-bundles with generic behavior}
%over arbitrary finite fields}
\smallskip
\begin{center}by \textsc{Alexander Polishchuk and Ka Fai Wong}\end{center}
\medskip

\subsection{Constructing a sufficiently generic pointed curve with a line bundle of degree $0$}

Let $k$ be an arbitrary field.
As a starting point for the construction of nice $G$-bundles in the next section,
we need to construct a geometrically irreducible smooth projective curve $C$ over $k$, a point $p\in C(k)$, and a line bundle $L$ of degree zero on $C$, with sufficiently
generic cohomological behavior.

\begin{lemma}\label{CLp-lem}
    Let $n$ and $l$ be some positive integers. There exist a smooth geometrically irreducible projective curve $C$ of genus $g\ge n+2$ over $k$, a point $p\in C(k)$, and $L \in \Pic^0(C)$ such that $H^0(C, \mathcal{O}(np))=k$ and $H^0(C,L^i(np))=0$ for $1\le |i|\le l$. 
\end{lemma}

\begin{proof}
We claim that it is enough to construct a smooth, hypergeometric curve $C$ of genus $g\ge n+l+2$, with three points 
$p, p_1, p_2\in C(k)$, which are not ramification points of the double covering $\pi:C\to \P^1$, such that $\pi(p)$, $\pi(p_1)$ and $\pi(p_2)$ are distinct. Indeed,
let us set $L=\OO_C(p_1-p_2)$. Then for every $i$, $1\le i\le d$, 
we have $h^0(\OO(np+ip_1))=h^0(\OO(np+ip_2))=1$ (since every element of $H^0(\OO(n\pi^{-1}(\pi(p))+i\pi^{-1}(\pi(p_1))))$ is a pull-back from 
$\P^1$).
This immediately implies that $h^0(\OO(np))=1$.
Also, if there exists an effective divisor $D$ such that $L^i(np)\simeq \OO(D)$ (resp., $L^{-i}(np)\simeq \OO(D)$)
then we would get $\OO(np+ip_1)\simeq \OO(ip_2+D)$ (resp., $\OO(np+ip_2)\simeq \OO(ip_1+D)$) which is a contradiction. 
%and for any ramification point $q$, the divisor??? one has $h^0(\OO(np+q))=1$, in particular, $h^0(\OO(np))=1$.
%Thus, if we set $L=\OO(q_1-q_2)$, where $q_1$ and $q_2$ are ramification points, then $h^
Note that these conditions continue to hold after passing to the corresponding curve $C_{\ov{k}}$ over an algebraic closure $\ov{k}$ of $k$.
In particular, $H^0(C_{\ov{k}},\OO)=\ov{k}$ which implies that $C_{\ov{k}}$ is irreducible.

Assume first that the characteristic of $k$ is different from $2$. Then we can define $C$ as the smooth completion of the affine curve
$$y^2=f(x),$$
and take $p,p_1,p_2$ to be some points over $x=0,1,\infty$. 
More precisely, we want to choose $f(x)$ to be a monic separable polynomial $f$ of degree $\ge N:=2(n+l+2)+1$ such that $f(0)$ and $f(1)$ are squares in $k^*$.
Then the projection $(x,y)\mapsto x$ would extend to a double covering $C\to \P^1$, which splits over $0$, $1$ and $\infty$.

If the characteristic of $k$ is zero, we can simply take $f(x)=x^N+16/9$. If $k$ has positive characteristic, we can replace $k$ be the corresponding finite subfield.
To find a polynomial $f$ as above, let us pick three distinct irreducible monic polynomials $f_1$, $f_2$ and $f_3$ of degrees $\ge N$.
Then it is easy to see that one of the polynomials
$$f_1, \ f_2, \ f_3, \ f_1f_2, \ f_1f_3, \ f_2f_3, \ f_1f_2f_3$$
can be taken as $f$: one should consider the images of these elements under the map 
$$f\mapsto (f(0)\mod (k^*)^2, f(1)\mod (k^*)^2)\in k^*/(k^*)^2\times k^*/(k^*)^2\simeq \Z_2\times \Z_2$$
and observe that one of them will go to zero.

In the case when the characteristic of $k$ is $2$, we can define $C$ as the smooth completion of the affine curve 
$$y^2+f(x)y= x(x+1)f(x),$$ 
    where $f(x)$ is a monic separable polynomial for degree $\ge N=2(n+l+2)+1$ and $f(0)=f(1)=1$. It is easy to check that such an affine curve is smooth. 
We find a polynomial $f(x)$ with these properties defined over $\Z/2\sub k$ in the same way as before.
    %Let us show that it is geometrically irreducible, i.e.,  $y^2+f(x)y+x(x+1)f(x)$ is irreducible in $\bar{k}(x)[y]$. Suppose 
    %$$y^2+f(x)y+x(x+1)f(x)=(y+r_1(x))(y+r_2(x)).$$ 
    %Then we have $r_2(x)=f(x)+r_1(x)$ and thus $(x^2+x)f(x)=r_1^2(x)+r_1(x)f(x)$, which implies that $f(x)(x^2+x+r_1(x))=r_1^2(x)$. It follows that $f(x)$ divides $r_1(x)$ as $f(x)$ is separable. By degree consideration, $f(x)$ and $r_1(x)$ have the same degree, but this is impossible.
    %    Note that the double cover $\pi: C \to \P^1$,which on the affine chart $z\neq0$ is given by $(x,y)\mapsto x$, is unramified and split at $\pi^{-1}(1)$ and $\pi^{-1}(0)$, hence we can take $p_1,p_2 \in \pi^{-1}(1)$ and $p\in \pi^{-1}(0) $. Also, the covering map $\pi$ is ramified at $\pi^{-1}(\beta)$ for all roots $\beta$ of $f$. It follows by the Riemann-Hurwitz formula that $g \geq n+l+2$.
\end{proof}

\begin{lemma}\label{H1-nonvan-lem}
   In the situation of Lemma \ref{CLp-lem}, we have $H^1(C,L^i(np))\neq 0$, for any integer $i$.
\end{lemma}

\begin{proof}
By Serre's duality, we have $h^1(C,L^i(np))= h^0(C, \omega_C \otimes L^{-i}(-np))$, where $\omega_C$ denotes the canonical bundle. 
The degree of the line bundle $\omega_C \otimes L^i(-np)$ is $2g-2-n\geq g$ by our assumption on $g$.
Hence, by Riemann-Roch inequality, $h^0(C, \omega_C \otimes L^{-i}(-np))\ge 1$.
%we see that $$h^1(C,L^{i}(np))=h^0(C, \omega_C \otimes L^{-i}(-np)) \geq h^1(C, \omega_C \otimes L^{-i}(-np)) + 1 > 0.$$ 
\end{proof}

\subsection{Constructing nice $G-$bundles for a split reductive $G$}
Let $G$ be a connected split reductive group over a field $k$, satisfying the assumption $\bf{Char_G}$ (see Sec.\ \ref{nice-G-sec}).
 In this section, we construct a nice $G-$bundle of level $n$ (see Definition \ref{nice-G-bun-def}) on the pointed curve $(C,p)$ constructed in Lemma \ref{CLp-lem}.

\subsubsection{Group theoretic data}\label{group-th-sec}
We fix $T\sub B\sub G$ where $T \simeq \mathbb{G}_m^r$ is a split maximal torus, $B$ a Borel subgroup.
We denote by $N$ the unipotent radical of $B$, and by $Z$ the center of $G$. We denote the Lie algebras of $G,B,T,N$ and $Z$ by $\mathfrak{g}$, $\mathfrak{b}$, $\mathfrak{t}$, $\mathfrak{n}$ and $\mathfrak{z}$ respectively. We will also denote the opposite Borel (resp., its unipotent radical) by $B^-$ (resp. $N^-$).

Let $X^\bullet(T)$ denote the character lattice of $T$. We denote the set of roots (resp., positive roots, negative roots and simple roots, with respect to $T\sub B$) by $\Phi$ (resp., $\Phi^+$, $\Phi^-$ and $\Delta$). For $\chi \in X^\bullet(T)$ and $\mathcal{F}_T \in \text{Bun}_T(C)$, we denote the associated line bundle by $\mathcal{F}_T^\chi$.

For each root $\a$, we denote by $\fg_\a\sub \fg$ the corresponding root subspace.
We have a natural $T$-equivariant identification 
$$N/[N,N]\rTo{\sim} \mathfrak{n}/[\mathfrak{n},\mathfrak{n}]=\bigoplus_{\a \in\De}\fg_\a$$ 
(where the latter sum is a product of additive groups $\G_a$, with the action of $T$ given by the simple roots $\a$). If $\alpha \in \Delta$, we denote the projection to the $\alpha$-root subspace by $\pi_\alpha$.
% in both group and Lie algebra cases.

\subsubsection{Lifting $T$-bundles to $B$-bundles}\label{B-lifting-sec} 
The idea of our construction of a nice $G$-bundle is as follows. We start with a $T$-bundle, choose its lifting to a $B$-bundle 
(with sufficiently generic behavior) and then take the induced $G$-bundle. 
In this subsection we fix some notation about Cech representatives of various principal bundles that will repeatedly appear in the rest of this section and explain how twisted unipotent bundles encode the lifts of a $T$-bundle to $B$-bundles.

The conjugation action of $T$ on $N$ induces a map $H^1(C,T) \to H^1(C, \text{Aut}(N))$. Thus, with every class $\ga\in H^1(C,T)$ we can associate a sheaf of groups $N^\gamma$
over $C$, which is a form of $N$ (i.e., locally isomorphic to $N$).

Let $\mathcal{F}_T$ be a $T-$bundle on $C$ given by the Cech $1$-cocycle $(\gamma_{ij}\in T(U_{ij}))$ with respect to some open covering $(U_i)$.
To lift the $T$-bundle $\mathcal{F}_T$ to a $B$-bundle, it suffices to choose an $N^\gamma$-bundle. Indeed, using the decomposition $B=T\cdot N$, we can define a $B$-bundle by a Cech $1$-cocycle 
%$\lbrace \beta_{ij}:U_{ij} \to B \rbrace$ for some cover $\coprod U_i \to C$. Note that  we can write 
$\sigma_{ij}=\gamma_{ij}\beta_{ij}$, where $\beta_{ij}\in N(U_{ij})$. The cocycle condition  $\sigma_{ij}|_{U_{ijk}} \sigma_{jk}|_{U_{ijk}}= \sigma_{ik}|_{U_{ijk}}$ is equivalent to
% $$ \gamma_{ij}|_{U_{ijk}} \beta_{ij}|_{U_{ijk}} \gamma_{jk}|_{U_{ijk}} \beta_{jk}|_{U_{ijk}}= \gamma_{ik}|_{U_{ijk}} \beta_{ik}|_{U_{ijk}}. $$
%This yields
$$ (\gamma_{jk}^{-1}|_{U_{ijk}} \beta_{ij}|_{U_{ijk}} \gamma_{jk}|_{U_{ijk}}) \beta_{jk}|_{U_{ijk}}  = \beta_{ik}|_{U_{ijk}} . $$
In other words, $(\beta_{ij})$ is a $\ga$-twisted $1$-cocycle, so it corresponds to an element in $H^1(C,N^\gamma)$. 

\subsubsection{Construction of a nice $G$-bundle} 
We fix $n\ge 1$.
Let $C,p,L$ be the data as in Lemma \ref{CLp-lem} with sufficiently large $l$. Note that to give a $T$-bundle $\FF_T$ amounts to giving an $r$-tuple of
line bundles $(\FF_T^{\a_1},\ldots,\FF_T^{\a_r})$, where $(\a_1,\ldots,\a_r)$ are simple roots.
Assuming that $l$ is sufficiently large $l$, we choose these line bundles in the form $\FF_T^{\a_i}=L^{n_i}$ for some positive integers $(n_i$)
so that $H^0(C,\mathcal{F}_T^\chi(np))=0$ for all $\chi\in X^\bullet(T)$ such that $\chi$ is either a root or a nonzero sum of two roots. 

Let $\ga\in H^1(C,T)$ denote the class of this $T$-bundle $\FF_T$. 
Consider the composition
\begin{equation}\label{N-ga-comp}
H^1(C,N^\gamma) \to H^1(C, (N/[N,N])^\gamma) \simeq \bigoplus_{\a\in\De}H^1(C,\mathbb{G}_a^\gamma)=  \bigoplus_{\a\in\De}H^1(C,\mathcal{F}_T^\alpha) \to 
\bigoplus_{\a\in\De}H^1(C,\mathcal{F}_T^\alpha(np)).
\end{equation}
Note that all the maps are surjective (the first map is surjective because $H^2(C,[N,N]^\gamma)=0$, as $[N,N]^\gamma$ is an iterated extension of line bundles). 
% $\beta \in H^1(C,N^\gamma)$ such that its image in $H^1(C,\mathcal{F}_T^\alpha(np))$ is nontrivial for all $\alpha \in \Delta$. 
Let $\FF_B$ denote the $B$-bundle on $C$ associated with $\FF_T$ and a class $\b\in H^1(C,N^\gamma)$ such that the image of $\b$
in $H^1(C,\mathcal{F}_T^\alpha(np))$ is nontrivial for all $\alpha \in \Delta$. Such a class exists due to surjectivity of \eqref{N-ga-comp} and the nonvanishing of
the spaces $H^1(C,\mathcal{F}_T^\alpha(np))$ (see Lemma \ref{H1-nonvan-lem}).
Let $P$ be the $G$-bundle induced from $\FF_B$.

% We now construct a $G$-bundle $P$ on $C$ that is nice of level $n$ at $p$.
\begin{prop}\label{nice-G-bun-prop-bis}
    The $G$-bundle $P$ is nice of level $n$ at $p$. 
\end{prop}

We need the following auxiliary result (most likely, well known).

\begin{lemma}\label{alpha-diff-lem}
For $\alpha \in \Delta$, $t\in \mathfrak{t}$ and $n\in N/[N,N] \simeq \mathfrak{n}/[\mathfrak{n},\mathfrak{n}]$, we have
the equality in $\fg_\a$,
    $$\pi_\alpha (\Ad(n)(t)-t)= d\alpha(t)\pi_\alpha(n),$$
where $d\alpha: \mathfrak{t} \to k$ is the differential of the root $\alpha$. 
\end{lemma}
\begin{proof}
We will repeatedly use the fact (see \cite[Prop.\ 10.5]{Humphreys}) that if $t\in T'\sub T$ and $n\in N'\sub N$, for some closed subgroups $T'$ and $N'$,
then $\Ad(n)(t)-t$ belongs to the Lie subalgebra of the subgroup $[T',N']$. For example, applying this to $N'=[N,N]$ shows that the left-hand side indeed depends only
on $n\mod [N,N]$.
Also, since the adjoint action of $N$ on $\mathfrak{n}/[\mathfrak{n},\mathfrak{n}]$ is trivial,
the identity
$$\Ad(n_1n_2)(t)-t=\Ad(n_1)(\Ad(n_2)(t)-t)+\Ad(n_1)(t)-t$$
shows that both sides are additive in $n$.
% since if $t$ is in the center, both sides of the equality vanishes.

%Both sides depend only on $t \mod \fz$, so we can assume that $G$ is semisimple.
 % The equality for the rank one case follows by a straightforward calculation. The case $n\in \fg_\a$ easily reduces to the case of rank one. To prove the general case, 
%  we must show that if $n \in U_\alpha$ then $\Ad(n)t-t \in \mathfrak{n}/[\mathfrak{n},\mathfrak{n}]$ projects to $0$ in the $\alpha'$-root space $\mathfrak{g}_{\alpha'}$ for any simple root $\alpha'\neq \alpha$. For this it suffices 
%it remains to show that $\pi_\a(\Ad(n)t-t)=0$ for $t \in \ker(d\alpha)$ and $n\in \fg_{\a'}$, where $\a'\in \De$, $\a'\neq \a$. Since $n$ commutes with $\ker(\alpha)$, it follows by  that $\Ad(n)t-t=0$. 
Thus, we can start with $n\in \fg_{\a'}$, where $\a'\in \De$. Then $\Ad(n)(t)-t\in \ft+\fg_{\a'}$, so both sides are zero unless $\a'=\a$. In the case $n\in \fg_\a$, both sides
depend on $t \mod \ker(d\a)$, so we are reduced to the rank $1$ case, which follows by a straightforward calculation. 
\end{proof}

\begin
{proof}
[Proof of Proposition \ref{nice-G-bun-prop-bis}]
{\bf Step 1}. Checking that $H^0(C,\mathfrak{g}_P(np))=\mathfrak{z}$.

%We will lift the $T$-bundle $\mathcal{F}_T$ to a $B$-bundle $\mathcal{F}_B$ whose induced $G$-bundle $P$ satisfies that $H^0(C,\mathfrak{g}_P(np))=\mathfrak{z}$. 
Since $P$ is the induced $G$-bundle of the $B$-bundle $\mathcal{F}_B$, we have $\mathfrak{g}_P=\mathfrak{g}_{\mathcal{F}_B}$.
Consider the following exact sequence of $B$-modules:
$$0 \to \mathfrak{b} \to \mathfrak{g}\to \mathfrak{n^-} \to 0.$$
%Using the root subspace filtration of $\mathfrak{n}^-$, we see that 
The vector bundle $\mathfrak{n}^-_{\mathcal{F}_B}(np)$ admits a filtration whose successive quotients are the line bundles 
$\FF_T^{\a}$ associated with negative roots $\a$, and therefore, it has trivial $H^0$ by the construction of $\FF_T$.
Hence, it remains to prove that $H^0(C,\mathfrak{b}_{\mathcal{F}_B}(np))=\fz.$ 

Consider the following exact sequence of $B$-modules
$$0 \to \mathfrak{n} \to \mathfrak{b}\to \mathfrak{t} \to 0,$$
which induces the short exact sequence of adjoint vector bundles
$$ 0 \to \mathfrak{n}_{\mathcal{F}_B}  \to \mathfrak{b}_{\mathcal{F}_B} \to \mathfrak{t}_{\mathcal{F}_B} \to 0,$$
where $\ft_{\FF_B}=\ft\ot\OO_C$ is a trivial bundle.
%Note again that $\mathfrak{n}_{\mathcal{F}_B}(np)$ has vanishing global sections. Indeed, $\mathfrak{n}$ admits a filtration $\mathfrak{n}\supset [\mathfrak{n},\mathfrak{n}] \supset [\mathfrak{n},[\mathfrak{n},\mathfrak{n}] ] \supset ...$
%, and each consecutive quotient is a direct sum of some root spaces $\mathfrak{g}_\alpha$ for $\alpha \in \Phi^+$. Since $H^0(C,\mathfrak{n}_{\mathcal{F}_B}(np))=0$ for all $\alpha \in \Phi$, it follows that $H^0(C, \mathfrak{n}_{\mathcal{F}_B}(np))=0$. 
Note that $\fn_{\FF_B}$ has a filtration whose successive quotients are the line bundles 
$\FF_T^{\a}$ associated with positive roots $\a$, hence, $H^0(C,\fn_{\FF_B})=0$, by the construction of $\FF_T$.

%Since $  \mathfrak{t}_{\mathcal{F}_B}(np)=(\mathcal{O}_C(np))^r$, now it remains to choose a lift $\mathcal{F}_B$ carefully (or equivalently, an element $\beta \in H^1(C,N^\gamma)$) such that the 
Thus, $H^0(C,\fb_{\FF_B}(np))$ is equal to the kernel of the connecting homomorphism
$$\de:H^0(C,\ft\ot \OO_C(np)) \rightarrow H^1(C, \mathfrak{n}_{\mathcal{F}_B}(np)).$$ 
It remains to prove that the kernel of $\de$ is contained in $\mathfrak{z}$.

In view of the following map between short exact sequences of $B$-modules
\[
\xymatrix{
  0 \ar[r] & \mathfrak{n} \ar[d] \ar[r] & \mathfrak{b} \ar[d] \ar[r] & \mathfrak{t} \ar[d] \ar[r] & 0 \\
  0 \ar[r] & \mathfrak{n}/[\mathfrak{n},\mathfrak{n}] \ar[r] & \mathfrak{b}/[\mathfrak{n},\mathfrak{n}] \ar[r] & \mathfrak{t} \ar[r] & 0
}
\]
it is enough to prove that the kernel of the connecting homomorphism of the associated bundles for the bottom sequence, 
$$\delta'_n: H^0(C, \ft\ot\OO_C(np)) \to H^1(C,\bigoplus_{\alpha \in \Delta}  \mathcal{F}_{T}^\alpha(np))$$ 
is contained in $\mathfrak{z}$. 
%Pushing the associated bundles of the bottom row out along $(\mathfrak{n}/[\mathfrak{n},\mathfrak{n}])_{\mathcal{F}_B} \to \mathcal{F}_T^\alpha$, we get an extension of vector %bundles
%$$0 \to \mathcal{F}_T^\alpha \to E^\alpha \to \mathfrak{t}_{\mathcal{F}_B} \to 0. $$
%We further reduce the problem to studying the connecting homomorphism $ H^0(C,\mathfrak{t}_{\mathcal{F}_B}(np)) \to H^1(C, \mathcal{F}_T^\alpha(np))$.

Since $H^0(C, \ft\ot \OO_C(np))=\ft$ by the construction of $(C,p)$, it is enough to calculate the connecting homomorphism 
$$\delta':\ft=H^0(C,\mathfrak{t}_{\mathcal{F}_B}) \to \bigoplus_{\a\in \De} H^1(C, \mathcal{F}_T^\alpha).$$
Let $\beta_\alpha \in H^1(C,\mathcal{F}_T^\alpha)$ denote the image of the class $\beta\in H^1(N^\ga)$ (which was used to define $\FF_B$.
Then we claim that the $\a$-component of $\de'$
%the connecting homomorphism 
%$H^0(C,\mathfrak{t}_{{\mathcal{F}_B}}) \simeq \mathfrak{t} \to H^1(C, \mathcal{F}_T^\alpha)$ 
is given by  $$f \mapsto d\alpha(f) \cdot \beta_\alpha.$$

Indeed, we can compute this using Cech representatives. As in Sec.\ \ref{B-lifting-sec}, we consider a Cech $1$-cocycle $(\ga_{ij})$ representing $\FF_T$ and $\ga$-twisted
$1$-cocycle $(\b_{ij})$ representing $\b\in H^1(C,N^\ga)$ (with respect to an affine covering $(U_i)$ of $C$). We start with a global section 
$f \in H^0(C,\mathfrak{t}_{{\mathcal{F}_B}})  \simeq \mathfrak{t}$.
Let $f_i$ be liftings of $f|_{U_i}$ to 
% which restricts to $f_i$ on $U_i$ (all $f_i$'s are the same constant $f$ but we use $f_i$ to emphasize it as the restricted section). Lifting these local sections $f$ to local sections in
$(\mathfrak{b}/[\mathfrak{n},\mathfrak{n}])_{{\mathcal{F}_B}}(U_i) \simeq \mathfrak{b}/[\mathfrak{n},\mathfrak{n}] \otimes \mathcal{O}$.
%, which we also denote by $f_i$. On the fiber 
Then on $U_{ij}$ we have 
$$f_i \equiv \Ad(\sigma_{ij})f_j \mod [\fn,\fn]_{\FF_B},$$
and the $\a$-component of $\de'(f)$ is represented by the $1$-cocycle 
$$\pi_\a(\Ad(\sigma_{ij})f_j-f_j)=\pi_\a(\Ad(\gamma_{ij} \beta_{ij})f_j-f_j)=\pi_\a(\Ad(\beta_{ij})f_j-f_j).$$ 
%That means, we need to show that
%$$ \pi_\alpha (\Ad(\beta_{ij})f_j-f_j)= d\alpha(f)\pi_\alpha(\beta_{ij}),$$
%but this 
Hence, our claim follows from Lemma \ref{alpha-diff-lem}.

Since the image of each $\beta_\alpha$ in $H^1(C, \mathcal{F}_T^\alpha(np))$ is still nonzero (by our choice of $\b$), we deduce that
if $f\in \ft$ is in the kernel of $\de'_n$ then $f \in \bigcap_{\alpha \in \Delta}\text{ker}(d\alpha)=\mathfrak{z}$. This finishes the proof that 
%Therefore, with this choice of $\beta$ , we obtain the resulting $B$-bundle $\mathcal{F}_B$  whose induced $G$-bundle $P:=\mathcal{F}_B\times^B G$ has the desired property that 
$H^0(C,\mathfrak{g}_P(np))=\mathfrak{z}$.

\medskip

\noindent
{\bf Step 2}. 
    Let us set $U=C-\lbrace p \rbrace$. Now we will check that if $\phi$ is an automorphism of the $G$-bundle $P|_U$, such that the 
    induced morphism $\phi_{\mathfrak{g}}:\mathfrak{g}_P|_U \to \mathfrak{g}_P|_U$, extends to a map $\wt{\phi}_{\mathfrak{g}}:\mathfrak{g}_P \to \mathfrak{g}_P(np)$, then $\phi \in Z(k)$. 

   First, we will show that $\phi$ is actually induced by a $B$-bundle automoprhism of $\FF_B$. Note that 
    $$\Hom(\mathfrak{b}_{\FF_B}, \mathfrak{n}^-_{\FF_B}(np))=H^0(C,\mathfrak{b}_{\FF_B}^\vee \otimes \mathfrak{n}^-_{\FF_B}(np))=0,$$ 
    \begin{equation}\label{n-b-vanishing}
    \Hom(\mathfrak{n}^-_{\FF_B}, \mathfrak{b}_{\FF_B} (np))=H^0(C,\mathfrak{b}_{\mathcal{F}_B} \otimes (\mathfrak{n}^-_{\mathcal{F}_B})^\vee(np)) =0.
    \end{equation}
     Indeed, this follows from the existence of a filtration of $\mathfrak{b}_{\FF_B}^\vee \otimes \mathfrak{n}^-_{\FF_B}$ 
     (resp., $ \mathfrak{b}_{\FF_B}\ot (\mathfrak{n}^-_{\FF_B})^\vee$)
     with the subquotients of the form
    $\mathcal{F}_T^{\chi}$, with $\chi$ either in $\Phi^-$ (resp., $\Phi$) or of the form $\chi=\a+\a'$ where $\alpha , \alpha' \in \Phi^-$ (resp., $\Phi$), since for such $\chi$ we have
%    by considering the root filtrations of $\mathfrak{b}_{\mathcal{F}_B}$ and $\mathfrak{n}^-_{\mathcal{F}_B}$, it suffices to see that
    $H^0(C,\mathcal{F}_T^{\chi}(np))=0$ by our construction. 
    So, the exact sequence 
    $$0\to\Hom(\fb_{\FF_B},\fb_{\FF_B}(np))\to \Hom(\fb_{\FF_B},\fg_{\FF_B}(np))\to \Hom(\fb_{\FF_B},\fn^0_{\FF_B}(np))=0$$
    shows that $\wt{\phi}_{\fg}$ maps $\fb_{\FF_B}$ to $\fg_{\FF_B}$, hence $\phi_{\fg}$ preserves the subbundle $\fb_{\FF_B}$.

 Let $\und{\Aut}_B(\FF_B)$ (resp. $\und{\Aut}_G(P)$) be the sheaf of automorphisms of $\FF_B$ as a $B$-bundle (resp. of $P$ as $G$-bundle). 
Note that $\und{\Aut}_B(\FF_B)$ is a subsheaf in $\und{\Aut}_G(P)$. We need to show that $\phi$ lies in
 $\Aut_B(\FF_B)(U) \sub \Aut_G(P)(U)$. The problem is local and thus we can assume $\FF_B$ to be trivial. Then $\phi$ is a function $U \to G$, and
 %$\mathfrak{g}_P=\mathfrak{g} \otimes \mathcal{O}$ and $\mathfrak{b}_P=\mathfrak{b} \otimes \mathcal{O}$, 
 the induced map $\phi_\mathfrak{g}$ is given by $\text{Ad}(\phi(x))$ acting on $\mathfrak{g}\ot \OO$. Since $\text{Ad}(\phi(x))$ preserves $\mathfrak{b}$, it follows that $\phi(x) \in B$ (because for $g\in G$, $\text{Ad}(g)$ preserves $\mathfrak{b}$ if and only if $g \in B$). This proves that $\phi$ comes from an automorphism of $\FF_B|_U$, which we still denote by $\phi$.

    %Now we may and do assume that $\phi$ is an automorphism of $B$-bundles, and from now on just denote the $B$-bundle by $P$. 
    Let $\phi_T$ be the induced automorphism of $\FF_T$ as a $T$-bundle. Then $\phi_T$ is given by an element in $T(U)$. Since any invertible function on $U$ is constant and thus  $T(U)=T(k)$, $\phi_T$ is actually given by some $t \in T(k)$. 
   Hence, the automorphism $\phi$ of $\FF_B$ is given by a collection $\phi_i=tn_i\in B(U_i)$, where $n_i\in N(U_i)$, satisfying 
   $$tn_i=\si_{ij}tn_i\si_{ij}^{-1},$$
   where $\si_{ij}=\ga_{ij}\b_{ij}$ is the Cech $1$-cocycle defining $\FF_B$ (see Sec.\ \ref{B-lifting-sec}). Using commutativity of $T$ we can rewrite this as
   $$(\ga_{ij}^{-1}n_i\ga_{ij})\b_{ij}=(t^{-1}\b_{ij}t)n_j,$$
   which implies that the $\ga$-twisted $1$-cocycles $(\b_{ij})$ and $(t^{-1}\b_{ij}t)$ have the same class in $H^1(C,N^\ga)$.
   Projecting this equality to $H^1(C,N/[N,N])$ and then to $H^1(C,\FF^\a_T)$, we deduce that 
   $$\a(t)\cdot \b_\a=\b_\a$$
   for each $\a\in\Delta$. Since all $\b_\a$ are nonzero, we obtain
    %Thus, locally $\phi =t \phi_N$, where $\phi_N$ comes from an automorphism of the $N^\gamma$-bundle $Q$ given by $\beta$. Indeed, one sees that such a decomposition of $\phi$ exists by considering the following exact sequence of sheaves on the curve
   % $$ 1 \to \Aut_{N^\gamma} Q \to \Aut_B P \to \Aut_T \mathcal{F}_T \to 1,$$
   % whose exactness can be seen locally. Now we have 
   % $$ t\phi_{N,j} \gamma_{ij} \beta_{ij} \phi_{N,i}^{-1}t^{-1}=\gamma_{ij} \beta_{ij}$$
   % Since $\phi_{N,j} \gamma_{ij} \beta_{ij} \phi_{N,i}^{-1}= \gamma_{ij} \beta_{ij}$, it follows that $t\beta t^{-1}= \beta$. But modulo $[N,N]$ this equality implies that $\alpha(t) \beta_\alpha =\beta_\alpha$ for all $\alpha \in \Delta$ and thus we have 
   $t \in \bigcap_{ \alpha \in \Delta} \text{ker}(\alpha) =Z$. 
    
    Thus, rescaling $\phi$ by a central element we may assume that $\phi_T=t=1$. We claim that in fact $\phi_{\fg}=1$ or equivalently $\wt{\phi}_{\fg}=1$.
    Note that we have a commutative diagram of maps of vector bundles on $C$,
            \[\begin{tikzcd}
0\arrow{r} 
    & \mathfrak{b}_{\FF_B} \arrow{d}{\wt{\phi}_{\fb}}\arrow{r}
        & \mathfrak{g}_{\FF_B} \arrow{d}{\wt{\phi}_{\mathfrak{g}}} \arrow{r}
            & \mathfrak{n}^-_{\FF_B} \arrow{d}{\wt{\phi}_{\fn^-}} \arrow{r}
                & 0 \\
0 \arrow{r}
    & \mathfrak{b}_{\FF_B}(np)  \arrow{r}
        & \mathfrak{g}_{\FF_B}(np) \arrow{r}
            & \mathfrak{n}^-_{\FF_B}(np) \arrow{r}
                & 0
\end{tikzcd}. \] 
    %from $\wt{\phi}_{\fg}$ we an extension of $\phi_{\fb}$ (resp., $\phi_{\fn^-}$) to the morphism of vector bundles
    %$\wt{\phi}_{\fb}:\fb_{\FF_B}\to \fb_{\FF_B}(np)$ (resp., $\wt{\phi}_{\fn^-}:\fn^-_{\FF_B}\to$ on $C$ 
Taking into account the vanishing \eqref{n-b-vanishing} and applying Lemma \ref{easy-Hom-lem} below, we see that to prove our claim it is enough to check
that $\wt{\phi}_{\fb}=1$ and $\wt{\phi}_{\fn^-}=1$. 

Since $\phi_T: \mathcal{F}_T \to \mathcal{F}_T$ is the identity, the induced maps $\mathcal{F}_T^\alpha \to \mathcal{F}_T^\alpha$ are also identity maps for all $\alpha \in \Phi$. 
Since $\Hom(\mathcal{F}_T^\alpha, \mathcal{F}_T^{\alpha'}(np))=0$ for any pair of distinct roots $\alpha,\alpha' \in \Phi$, we can apply Lemma \ref{easy-Hom-lem}
below successively to the filtration 
$$\fb_{\FF_B}\supset \mathfrak{n}_{\FF_B}\supset [\mathfrak{n},\mathfrak{n}]_{\FF_B} \supset [\mathfrak{n},[\mathfrak{n},\mathfrak{n}] ]_{\FF_B} \supset\ldots$$
(resp., a similar filtration of $\fn^-_{\FF_B}$) and deduce 
%that restricting $\tilde{\phi}_\mathfrak{b}:\mathfrak{b}_P \to \mathfrak{b}_P(np) $ to $\tilde{\phi}_\mathfrak{n}:\mathfrak{n}_P \to \mathfrak{n}_P(np)$ is also the identity map (this restriction makes sense again because $H^0(C,\mathfrak{n}_P ^\vee \otimes \mathfrak{t}_P(np))=0$). Hence we have the following diagram
that $\wt{\phi}_{\mathfrak{b}}=1$ (resp., $\wt{\phi}_{\fn^-}=1$), as required. 

Finally, we claim that if an automorphism $\phi$ of a $B$-bundle $\FF_B$ satisfies $\phi_T=1$ and $\phi_{\fg}=1$ then $\phi=1$. 
The problem is local, so we can assume that $\FF_B$ is trivial and $\phi$ is given by a map $U \xrightarrow{\phi} B$. 
Then
 $\phi_\mathfrak{g}$ is given by the composition 
 %$:U \xrightarrow{\phi}B \xrightarrow{\Ad_B} \text{Aut}(\mathfrak{b})$, and  is given by the composition 
 $$U \xrightarrow{\phi}B \hookrightarrow G \xrightarrow{\Ad_G} \text{Aut}(\mathfrak{g}).$$ 
 Since $\phi_{\fg}=1$, we deduce
 %Since $\phi_T=1$, we have that $\phi$ takes values in $N \subset B$. Since $\phi_\mathfrak{g}=1$, we have 
 that $\phi$ takes values in $\text{ker}(\Ad_G)=Z$ (see \cite[Prop.\ 3.3.8]{Conrad}).
On the other hand, since  $\phi_T=1$, $\phi$ takes values in $N$. But $N \cap Z=1$, so $\phi=1$.
%Finally, it remains to see that this intersection is trivial. Indeed, consider $T$ acting on $N$ via conjugation and the filtration $N \supset [N,N] \supset [N,[N,N]] \supset ...$ which is $T$-stable, each successive quotient has no nontrivial fixed point and hence $N$ has no nontrivial fixed point. 
\end{proof}

\begin{lemma}\label{easy-Hom-lem}
Let $D$ be an effective Cartier divisor on a scheme $S$. 
%and $id_{\EE'}(D)$ (resp. $id_{\EE''}(D)$) denote the canonical map $\EE' \to \EE'(D)$ (resp. $\EE' \to \EE''(D)$) induced by the canonical section $\mathcal{O} \to \mathcal{O}(D)$. If $\Hom(\EE'',\EE'(D))=0$, and 
Suppose we have the following commutative diagram of coherent sheaves on $S$, in which both rows are exact, 
\[\begin{tikzcd}
0\arrow{r} 
    & \EE' \arrow{d}{1}\arrow{r}
        & \EE \arrow{d}{\phi} \arrow{r}
            & \EE'' \arrow{d}{1} \arrow{r}
                & 0 \\
0 \arrow{r}
    & \EE'(D)  \arrow{r}
        & \EE(D) \arrow{r}
            & \EE''(D) \arrow{r}
                & 0
\end{tikzcd} \]
Assume in addition that $\Hom(\EE'',\EE'(D))=0$. Then
then $\phi=1$.
% is the canonical map induced by the canonical section $\mathcal{O}\to \mathcal{O}(D)$, and consequently $\phi$ is the identity map away from the support of $D$.
\end{lemma}

The proof is straightforward (by considering $\phi-1$).
%\begin{proof}
%   We prove (1). One reduces to show that $\phi=0$ if the other two vertical maps are zero. One then applies $\Hom(-,\EE)$ to the short exact sequence to get the exact sequence
%   $$0 \to \Hom(\EE'',\EE)\to \Hom(\EE,\EE) \to \Hom(\EE',\EE). $$
%   Since $\phi$ is mapped to $0 \in \Hom(\EE',\EE)$, it comes from $\phi' \in \Hom(\EE'',\EE)$. Since $\Hom(\EE'',\EE')=0$, one has an injection $\Hom(\EE'',\EE) \hookrightarrow%\Hom(\EE'',\EE'')$. But since $\phi'$ goes to $0 \in \Hom (\EE'', \EE'')$, $\phi'$ is 0 and so is $\phi$.
%   The proof for (2) is essentially the same.
 %  \end{proof}

\subsection{$B$-structures of very negative degrees on the trivial bundle over $\P^1$}
Let $G$ be a connected split reductive group over a field $k$. 
We keep the notation of Sec.\ \ref{group-th-sec}. In addition, we denote by $X_\bullet(T)$ the coweight lattice of $T$.

Given a $B$-bundle $F$ on a curve $C$, for every $\alpha \in X^\bullet(T)$ we define the degree $\deg_\a(F)\in \Z$ as degree of the line bundle associated with $F$ and
the homomorphism $B\to T\to \G_m$ given by $\alpha$. 
%In this section, we will prove that the trivial $G$-bundle on $\mathbb{P}^1$ has $B$-structures with arbitrarily negative degrees for all simple roots. 

\begin{prop}\label{B-str-prop}
For any positive integer $m$, there exists a $B$-bundle $F$ on $\mathbb{P}^1$ such that the induced $G$-bundle is trivial and %$\text{deg}(P^\alpha) < -m$
$\deg_\a(F)<-m$ for all $\alpha \in \Delta$.
\end{prop}

\begin{proof}
   We use the standard open cover $(U_i)_{i=0,1}$ of $\P^1$, where  $U_0$ (resp. $U_1$) is the affine line $\Spec k[t]$  (resp. $\Spec k[t^{-1}]$), and $U_{01}=\Spec k[t,t^{-1}]$ the intersection. Let us consider the $B$-bundle $F$ on $\P^1$, trivial over $U_0$ and $U_1$, with a transition function $\ga\cdot\b\in B(k[t,t^{-1}])$,
%   The line bundle $\mathcal{O}(l)$ is given the map $t^l: U_{01} \to k^*$. Now let $\gamma$ denote the diagonal matrix $\text{diag}(t^{l_1},...,t^{l_n})$ with $l_i-l_{i+1}<-m $, then the problem amounts to finding $\beta \in N(k[t,t^{-1}])$ such that $\gamma \beta$ = $\phi_0 \phi_1$ for some $\phi_0 \in G(k[t])$ and $\phi_1 \in G(k[t^{-1}])$. The solution to this problem follows from the next lemma.
where $\ga\in T(k[t,t^{-1}])$, $\b\in N(k[t,t^{-1}])$. We will choose $\ga$ to be a coweight, i.e., a homomorphism $\G_m\to T$.
Then the requirements on $\ga$ and $\b$ are that $\lan \ga,\a\ran<-m$ for every $\a\in \De$, and that there exist $\phi_0 \in G(k[t])$ and $\phi_1 \in G(k[t^{-1}])$, such that
$$\ga\b=\phi_0\phi_1$$
(i.e., the induced $G$-bundle is trivial).
Existence of such data follows from a more precise Lemma \ref{B-str-lem} below: in the notations of this lemma we take $\b=n_+^{-1}$, $\phi_0=n_-$, $\phi_1=g$.
\end{proof}

\begin{lemma}\label{B-str-lem}
    Let $G$ be a connected split reductive group. For any $m>0$, there exists $\gamma \in X_\bullet(T) \subset T(k[t^{\pm 1}])$ such that $\langle \gamma, \a \rangle <-m$ for every $\a \in \Delta$ and admits a decomposition $\gamma= n_-gn_+$ for some $n_-\in N_-(k[t])$, $g\in G(k[t^{-1}])$ and $n_+ \in N(k[t])$.
\end{lemma}

\begin{proof}
     It suffices to prove the assertion for $[G,G]$ instead of $G$, so we can assume $G$ to be semisimple. Furthermore, by considering (commuting) simple subgroups of $G$ corresponding to the simple factors of $\fg$, we reduce to the case of a simple $G$.
%     The same reasoning here also allows us to reduce the semisimple case to the simple case.

     We will use induction on the rank with the following induction step (we will prove the base case later).
          Let $I\subset \Delta$ denote the set of simple roots (identified with the nodes in the Dynkin diagram) to which the affine root attaches on the extended Dynkin diagram 
          (note that $|I|=1$ unless $G$ is of type $A$).
          Let $P$  be the standard parabolic subgroup of $G$ that corresponds to the diagram with the subset $I$ removed. Also, let $L$ denote the derived subgroup of the Levi subgroup corresponding to $P$, and let $T_L$ denote the maximal torus of $L$ contained in $T$. Assuming that the assertion holds for smaller rank, we have $\gamma_L \in X_\bullet(T_L)$ such that $\langle \gamma_L, \alpha \rangle <-m $ for all $\alpha \in \Delta(L)=\De\setminus I$, and 
          % and when regarded as an element in $X_\bullet(T)$ via the embedding $T_L \to T$ satisfies all the requirements except possibly that $\langle \gamma_L, \alpha \rangle$ is not smaller than $-m$. By induction hypothesis, we have a decomposition 
          $$\gamma_L=n_-gn_+$$
     for some $n_- \in (N_-)_L(k[t])$, $n_+ \in N_L(k[t])$ and $g \in L(k[t^{-1}])$, where $N_{L}$ (resp. $(N_-)_L$) denote the corresponding unipotent subgroup (resp. opposite unipotent subgroup).
     
  %   Depending on the type of the Dynkin diagram, we will choose $I$ accordingly and 
    % modify $\gamma_L$ as follows. 
     Let $\lambda$ be the highest positive root of $G$, $\la^\vee$ the corresponding coroot.
     % and $I$ be the set of vertices adjacent to $\lambda$ in the affine Dynkin diagram. 
     Then we have
     \[ \langle \lambda^\vee, \alpha \rangle = \begin{cases} 
          a, & \alpha \in I, \\
          0, & \alpha \in \Delta-I,
       \end{cases}
    \]
where $a$ is equal to $1$ or $2$. 
%and that $\lambda$ is a root of maximal length. Then clearly t
Hence, there exists a positive integer $c$ such that 
$$\langle \gamma_L-c\lambda^\vee, \alpha \rangle <-m ,$$
for all $\alpha \in \Delta$.

Let us denote $G_\lambda$ be the subgroup of $G$ generated by the root groups $U_\lambda$ and $U_{-\lambda}$. 
%Since $\lambda^\vee \in G_\lambda(k[t^\pm])$, 
It follows from the rank one case that there exists a decomposition in $G_\la(k[t,t^{-1}]$,
$$-c\lambda^\vee=n'_-g'n'_+$$
for some $n'_- \in U_{-\la}(k[t])$, $n'_+\in U_\la(k[t])$
%(N_-)_{G_\lambda}(k[t])$, $n'_+ \in N_{G_\lambda}(k[t])$ 
and $g' \in G_\lambda(k[t^{-1}])$. 
%, where $N_{G_\lambda}$ (resp. $(N_-)_{G_\lambda}$) denote the corresponding unipotent subgroup (resp. opposite unipotent subgroup).

Finally, we claim that the coweight $\gamma_L-c\lambda^\vee\in X_\bullet(T)$ has the desired decomposition. 
Indeed, since $\lambda$ is orthogonal to $\alpha$ for all $\alpha \in \Delta-I$, it follows from the the Chevalley's commutator formula (see \cite[Prop.\ 5.1.14]{Conrad}) 
that $G_\lambda$ commutes with $L$. Hence, 
$$\gamma_L-c\lambda^\vee = (n_-n'_-)(gg')(n_+n_+')$$
is the desired decomposition.
 
%Now, we will apply the above induction argument to simple groups of all types. 

Since in our induction step we have $|I|=1$ unless $G$ is of type $A$, it remains to check that the assertion holds for $G=\SL_2$ and $G=\SL_3$.     
%     \begin{case}
%         For Type A, the base cases are the rank one and rank two cases. They follow from the equalities 
These cases follow from the identities
         \begin{gather}\nonumber
 \begin{bmatrix} t^{-m} & 0 \\ 0 & t^m \end{bmatrix}
 =   \begin{bmatrix}  1 &   0 \\   -t^m+1 &   1   \end{bmatrix}
 \begin{bmatrix} t^{-m} & -t^{-m}-1\\ -t^{-m}+1 & t^{-m} \end{bmatrix}
\begin{bmatrix} 1 & t^m+1 \\ 0 & 1\end{bmatrix},
\end{gather}
%and 
\begin{gather}\nonumber
\begin{bmatrix} t^{-m} &0& 0\\ 0&1&0 \\ 0 &0& t^m \end{bmatrix}=
  \begin{bmatrix}1 &   0 &0 \\    0&1&0 \\   -t^m+1 & 0&   1   \end{bmatrix}  \begin{bmatrix} t^{-m} &0 & -t^{-m}-1\\ 0&1&0 \\ -t^{-m}+1 &0 & t^{-m} \end{bmatrix}
\begin{bmatrix} 1 & 0 & t^m+1\\ 0 & 1& 0\\ 0 & 0& 1\end{bmatrix}.
\end{gather}        
\end{proof}


\begin{thebibliography}{9}
%\bibitem{BPGN} L.~Brambila-Paz, I.~Grzegorczyk, P.~E.~Newstead,
%{\it Geography of Brill-Noether loci for small slopes}, J. Algebraic Geom. 6 (1997), no. 4, 645--669.
\bibitem{AA} A.~Aizenbud, N.~Avni,
{\it Representation growth and rational singularities of the moduli space of local systems},
Invent. Math. 204 (2016), no. 1, 245--316.
\bibitem{BD} A.~Beilinson, V.~Drinfeld, {\it Quantization of Hitchin's integrable system and Hecke eigensheaves}, Preprint available at
www.math.uchicago.edu/$\sim$drinfeld/langlands/QuantizationHitchin.pdf
\bibitem{DS} V.~Drinfeld, C.~Simpson,
{\it $B$-structures on $G$-bundles and local triviality}, Math. Res.
Lett. 2, 823--829 (1995).
\bibitem{BK-Hecke} A.~Braverman, D.~Kazhdan,
 {\it Some examples of Hecke algebras for two-dimensional local fields}, Nagoya Math. J. 184 (2006), 57--84.
\bibitem{BK} A.~Braverman, D.~Kazhdan,
{\it Automorphic functions on moduli spaces of bundles on curves over local fields: a survey}, arXiv: 2112.08139. 
\bibitem{BKP} A.~Braverman, D.~Kazhdan, A.~Polishchuk, {\it Automorphic functions for nilpotent extensions of
curves over finite fields}, arXiv:2303.16259.
%\bibitem{Bur8} N.~Burbaki, {\it Lie groups and Lie algebras}, ch. VIII.
%\bibitem{Carter} R.~W.~Carter,
%{\it Finite groups of Lie type. Conjugacy classes and complex characters}, John Wiley and Sons, Inc., New York, 1985.
%\bibitem{CM} Collingwood, McGovern
\bibitem{Conrad} B.~Conrad, {\it Reductive Group Schemes}, Panor. Synth\`eses 42/43, Soc. Math. France, Paris, 2014.
%\bibitem{Drinfeld-aut} V.~G.~Drinfeld, {\it Two-Dimensional $l$-Adic Representations of the Fundamental Group of a Curve over a Finite Field
%and Automorphic Forms on $\GL(2)$}, AJM 105 (1983), 85--114.
\bibitem{EFK1} P.~Etingof, E.~Frenkel, D.~Kazhdan, {\it An analytic version of the Langlands correspondence for complex curves},  in
{\it Integrability, quantization, and geometry II. Quantum theories and algebraic geometry}, 137--202, AMS, Providence, RI, 2021.
\bibitem{EFK2} P.~Etingof, E.~Frenkel, D.~Kazhdan, {\it Hecke operators and analytic Langlands correspondence for curves over local fields}, 
Duke Math. J. 172 (2023), no. 11, 2015--2071.
%\bibitem{JL} Jacquet, Langlands
\bibitem{Faltings} G.~Faltings, {\it Algebraic loop groups and moduli spaces of bundles},
J. Eur. Math. Soc. 5 (2003), no. 1, 41--68.
\bibitem{FGA} B.~Fantechi, L.~G\"ottsche, L.~Illusie, S.~L.~Kleiman, N.~Nitsure, A.~Vistoli, {\it Fundamental Algebraic Geometry. Grothendieck's FGA Explained}, AMS, 
Providence, RI, 2005.
\bibitem{GK1} D.~Gaitsgory, D.~Kazhdan, {\it Representations of algebraic groups over a 2-dimensional local field}, Geom. Funct. Anal. 14 (2004), no. 3, 535--574. 
\bibitem{GK2} D.~Gaitsgory, D.~Kazhdan, 
{\it Algebraic groups over a 2-dimensional local field: some further constructions}, in {\it Studies in Lie theory}, 97--130, Birkh\"auser, Boston, MA, 2006.
\bibitem{GPS} P.~Gille, R.~Parimala, V.~Suresh, {\it Local triviality for $G$-torsors}, Math. Ann. 380 (2021), no. 1-2, 539--567.
\bibitem{Harder} G.~Harder, {\it \"Uber die Galoiskohomologie halbeinfacher algebraischer Gruppen. III}, J. Reine Angew. Math. 274-275 (1975), 125--138.
\bibitem{Humphreys} J.~E.~Humphreys, {\it Linear algebraic groups}, Springer, New York-Heidelberg, 1975.
\bibitem{K-YD} D.~Kazhdan, A.~Yom Din, {\it On irreps of a Hecke algebra of a non-reductive group}, arXiv:2209.05536.
%\bibitem{Ko} M. Kontsevich, {\it Notes on motives in finite characteristic}, in {\it Algebra, Arithmetic, and Geometry, in
%honor of Yu.I. Manin}, Vol. II, eds. Yuri Tschinkel and Yuri Zarhin, pp. 213--247, Prog. in Math. 270, Birkh\"auser, 2010.
%\bibitem{Laf-L} L.~Lafforgue, {\it Chtoucas de Drinfeld et correspondance de Langlands}, Invent. Math. 147 (2002), no. 1, 1--241.
%\bibitem{Laf-V} V.~Lafforgue, {\it Shtukas for semi-simple groups and Langlands correspondence for function field}, J. Amer. Math. Soc. 31 (2018), no. 3, 719--891.
%\bibitem{LW} S.~Lang, A.~Weil, {\it Number of points of varieties in finite fields}, Amer. J. Math. 76 (1954), 819--827.
%\bibitem{Milne} Milne
\bibitem{Newstead} P.~E.~Newstead, {\it Introduction to moduli problems and orbit spaces}, Tata Institute of Fundamental Research, Bombay; Narosa Publishing House, New Delhi, 1978.
\bibitem{Sorger} C.~Sorger, {\it Lectures on moduli of principal $G$-bundles over algebraic curves}, in {\it School on Algebraic Geometry (Trieste, 1999)}, 1--57,
ICTP, Trieste, 2000.
%\bibitem{Springer} Springer, {\it Some arithmetical results on semi-simple Lie algebras}
\bibitem{Weil} A.~Weil, Adeles and Algebraic Groups, Birkh\"auser, Boston, MA, 1982.
\end{thebibliography}
\end{document}